\documentclass[11pt,a4paper]{article}
\usepackage[T1]{fontenc}
\usepackage[utf8]{inputenc}
\usepackage{amsmath,amsthm,mathtools}
\usepackage{newtxtext,newtxmath}
\usepackage[margin=28mm,headheight=15pt]{geometry}
\usepackage{microtype}
\usepackage{enumitem,needspace}
\usepackage[numbers,sort&compress]{natbib}
\usepackage[colorlinks=true,linkcolor=black,citecolor=black,urlcolor=black]{hyperref}
\usepackage{fancyhdr}
\usepackage{tikz}
\usepackage{placeins}
\numberwithin{equation}{section}
\newtheorem{theorem}{Theorem}[section]
\newtheorem{proposition}[theorem]{Proposition}
\newtheorem{lemma}[theorem]{Lemma}
\newtheorem{corollary}[theorem]{Corollary}
\theoremstyle{remark}
\newtheorem{remark}{Remark}

\newcommand{\R}{\mathbb R}
\newcommand{\Sph}{\mathbb S^{N-1}}
\newcommand{\nS}{\nabla_{\!S}}
\newcommand{\HS}{\nabla_{\!S}^{\,2}}
\newcommand{\DS}{\Delta_S}
\newcommand{\gS}{g_S}
\newcommand{\UU}{U_*}
\newcommand{\RR}{R_0}
\newcommand{\PP}{\mathcal P}
\newcommand{\MA}{\mathcal M}
\newcommand{\id}{\mathrm{Id}}
\DeclareMathOperator{\supp}{supp}

\title{\bfseries Whole-space eventual log-concavity and level-set geometry\\
	for the multidimensional Fisher--KPP equation}
\author{
	Xing Liang\thanks{School of Mathematical Sciences,
		University of Science and Technology of China,
		Hefei, Anhui 230026, China.
		The author is supported by NSFC12531008, NSFC12331006 and XDB0900100. 
		Email: \texttt{xliang@ustc.edu.cn}.}
	\and
	Ziqi Zhen\thanks{School of Mathematical Sciences,
		University of Science and Technology of China,
		Hefei, Anhui 230026, China.
		Email: \texttt{wy39108797936@mail.ustc.edu.cn}.}
}
\date{}
\hypersetup{pdftitle={Whole-space eventual log-concavity and level-set geometry for the multidimensional Fisher-KPP equation},
	pdfsubject={Whole-space logarithmic Hessian bounds, level-set convexity, front geometry, and Gaussian exterior asymptotics},
	pdfauthor={Xing Liang and Ziqi Zhen},pdfkeywords={Fisher-KPP; eventual log-concavity; level-set convexity; front asymptotics; logarithmic Hessian}}
\begin{document}
	\maketitle
	\vspace{-1.4em}
	\begin{abstract}
		We establish eventual strict log-concavity and quantitative level-set curvature bounds for the Fisher–KPP equation \(u_t=\Delta u+u-u^2\) in \(\mathbb R^N\), \(N\geq2\), with nontrivial, nonnegative, smooth compactly supported initial data. Specifically, there exist \(T > 1\) and \(c>0\) such that \(D_x^2\log u(t,x)\leq-(c/t)e^{-2(\sqrt2-1)t}\mathrm{Id}\) for all \(t\geq T\) and \(x\in\mathbb R^N\).
		Consequently, after a single waiting time, the solution has exactly one critical point, which is a nondegenerate global maximizer, and every positive level set below the maximum is smooth, closed, and strictly convex with positive principal curvatures. Moreover, for levels in any fixed closed interval \(K\subset(0,1)\), the principal curvatures are uniformly bounded below by \((1-C\log t/t)/(2t)\) for all sufficiently large \(t\).
		The proof combines one-sided second angular derivative estimates for the fixed-front region, absorptive comparison in the invaded interior, leading-edge barriers and blow-up arguments through the diffusive scale, and uniform Gaussian amplitude estimates in the exterior.
	\end{abstract}
	
	
	\section{Introduction}\label{sec:intro}
	
	\subsection{Large-time asymptotics and geometric questions}
	
	Consider the Fisher--KPP equation
	\begin{equation}\label{eq:kpp}
		u_t=\Delta u+u-u^2,\qquad t>1,\quad x\in\R^N,
		\qquad u(1,x)=u_0(x).
	\end{equation}
	Unless explicitly stated otherwise, \(N\ge2\), and the initial datum satisfies
	\begin{equation}\label{eq:data}
		u_0\in C_c^\infty(\R^N),\qquad
		u_0\ge0,\qquad u_0\not\equiv0.
	\end{equation}
	Choose \(R_*>0\) such that \(\supp u_0\subset B_{R_*}(0)\).
	We take \(t=1\) as the initial time to simplify the logarithmic changes of variables used below.
	
	The equation describes invasion of the unstable state \(0\) by the
	stable state \(1\). Fisher \citep{Fisher1937} and Kolmogorov, Petrovskii
	and Piskunov \citep{KPP1937} introduced its travelling fronts, and
	Aronson and Weinberger \citep{AW1978} established spreading at speed
	\(2\) for compactly supported data. Bramson \citep{Bramson1983}, by
	probabilistic methods, identified the logarithmic delay of the critical
	front in one space dimension. In several dimensions, G{\"a}rtner
	\citep{Gartner1982}, also by probabilistic methods, determined the front
	location up to an \(O(1)\) term, including the dimension-dependent
	logarithmic correction. A PDE approach to the logarithmic correction
	was developed by Hamel, Nolen, Roquejoffre and Ryzhik
	\citep{HNRR2013}; Ducrot \citep{Ducrot2015} later adapted this PDE
	argument to the multidimensional setting. Nolen, Roquejoffre and
	Ryzhik \citep{NRR2017} subsequently refined the one-dimensional PDE
	analysis to prove convergence to a single travelling wave.
	The same authors later rigorously established the universal $-3\sqrt{\pi}\,t^{-1/2}$ correction to the one-dimensional front position through matched asymptotic analysis \citep{NRR2019}.

	Our analysis builds substantially on the sharp large-time asymptotics
	for the multidimensional Fisher--KPP equation obtained by Roquejoffre,
	Rossi and Roussier-Michon \citep{RRM2019}.
	Let \(\UU\) be the decreasing critical travelling wave, with a fixed
	normalization:
	\begin{equation}\label{eq:wave}
		\UU''+2\UU'+\UU-\UU^2=0,\qquad
		\UU'<0,\qquad \UU(-\infty)=1,\quad \UU(+\infty)=0.
	\end{equation}
	Theorem~1.1 of \citet{RRM2019} is stated under
	their normalized assumption~(3), while the remark
	on p.~7268 extends its conclusion to \eqref{eq:data}.
	Choose \(x_0\in\mathbb R^N\) such that \(u_0\)
	is bounded below by a positive constant on a ball
	centered at \(x_0\), and set
	\(\widetilde u(t,y)=u(t,y+x_0)\).
	Let \(\widetilde s^\infty\in\operatorname{Lip}(\Sph)\) be its directional phase described in \cite[Theorem~1.1]{RRM2019}.
	For \(x=\rho\omega\), we have
	\[
	|x-x_0|+
	\widetilde s^\infty\left(
	\frac{x-x_0}{|x-x_0|}\right)
	=\rho+\widetilde s^\infty(\omega)
	-x_0\cdot\omega+O(\rho^{-1}),
	\]
	uniformly in \(\omega\) as \(\rho\to\infty\).
	Thus the phase relative to the fixed origin is
	\(s^\infty(\omega)=
	\widetilde s^\infty(\omega)-x_0\cdot\omega
	\in\operatorname{Lip}(\Sph)\).
	Set
	\begin{equation}\label{eq:R0}
		k=\frac{N+2}{2},\qquad \RR(t)=2t-k\log t.
	\end{equation}
	The uniform front asymptotics of \cite[Theorem~1.1]{RRM2019} then take the form
	\begin{equation}\label{eq:rrm}
		u(t,\rho\omega)=
		\UU\bigl(\rho-\RR(t)+s^\infty(\omega)\bigr)+o(1)
		\qquad(t\to\infty),
	\end{equation}
	uniformly for \(\rho>0\) and \(\omega\in\Sph\).
	A finite change in the time origin only changes the limiting phase.
	
	The core of the proof of Theorem~1.1 in \citet{RRM2019} is inspired by
	ideas from the one-dimensional PDE analysis of Nolen, Roquejoffre and
	Ryzhik \citep{NRR2017}. In the present work, we will also use several further ingredients
	from \citep{RRM2019}: self-similar asymptotics in the diffusive zone,
	radial barriers, and estimates for first angular derivatives.
	The use of the diffusive-zone asymptotics under \eqref{eq:data} is explained at the beginning of
	Section~\ref{cont:sec:through-diffusive}.
	
	Several works have studied the geometry of expanding level sets.
	Jones \citep{Jones1983} proved, using a reflection argument, that for sufficiently large times the normal line at every regular point of a level set intersects the convex hull of the initial support.
	Rossi \citep{Rossi2017} used this property to prove eventual
	star-shapedness and a uniform \(O(1)\) bound on the radial oscillation
	of level sets. More recently, Hamel and Rossi \citep{HamelRossi2025} studied asymptotic local one-dimensional symmetry for classes of possibly unbounded initial supports.

	Along a different line of research, estimates on principal curvatures provide a finer description of level-set geometry.
	Chang, Ma and Yang \citep[Theorem~1.1]{ChangMaYang2010}
	established a boundary minimum principle for a
	gradient-weighted least principal curvature
	of strictly convex level sets for certain semilinear
	elliptic equations.
	For the heat equation on convex rings,
	Chen, Ma and Salani \citep{ChenMaSalani2019}
	proved positivity of spatial principal curvatures
	under space-time quasiconcavity, $u_t>0$,
	and suitable boundary conditions.

	Brascamp and Lieb \citep{BrascampLieb1976}
	showed that the heat equation on convex domains
	preserves log-concavity of the initial data.
	On bounded strongly convex domains,
	Greco and Kawohl \citep{GrecoKawohl1999}
	obtained related results for semilinear parabolic
	equations, and Gallo, Moraschi and Squassina
	\citep{GalloMoraschiSquassina2026} for
	variable-coefficient logistic equations, under
	suitable structural and boundary assumptions.
	These results concern the preservation of an initial
	concavity structure, whereas we impose no concavity
	assumption on the initial data.

	Eventual concavity without initial concavity assumptions has also been studied by Lee and Vázquez. For the whole-space porous medium equation, they established eventual strict concavity of the pressure through higher-order convergence to the self-similar Barenblatt profile \citep{LeeVazquez2003}. They subsequently used concavity properties of asymptotic profiles and higher-order convergence to establish eventual concavity results for suitable transforms of solutions to the heat and porous medium equations  on bounded strictly convex domains \citep{LeeVazquez2008}. 
	For the whole-space heat equation, Lee and Vázquez \citep[Theorem~5.1]{LeeVazquez2003} also proved eventual log-concavity for nonnegative compactly supported data. Ishige \citep{Ishige2024} further studied more general eventual concavity properties of the heat flow.

	Returning to the Fisher–KPP equation, the known geometric properties of its level sets, including eventual star-shapedness and a uniform \(O(1)\) bound on their radial oscillation, do not imply convexity, even when combined with the sharper front asymptotics \eqref{eq:rrm}. We therefore ask whether positive superlevel
	sets eventually become strictly convex for initial data satisfying \eqref{eq:data},
	and seek quantitative lower bounds for their
	principal curvatures on fixed compact level
	intervals in $(0,1)$.
	
	In the homogeneous Fisher--KPP model, the superlevel set
	$\{x\in\R^N:u(t,x)>\lambda\}$ represents the region
	where population density exceeds a threshold $\lambda$.
	Eventual convexity rules out inward indentations
	of this region, rather than merely making them
	negligible relative to its expansion.
	We further ask whether this holds simultaneously
	for all nonempty positive superlevel sets after
	a single waiting time, independent of the threshold.
	A stronger question is whether $u(t,\cdot)$
	eventually becomes strictly log-concave throughout
	$\R^N$.
	Unlike the curvature bounds on fixed compact level
	intervals, this whole-space property also concerns
	the low-density region far ahead of the front.

	The front asymptotics \eqref{eq:rrm} do not by themselves provide the second-order information needed for either quantitative curvature bounds or whole-space log-concavity. In particular, the directional phase \(s^\infty\) is only known to be Lipschitz, so we cannot obtain the second-order angular estimates needed for curvature bounds simply by differentiating the limiting profile. Moreover, uniform convergence does not yield the relative second-order control required for the logarithmic Hessian where \(u\) is small. We obtain these estimates from the parabolic equation.

	\subsection{Main results}
	
	Write
	\[
	M(t)=\max_{x\in\R^N}u(t,x).
	\]
	At each \(t>1\), spatial decay and the strong maximum principle
	give \(0<M(t)<\infty\).
	Moreover, \(M(t)<1\) for all
	sufficiently large \(t\), as shown at the beginning of Section~2.
	Hessian inequalities below are inequalities of
	quadratic forms with respect to the Euclidean metric unless the
	spherical metric is specified.
	Principal curvatures are taken with the outward normal and the sign
	convention that a sphere has positive curvatures.
	
	Set
	\begin{equation}\label{eq:beta-star}
		\beta_*=2(\sqrt2-1)\in(0,1).
	\end{equation}
	The principal result is eventual strict log-concavity on the whole
	space.
	\begin{theorem}[Whole-space eventual strict log-concavity]
		\label{cont:thm:whole-space}
		Under \eqref{eq:data}, there are \(T>1\) and \(c>0\) such that
		\begin{equation}\label{cont:eq:whole-space-hessian}
			D_x^2\log u(t,x)
			\le-\frac{c}{t}e^{-\beta_*t}\id
			\qquad(t\ge T,\ x\in\R^N).
		\end{equation}
		In particular, \(u(t,\cdot)\) is strictly log-concave on \(\R^N\)
		for every \(t\ge T\).
	\end{theorem}
	
	\begin{corollary}[Simultaneous strict convexity of all level sets below the maximum]
		\label{cont:cor:all-levels}
		Under the hypotheses of Theorem~\ref{cont:thm:whole-space},
		let \(T\) be as in that theorem. 
		Then, for every \(t\ge T\), \(u(t,\cdot)\) has exactly one critical point, which is its unique nondegenerate global maximizer.
		Every value \(0<\lambda<M(t)\)
		is regular, and the corresponding level set is smooth, closed, and
		strictly convex, with every principal curvature positive.
	\end{corollary}

	For levels in any fixed closed interval \(K\subset(0,1)\),
	we obtain sharper quantitative information: 
	angular derivative bounds and a curvature lower bound on the
	\(t^{-1}\) scale.
	We first give a description of level sets based on the front
	asymptotics \eqref{eq:rrm} of
	\citet[Theorem~1.1]{RRM2019} and the radial derivative convergence in Lemma~\ref{lem:radialC1}.
	For \(0<\lambda<1\), write
	\[
	\Gamma_\lambda(t)=\{x\in\R^N:u(t,x)=\lambda\},
	\qquad z_\lambda=\UU^{-1}(\lambda).
	\]
	For every closed interval \(K\subset(0,1)\) and all sufficiently large \(t\),
	the level sets \(\Gamma_\lambda(t)\), \(\lambda\in K\), are smooth closed
	radial graphs
	\begin{equation}\label{eq:radialmain}
		\Gamma_\lambda(t)=\{R_\lambda(t,\omega)\omega:\omega\in\Sph\},
	\end{equation}
	and their radii satisfy
	\begin{equation}\label{eq:locationmain}
		R_\lambda(t,\omega)
		=2t-k\log t+z_\lambda-s^\infty(\omega)+o(1),
	\end{equation}
	uniformly in \(\lambda\in K\) and \(\omega\in\Sph\).
	We will prove \eqref{eq:radialmain} and \eqref{eq:locationmain} in Section~\ref{sec:geometry}.
	
	Using this radial-graph representation, we then establish angular derivative estimates and a quantitative lower bound for the principal curvatures in the following theorem.
	The symbols \(\nS\) and \(\HS\) denote the gradient and covariant
	Hessian for the standard metric \(\gS\) on \(\Sph\).
	A tensor inequality involving \(\HS R\) is always an inequality
	between quadratic forms.

	\Needspace{9\baselineskip}
	
	\begin{theorem}[Quantitative curvature on compact level intervals]\label{thm:main}
		Assume \eqref{eq:data}. Let \(K\subset(0,1)\) be a closed interval.
		Then there exist constants \(T>1\) and \(C>0\), depending on \(N,u_0,K\),
		such that the radial graphs \eqref{eq:radialmain} satisfy
		\begin{equation}\label{eq:derivmain}
			|\nS R_\lambda(t,\omega)|\le C,\qquad
			\HS R_\lambda(t,\omega)\le C\log t\,\gS
		\end{equation}
		for every \(t\ge T\), \(\lambda\in K\), and \(\omega\in\Sph\).
		Moreover, each \(\Gamma_\lambda(t)\) bounds a strictly convex body,
		and its principal curvatures satisfy
		\begin{equation}\label{eq:curvmain}
			\kappa_j(t,\omega)\ge
			\frac{1-C\log t/t}{2t}>0,\qquad 1\le j\le N-1.
		\end{equation}
	\end{theorem}

	\subsection{Proof strategy and organization}
	\label{cont:sec:strategy}
	
	The proof covers space by four overlapping large-time regimes
	(see Figure~\ref{fig:four-regimes}). We
	begin with the fixed-front region, where
	\(r:=\rho-R_0(t)\) remains bounded and the critical travelling wave
	gives the leading profile. Farther behind the front, \(u\) is uniformly close to the stable state \(1\), so the equations for its Cartesian second derivatives have a positive absorption coefficient. Ahead of the fixed-front region, we use the self-similar leading-edge asymptotics to control the solution through the subdiffusive range and up to the diffusive scale \(r=O(\sqrt t)\). Beyond the diffusive range, we obtain a Gaussian approximation with a nonlinear amplitude. Together, these estimates yield the desired whole-space logarithmic-Hessian bound.

	In the fixed-front region, the sharp asymptotics of Roquejoffre,
	Rossi and Roussier-Michon determine the radial profile. We derive the
	corresponding first and second radial derivative convergence from
	these asymptotics and uniform spatial derivative bounds. The remaining difficulty is
	second-order angular control. 
	As noted by
	\citet[Remark~1]{RRM2019}, the Kato argument for first
	angular derivatives does not directly yield absolute
	second-derivative bounds, because the quadratic
	first-derivative term loses its favorable sign when
	taking absolute values. They suggested the possibility
	of a one-sided upper estimate, but its use remained
	unclear.

	Our approach is to establish such a one-sided estimate and use it to control the curvature of the level sets.
	For a radial graph \(\rho=R(t,\omega)\), the second fundamental form
	contains
	\[
	R^2\gS+2\,dR\otimes dR-R\HS R,
	\]
	so an upper bound for \(\HS R\) is precisely the one-sided information
	needed for a curvature lower bound. {More precisely, we first establish a one-sided estimate for second rotation derivatives, and then derive the fixed-front bound for \(\nabla_S^2u\) by choosing great-circle rotations.} Here we adapt their first-angular-derivative strategy \citep[Proposition~6]{RRM2019}, using a scalar supersolution based on their Proposition~1.
	Together with the radial nondegeneracy and the first angular, mixed and radial derivative estimates, this yields the compact-level curvature bound and the full logarithmic Hessian in the fixed-front region, without requiring second angular derivatives of the limiting phase \(s^\infty\).

	In the invaded interior, for a fixed unit vector \(e\), the Cartesian
	second derivative \(h_e=u_{ee}\) satisfies
	\[
	(h_e)_t-\Delta h_e+(2u-1)h_e=-2(u_e)^2\le0.
	\]
	Here \(u\) is uniformly close to the stable state \(1\), so the
	zeroth-order term is absorptive. Using the negative ordinary Hessian on a moving sphere behind the
	front as lateral boundary data, an explicit comparison controls any
	initially positive second derivatives and eventually yields Hessian
	negativity throughout the enclosed ball, without assuming initial
	concavity.
	
	To pass from the fixed front to the diffusive scale, we prove two-sided bounds \(v\asymp r\) for the conjugated solution \(v=e^r u\), where \(r=\rho-\RR(t)\), from a fixed distance ahead of the front up to any fixed multiple of \(\sqrt t\). These bounds are suggested by the self-similar asymptotics of \citet[Theorem~2.1]{RRM2019},
	which give, for some positive Lipschitz function \(\mathfrak c\),
	\[
	t^{-1/2}v(t,\sqrt t\,\eta,\omega)
	\longrightarrow
	\mathfrak c(\omega)\eta e^{-\eta^2/4}
	\qquad (\eta>0).
	\]
	We establish the required uniform bounds throughout this range by constructing suitable upper and lower barriers.
	The nondegeneracy \(v\asymp r\), together with the rotation estimates, converts the absolute derivative bounds into relative angular and mixed estimates for \(\log u\).
	The radial logarithmic curvature is obtained separately. A
	subdiffusive blow-up reduces to a caloric rigidity argument, while
	at the diffusive scale, the above self-similar asymptotics identify the blow-up limit,
	and local parabolic regularity upgrades the rescaled convergence to
	spatial \(C^2_{\rm loc}\). Combining the
	radial, tangential and mixed estimates gives the logarithmic-Hessian
	bound throughout this region.
	
	Farther into the exterior region, write
	\(\sigma=t-1\) and \(p=x/(2\sigma)\), so that \(|p|=1\)
	corresponds to the critical speed. We use an exponential-moment
	formulation inspired by Brunet's one-dimensional analysis of the
	solution beyond the propagating front \citep{Brunet2023}.
	For elapsed time \(s\), let \(\mathcal U_p(s,y)\) denote the
	exponentially conjugated solution in the frame moving with velocity
	\(2p\). It satisfies the exact factorization
	\[
	\mathcal U_p(s,y)
	=
	\mathcal K_s(y)\,
	\mathfrak a_s\!\left(p+\frac{y}{2s}\right),
	\]
	where \(\mathcal K_s\) is the heat kernel and
	\(\mathfrak a_s\) is a finite-time amplitude carrying the nonlinear
	history. For each fixed \(p\) with \(|p|>1\), the finite-time amplitude \(\mathfrak a_s(p)\) converges to \(\mathfrak a(p)\) as \(s\to\infty\).

	The main issue in this region is to control this amplitude uniformly both as
	\(|p|\downarrow1\) and as \(|p|\) becomes large. We therefore prove
	estimates with explicit dependence on \(|p|-1\) in the near-critical
	regime, and estimates uniform for unbounded \(|p|\) once the gap from
	the critical sphere is fixed. Relative estimates comparing \(\mathfrak a_s\) with \(\mathfrak a\), together with the stability of \(\mathfrak a\) under the \(O(s^{-1/2})\) changes in \(p\) generated by diffusive displacements, then give a uniform \(C^0\) Gaussian approximation. Local parabolic
	regularity upgrades this approximation to spatial \(C^2\) control and
	yields the exterior logarithmic-Hessian bound. The four overlapping regions then yield the whole-space
	logarithmic-Hessian estimate of
	Theorem~\ref{cont:thm:whole-space}.

	Sections~\ref{sec:reflection} and~\ref{sec:angular} establish
	preliminary derivative estimates and a one-sided second
	angular estimate, respectively.
	Section~\ref{sec:front-geometry-concavity} uses them to obtain the
	compact-level geometry and fixed-front logarithmic Hessian, and then
	propagates Hessian negativity into the invaded interior.
	Section~\ref{cont:sec:through-diffusive} carries the
	logarithmic-Hessian estimate through the diffusive scale. Sections~\ref{cont:sec:near-critical}--\ref{cont:sec:fixed-gap}
	establish the near-critical and fixed-gap amplitude estimates.
	Section~\ref{cont:sec:exterior} converts these estimates into Gaussian
	exterior control and completes the whole-space proof.
	Standard parabolic estimates used throughout are as in
	\citep{LSU1968,Lieberman1996}.

	\begin{figure}[!htbp]
		\centering
		\begin{tikzpicture}[x=1cm,y=1cm,>=stealth,
			every node/.style={inner sep=1.5pt}, line cap=round]
			
			%------------------------------------------------
			% Front profile
			%------------------------------------------------
			\draw[->,line width=0.45pt] (0.55,0)--(14,0);
			\draw[->,line width=0.45pt] (0.55,0)--(0.55,2.18)
			node[above] {$u$};
			\draw (0.48,1.85)--(0.62,1.85);
			\node[left] at (0.48,1.85) {$1$};
			
			\draw[line width=0.95pt,smooth,
			domain=0.55:13.82,samples=160]
			plot (\x,{1.85/(1+exp(1.25*(\x-6.35)))});
			
			\draw[densely dashed,black!50,line width=0.35pt]
			(6.35,0)--(6.35,1.84);
			\draw (6.35,0.07)--(6.35,-0.07);
			\node[below] at (6.35,-0.06) {$0$};
			\node[below right] at (11.75,-0.03)
			{$r=\rho-R_0(t)$};

			\fill[black!8] (5.07,-1.95) rectangle (5.77,-1.05);
			\fill[black!8] (7.27,-1.95) rectangle (7.92,-1.05);
			\fill[black!8] (10.70,-1.95) rectangle (11.33,-1.05);
			
			% Invaded interior
			\draw[<-,line width=0.65pt]
			(0.66,-1.05)--(5.77,-1.05);
			\draw[line width=0.55pt]
			(5.77,-1.14)--(5.77,-0.96);
			\node[above,align=center,font=\scriptsize]
			at (2.65,-0.97)
			{Invaded interior\\[-1pt]$r\le -L_0$};
			
			% Fixed front
			\draw[line width=0.65pt]
			(5.07,-1.95)--(7.92,-1.95);
			\foreach \x in {5.07,7.92}
			\draw[line width=0.55pt]
			(\x,-2.04)--(\x,-1.86);
			\node[above,align=center,font=\scriptsize]
			at (6.49,-1.86)
			{Fixed front\\[-1pt]$|r|\le L_{\mathrm{fr}}$};
			
			% Diffusive leading edge
			\draw[line width=0.65pt]
			(7.27,-1.05)--(11.33,-1.05);
			\foreach \x in {7.27,11.33}
			\draw[line width=0.55pt]
			(\x,-1.14)--(\x,-0.96);
			\node[above,align=center,font=\scriptsize]
			at (9.32,-0.97)
			{Diffusive leading edge\\[-1pt]
				$L_K\le r\le K\sqrt t$};
			
			% Gaussian exterior
			\draw[->,line width=0.65pt]
			(10.70,-1.95)--(13.85,-1.95);
			\draw[line width=0.55pt]
			(10.70,-2.04)--(10.70,-1.86);
			\node[above,align=center,font=\scriptsize]
			at (12.28,-1.86)
			{Gaussian exterior\\[-1pt]
				$r\ge r_{\mathrm{ext}}(t)$};
			
		\end{tikzpicture}
		
		\caption{Schematic of the four overlapping spatial regions
			in the proof of Theorem~\ref{cont:thm:whole-space}.
			This one-dimensional sketch represents the solution profile along any fixed direction \(\omega\in\Sph\); \mbox{no radial symmetry is assumed.}
			Here $r_{\mathrm{ext}}(t)$ denotes the Gaussian-exterior cutoff,
			with $r_{\mathrm{ext}}(t)\gtrsim\sqrt t$
			for sufficiently large $t$.
			The other cutoff parameters are chosen in the proof.}
		\label{fig:four-regimes}
	\end{figure}
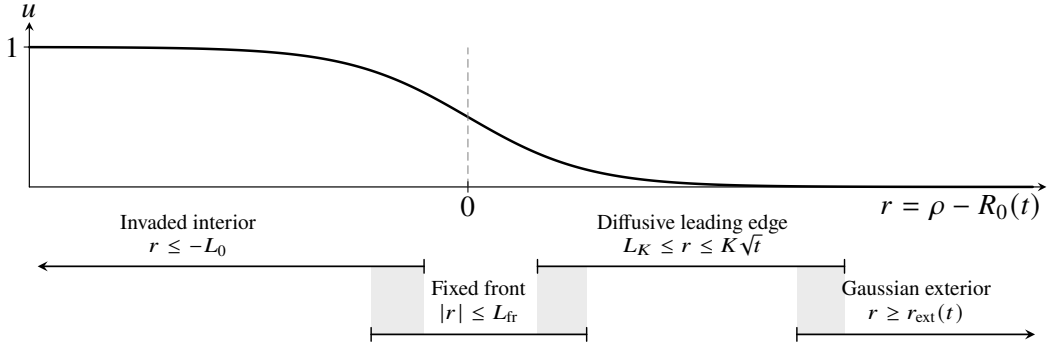
	\FloatBarrier

	\Needspace{12\baselineskip}
	\section{Reflection comparison and the front region}\label{sec:reflection}

	In this section, we collect several preliminary estimates for later use. We first show that \(0<u(t,x)<1\) for all sufficiently large \(t\) and all \(x\in\mathbb R^N\), and record uniform bounds on its spatial derivatives.

	Set \(K_0=\max\{1,\|u_0\|_\infty\}\). Since \(u_0\le K_0\) and
	\((\partial_t-\Delta)K_0-K_0+K_0^2=K_0(K_0-1)\ge0,\)
	comparison and the strong maximum principle give \(0<u(t,x)\le K_0\) for \(t>1\) and \(x\in\mathbb R^N\).
	For \(w=u-1\), the inequality \(w_t-\Delta w+w=-w^2\le0\)
	yields, with \(s=t-1\),
	\[
	u(t,x)-1\le e^{-s}\bigl[(G_s*u_0)(x)-1\bigr],
	\]
	where \(G_s\) is the heat kernel. Since
	\(\|G_s*u_0\|_\infty
	\le(4\pi s)^{-N/2}\|u_0\|_1\longrightarrow0,\)
	there is \(T_*>1\) such that \(0<u(t,x)<1\) for all
	\(t\ge T_*\) and \(x\in\R^N\).
	Standard interior estimates, applied after a fixed positive time,
	give
	\begin{equation}\label{eq:globalderivs}
		\sup_{t\ge2,\ x\in\R^N}
		\bigl(|\nabla_xu(t,x)|+|D_x^2u(t,x)|\bigr)<\infty.
	\end{equation}
	The same statement holds for any fixed finite number of spatial
	derivatives, with the corresponding constant.

	We next use a reflection argument to establish exterior monotonicity and an angular gradient estimate, following \citep{Jones1983,Rossi2017}.

	\begin{lemma}\label{lem:reflection}
		Let \(t>1\), \(x\in\R^N\), and \(|x|>R_*\). Then for every unit vector \(e\)
		such that \(x\cdot e>R_*\),
		\begin{equation}\label{eq:dirnegative}
			\nabla_xu(t,x)\cdot e<0.
		\end{equation}
		In particular \(u_\rho(t,\rho,\omega)<0\) when \(\rho>R_*\). Moreover,
		\begin{equation}\label{eq:angularcone}
			|\nS u(t,\rho,\omega)|
			\le \frac{R_*}{\sqrt{1-R_*^2/\rho^2}}\,
			\bigl(-u_\rho(t,\rho,\omega)\bigr),
			\qquad \rho>R_*.
		\end{equation}
	\end{lemma}
	
	\begin{proof}
		Fix \(a>R_*\) and a unit vector \(e\). Let
		\(H_{a,e}=\{y:y\cdot e>a\}\) and \(y^*=y-2(y\cdot e-a)e\).
		In this half-space set \(w(t,y)=u(t,y^*)-u(t,y)\).
		It satisfies
		\[
		w_t=\Delta w+\bigl(1-u(t,y^*)-u(t,y)\bigr)w,\qquad
		w=0\quad\hbox{on }\partial H_{a,e}.
		\]
		The initial value is nonnegative, since \(u_0=0\)
		in \(H_{a,e}\), and is not identically zero because
		the reflection of the positive part of \(u_0\)
		lies inside this half-space.
		Bounded comparison, the strong maximum principle,
		and the boundary point lemma yield \(w>0\)
		in \(H_{a,e}\) and \(\partial_e w>0\) on its boundary.
		There \(\partial_e w=-2\partial_e u\), proving
		\eqref{eq:dirnegative} at every point of the plane.
		Given \(x,e\) as in the statement, take \(a=x\cdot e\).
		
		For \(\rho=|x|>R_*\), put \(d=R_*/\rho\) and choose any unit vector
		\(v\perp\omega\). For \(d<\alpha<1\), the two vectors
		\(e_\pm=\alpha\omega\pm\sqrt{1-\alpha^2}\,v\)
		satisfy the condition in \eqref{eq:dirnegative}. Therefore
		\[
		|\nabla_xu\cdot v|
		\le\frac{\alpha}{\sqrt{1-\alpha^2}}(-u_\rho).
		\]
		Let \(\alpha\downarrow d\), take the supremum over \(v\), and use
		\(|\nS u|=\rho|(\id-\omega\otimes\omega)\nabla_xu|\).
	\end{proof}
	
	In particular, \eqref{eq:dirnegative} and the compact
	support of $u_0$ imply that, for every $t\ge1$,
	$u(t,\cdot)$ attains its maximum and all its
	maximizers lie in $\overline B_{R_*}(0)$.

	The reflection estimate controls first angular derivatives. To estimate their radial derivatives, we introduce the infinitesimal generators of rotations.
	In the rotation estimates, \(A\) denotes a real skew-symmetric matrix,
	and all suprema over \(A\) below are understood to be taken over such
	matrices.
	Define
	\begin{equation}\label{eq:rotation}
		X_A f(x)=(Ax)\cdot\nabla_x f(x).
	\end{equation}
	Since \(e^{sA}\in SO(N)\), the vector field \(x\mapsto Ax\) is the
	infinitesimal generator of a one-parameter group of Euclidean rotations,
	and hence a Euclidean Killing field. Its restriction to \(\mathbb S^{N-1}\)
	is a Killing field for the standard round metric.
	On the sphere, \(X_A\) differentiates \(f\) along the rotation orbits
	\(s\mapsto e^{sA}\omega\).
	In polar coordinates,
	\[
	X_A=(A\omega)\cdot\nabla_S,
	\]
	so \(X_A\) differentiates only in the angular variables. 
	Here \(\DS\) denotes the Laplace--Beltrami operator on
	\((\Sph,\gS)\).
	Moreover,
	\begin{equation}\label{eq:commute}
		[X_A,\Delta_x]=[X_A,\partial_\rho]=[X_A,\Delta_S]=0.
	\end{equation}
	
	Here and below, \(\|A\|_{\mathrm{op}}
	:=\sup_{|v|=1}|Av|\) denotes the Euclidean operator norm.

	\begin{lemma}\label{lem:firstmixed}
		There is a constant \(C\) such that, for \(t\ge3\),
		\begin{equation}\label{eq:rotationfirst}
			\sup_{\|A\|_{\mathrm{op}}\le1}
			\bigl(|X_Au(t,x)|+|\nabla_xX_Au(t,x)|\bigr)\le C
			\quad\text{for all }x\in\R^N.
		\end{equation}
		Consequently, for all \(\rho>0\),
		\begin{equation}\label{eq:firstmixed}
			|\nS u|+|\partial_\rho\nS u|+|u_{\rho\rho}|\le C.
		\end{equation}
	\end{lemma}
	
	\begin{proof}
		When \(\rho\ge2R_*\), \eqref{eq:angularcone} and
		\eqref{eq:globalderivs} imply \(|\nS u|\le C\).
		Thus \(|X_Au|\le C\) there, uniformly for \(\|A\|_{\mathrm{op}}\le1\).
		In the remaining ball, \(|X_Au|\le2R_*|\nabla_xu|\le C\).
		This gives the first bound in \eqref{eq:rotationfirst} for \(t\ge2\).
		
		For fixed \(A\), \(z=X_Au\) solves the linear parabolic equation
		\begin{equation*}
			z_t=\Delta_xz+(1-2u)z.
		\end{equation*}
		Interior gradient estimates on fixed space-time cylinders give
		\(|\nabla_xz|\le C\) for \(t\ge3\). The coefficients and the bound on \(z\)
		are uniform in \(A\), so the constant is uniform as well.
		Equivalently, this follows from Duhamel's formula on a time interval of
		length one and the integrable estimate
		\(\|\nabla G_s\|_{L^1(\R^N)}\le C s^{-1/2}\) for the heat kernel.
		
		At a point \(\omega\in\Sph\), given any unit tangent vector
		\(v\in T_\omega\Sph\), choose \(A=v\otimes\omega-\omega\otimes v.\)
		Then \(A\omega=v\) and \(\|A\|_{\mathrm{op}}=1\). Hence
		\(X_Au=v\cdot\nabla_Su\) and \(\partial_\rho X_Au
		=v\cdot\partial_\rho\nabla_Su.\)
		The bound for
		\(\partial_\rho X_Au=\omega\cdot\nabla_xz\) therefore controls all
		components of \(\partial_\rho\nS u\).
		Finally \(u_{\rho\rho}=D_x^2u[\omega,\omega]\) is bounded by
		\eqref{eq:globalderivs}.
	\end{proof}

	Finally, we consider radial derivatives in fixed-front regions. For \(L>0\), a \emph{fixed-front region} will mean
	\[
	|\rho-\RR(t)|\le L.
	\]
	The uniform spatial derivative bounds established above and one-dimensional interpolation upgrade
	the front asymptotics \eqref{eq:rrm} to the following derivative
	convergence, without differentiating the Lipschitz phase \(s^\infty\).
	
	\begin{lemma}\label{lem:radialC1}
		For every \(L>0\) and \(j=1,2\),
		\begin{equation}\label{eq:radialC1}
			\sup_{\substack{|\rho-\RR(t)|\le L\\\omega\in\Sph}}
			\left|\partial_\rho^j u(t,\rho,\omega)
			-\UU^{(j)}\bigl(\rho-\RR(t)+s^\infty(\omega)\bigr)\right|
			\longrightarrow0\qquad(t\to\infty).
		\end{equation}
	\end{lemma}
	
	\begin{proof}
		Fix \(L>0\), and for sufficiently large \(t\) set
		\[
		F_{t,\omega}(r)
		=u\bigl(t,(\RR(t)+r)\omega\bigr)
		-\UU(r+s^\infty(\omega)),\qquad |r|\le L+1.
		\]
		Then \(\eta(t):=\sup_{\omega,\,|r|\le L+1}
		|F_{t,\omega}(r)|\to0\) by \eqref{eq:rrm}. The uniform third
		spatial derivative bound and smoothness of \(\UU\) give
		\(|\partial_r^3F_{t,\omega}|\le C_L\) on this interval.
		For \(|r|\le L\) and \(0<h\le1\), centered Taylor differences
		yield
		\begin{equation*}
			\begin{aligned}
				|\partial_rF_{t,\omega}(r)|
				&\le\eta(t)/h+C_Lh^2,\\
				|\partial_{rr}F_{t,\omega}(r)|
				&\le4\eta(t)/h^2+C_Lh.
			\end{aligned}
		\end{equation*}
		Choose \(h=\eta(t)^{1/3}\) when \(0<\eta(t)\le1\).
		If \(\eta(t)=0\), both derivatives vanish on the smaller
		interval. This proves the two convergences, uniformly in \(\omega\).
	\end{proof}
	
	\section{A one-sided second angular estimate}\label{sec:angular}

	In this section, we establish a one-sided second angular derivative estimate needed for the curvature bounds in Theorem~\ref{thm:main} and the logarithmic-Hessian estimates from the fixed-front region through the diffusive scale.

	\begin{proposition}\label{prop:coreHess}
		{
			\textup{(i)} There are \(T>1\) and \(C>0\) such that, for \(t\ge T\),
			\(r\ge-3\log t\), \(\omega\in\Sph\), and every real
			skew-symmetric matrix \(A\) with \(\|A\|_{\mathrm{op}}\le1\),
			\begin{subequations}
				\begin{equation}\label{eq:globalRotation}
					X_A^2u\bigl(t,(\RR(t)+r)\omega\bigr)
					\le Ce^{-r}(1+r+3\log t).
				\end{equation}
				
				\textup{(ii)} In particular, for every \(L>0\), there are \(T_L>1\) and \(C_L>0\) such that
				\begin{equation}\label{eq:coreHess}
					\HS u(t,\rho,\omega)\le C_L\log t\,\gS
					\quad\text{if }t\ge T_L,\quad |\rho-\RR(t)|\le L.
				\end{equation}
			\end{subequations}
			The inequality in \eqref{eq:coreHess} is understood in the sense of
			quadratic forms, with spherical derivatives taken at fixed \(\rho\).
		}
	\end{proposition}
	
	{
		Assertion (i) provides the one-sided rotation estimate used for
		logarithmic-Hessian control through the diffusive scale in Section~5,
		while assertion (ii) is its fixed-front consequence used for level-set
		curvature and logarithmic-Hessian control in Section~4.
	}
	
	Roquejoffre, Rossi and Roussier-Michon
	\citep[Proposition~6 and Remark~1]{RRM2019}
	controlled first angular derivatives and noted that Kato's
	inequality does not give absolute second derivative bounds. They
	also pointed out the possibility of a one-sided upper estimate,
	without identifying how to use it. We establish such an estimate by exploiting the favorable sign in the equation for the second rotation derivative and adapting their scalar supersolution construction to a logarithmic rear boundary. The auxiliary calculations needed for
	the adaptation and for the comparison at spatial infinity are
	recorded in \hyperref[app:sec3]{the appendix}.

	\subsection{Transformed equation and logarithmic rear boundary}
	\label{Transformed_equation}
	
	We use the moving-frame and self-similar transformations of
	\citet[Section~2, equations~(8)--(17)]{RRM2019}, but use a different rear boundary.
	We first work in polar coordinates and pass to the moving frame by setting
	\[
	r=\rho-\RR(t),
	\qquad u(t,\rho,\omega)=e^{-r}v(t,r,\omega).
	\]
	Then \(v\) satisfies
	\begin{equation}\label{eq:v}
		v_t=v_{rr}
		+\left(
		\frac{N-1}{r+R_0(t)}-\frac{k}{t}
		\right)(v_r-v)
		+\frac{\DS v}{(r+R_0(t))^2}
		-e^{-r}v^2.
	\end{equation}
	
	To study the diffusive scale \(r=O(\sqrt t)\), we introduce the self-similar variables \(\eta=r/\sqrt t\) and \(\tau=\log t\).
	We then shift the self-similar coordinate:
	\begin{equation}\label{eq:logshift}
		s(\tau)=-3\tau e^{-\tau/2},
		\qquad
		\xi=\eta-s(\tau)
		=\frac r{\sqrt t}-s(\tau)
		=\frac{r+3\log t}{\sqrt t}.
	\end{equation}
	Thus \(\xi\ge0\) is the transformed half-line, and its boundary corresponds to
	the rear boundary \(r=-3\log t\). 
	Set
	\begin{align}
		D(\tau,\xi)
		&=\rho
		=2e^\tau+\xi e^{\tau/2}-(k+3)\tau,
		\label{eq:D}
		\\
		h(\tau,\xi)
		&=
		\frac{(N-1)e^\tau}{D(\tau,\xi)}-k.
		\label{eq:h}
	\end{align}
	Choose \(t_*\ge4\) so large that
	\begin{equation}\label{eq:start}
		2t-(k+3)\log t\ge t
		\qquad (t\ge t_*).
	\end{equation}
	Set \(\tau_*=\log t_*\).
	This gives \(D(\tau,\xi)=\rho\ge e^\tau\) for
	\(\tau\ge\tau_*\), \(\xi\ge0\), thus \(h\) is uniformly bounded.
	
	Finally, we apply a Gaussian conjugation and define \(W\) by the transformation
	\begin{equation}\label{eq:transform}
		u(t,\rho,\omega)
		=
		\sqrt t\,e^{-r}e^{-\xi^2/8}
		W(\tau,\xi,\omega).
	\end{equation}
	This conjugation transforms the leading self-similar operator into a self-adjoint one, as shown below.
	The calculations in Appendix~\ref{app:transform} yield the following transformed equation and coefficient estimates:
	\begin{equation}\label{eq:W}
		W_\tau+\MA W
		=
		l_1W_\xi+l_2W+\alpha\DS W-BW^2,
	\end{equation}
	where
	\begin{align}
		\MA
		&=
		-\partial_{\xi\xi}
		+\frac{\xi^2}{16}
		-\frac34,
		\label{eq:M}
		\\
		l_1
		&=
		(h-3)e^{-\tau/2},
		&
		l_2
		&=
		-\frac32-h-\frac{\xi}{4}l_1,
		\label{eq:l}
		\\
		\alpha
		&=
		\frac{e^\tau}{D^2}>0,
		&
		B
		&=
		\exp\left(
		\frac{3\tau}{2}
		-\frac{\xi^2}{8}
		-(\xi+s)e^{\tau/2}
		\right)>0.
		\label{eq:alphaB}
	\end{align}
	The coefficients are independent of \(\omega\), and we have the estimates
	\begin{equation}\label{eq:lbounds}
		|l_1|\le Ce^{-\tau/2},
		\qquad
		|l_2|
		\le
		C\bigl(\tau e^{-\tau}
		+\xi e^{-\tau/2}\bigr).
	\end{equation}
	
	\begin{remark}
		Unlike \citet{RRM2019}, who use a rear boundary
		$r=-t^\delta$ with $0<\delta<1/4$, we use the
		logarithmic boundary $r=-3\log t$ in
		\eqref{eq:logshift}, motivated by the effective
		Dirichlet-boundary mechanism discussed in
		\citet[Section~3.2]{NRR2017}.
		Indeed, at this boundary, the transformation
		\eqref{eq:transform} gives
		\(W=t^{-7/2}u\) and \(X_A^2W=t^{-7/2}X_A^2u\).
		Thus, $u\le K_0$ and Lemma~\ref{lem:firstmixed}
		yield the required boundary smallness: $W\le K_0t^{-7/2}$ and
		$|X_A^2W|\le Ct^{-5/2}$, respectively.
	\end{remark}

	\subsection{A scalar supersolution}
	
	We adapt the scalar supersolution construction of \citet[Proposition~1, formula~(24), and Section~3.2]{RRM2019} to the coefficients \(l_1,l_2\) in \eqref{eq:l}, for use in the one-sided comparison of second rotation derivatives. Set
	\begin{equation}\label{eq:P}
		\PP=\partial_\tau+\MA-l_1\partial_\xi-l_2-\alpha\DS.
	\end{equation}
	For scalar functions independent of \(\omega\), the spherical diffusion term
	\(\alpha\Delta_S\) vanishes.
	
	\begin{lemma}\label{lem:barrier}  
		There is a smooth positive function \(\Phi(\tau,\xi)\), defined for  
		\(\tau\ge\tau_*\), \(\xi\ge0\), such that  
		\begin{equation}\label{eq:barrierP}  
			\PP\Phi\ge0,\qquad  
			0<\Phi(\tau,\xi)\le C e^{-\xi^2/16}.  
		\end{equation}  
		It can be chosen to satisfy  
		\begin{align}  
			\Phi(\tau,0)&\ge c {e^{-\tau/2}}, \label{eq:barrierleft}\\  
			\Phi(\tau,\xi)&\ge c {e^{-\tau/2}}e^{-\xi^2/16}  
			&&(\xi\ge2), \label{eq:barrierfar}\\  
			\Phi(\tau,\xi)&\le C({e^{-\tau/2}}+\xi)  
			&&(0\le\xi\le a_0/2), \label{eq:barriernear}  
		\end{align}  
		where \(a_0>0\) is fixed and small.  
	\end{lemma}
	
	\begin{proof}
		Use the profiles and cutoffs of the upper-barrier construction in
		\citep[pp.~7273--7274]{RRM2019}:
		\[
		\phi_0(\xi)=\xi e^{-\xi^2/8},\qquad
		G(\xi)=e^{-\xi^2/16},\qquad
		\varphi(\xi)=\cos\frac{\pi\xi}{2a_0},
		\]
		where \(0<a_0<1/4\) will be chosen small. Take smooth
		\(0\le\gamma_1,\gamma_2\le1\) such that
		\[
		\begin{gathered}
			\gamma_1=1\text{ on }[0,a_0/2],\qquad
			\gamma_1=0\text{ on }[2a_0/3,\infty),\\
			\gamma_2=0\text{ on }[0,1],\qquad
			\gamma_2=1\text{ on }[2,\infty).
		\end{gathered}
		\]
		In particular, \(\varphi\gamma_1\ge0\). Set
		\begin{equation}\label{eq:Phi}
			\Phi=m(\tau)\varphi\gamma_1+m(\tau)G\gamma_2
			+\zeta(\tau)\phi_0,
		\end{equation}
		where \(m\) and \(\zeta\) are the solutions of the linear system:
		\begin{equation}\label{eq:ODE}
			\begin{gathered}
				m'+m=C_0\zeta e^{-\tau/2},\qquad
				\zeta'=C_0(m+\zeta e^{-\tau/2}),\\
				m(\tau_*)>0,\qquad\zeta(\tau_*)>0.
			\end{gathered}
		\end{equation}
		Here \(C_0 > 0\) will be determined later.
		The system preserves positivity, so \(m(\tau),\zeta(\tau)>0\) for all \(\tau\ge\tau_*\).
		The three-region verification, with the bounds
		\eqref{eq:lbounds} for the logarithmic rear boundary, is given
		in Appendix~\ref{app:barriercheck}. It shows that the sufficient
		conditions for \(\PP\Phi\ge0\) are
		\[
		m'+m\ge C\zeta e^{-\tau/2},\qquad
		\zeta'\ge C(m+\zeta e^{-\tau/2}).
		\]
		Both follow from \eqref{eq:ODE} once \(C_0\) is chosen sufficiently large.
		
		For the amplitude bounds, the positive solution of the linear system
		\eqref{eq:ODE} satisfies \(\zeta\ge\zeta(\tau_*)>0\). The combination
		\(H=\zeta+C_0m\) obeys
		\[
		H'=C_0(1+C_0)\zeta e^{-\tau/2}
		\le C_0(1+C_0)e^{-\tau/2}H.
		\]
		The time coefficient is integrable, so \(H\), hence \(m,\zeta\),
		is bounded. Variation of constants gives
		\[
		m(\tau)=m(\tau_*)e^{-(\tau-\tau_*)}
		+C_0\int_{\tau_*}^{\tau}
		e^{-(\tau-\sigma)}\zeta(\sigma)e^{-\sigma/2}\,d\sigma.
		\]
		Using the positive lower and finite upper bounds for \(\zeta\),
		and positivity on the initial compact time interval, yields
		\begin{equation}\label{eq:ampbounds}
			ce^{-\tau/2}\le m(\tau)\le Ce^{-\tau/2}
			\qquad(\tau\ge\tau_*).
		\end{equation}
		Finally, each spatial profile in \eqref{eq:Phi} is
		bounded by a multiple of \(G\); at zero \(\Phi=m\), for
		\(\xi\ge2\) we have \(\Phi\ge mG\), and near zero
		\(\Phi\le C(m+\zeta\xi)\). These facts prove
		\eqref{eq:barrierP}--\eqref{eq:barriernear}.
	\end{proof}

	\subsection{One-sided comparison}
	
	The comparison starts at \(t_*>1\), when the solution is no longer compactly supported, and takes place on the unbounded half-line \(\xi\ge0\). We therefore need a finite-time Gaussian tail estimate to control the initial data and justify exhaustion by bounded cylinders.
	
	\begin{lemma}\label{lem:tails}
		For each finite \(T_1>t_*\), there are constants \(C,\overline R\),
		depending on \(T_1\), such that
		\begin{equation}\label{eq:physicaltail}
			|\nabla_xu(t,x)|+|D_x^2u(t,x)|
			\le C\exp\left(-\frac{(|x|-\overline R)^2}{4(t-1)}\right)
		\end{equation}
		for \(t_*\le t\le T_1\) and all sufficiently large \(|x|\).
		Consequently,
		\begin{equation}\label{eq:qtail}
			\sup_{\substack{\tau_*\le\tau\le\log T_1\\
					\omega\in\Sph,\ \|A\|_{\mathrm{op}}\le1}}
			(X_A^2W)_+(\tau,\xi,\omega)
			\le C_{T_1}e^{-3\xi^2/32}
			\qquad(\xi\ge0).
		\end{equation}
	\end{lemma}
	
	The proof uses heat-kernel comparison and interior parabolic
	regularity; see Appendix~\ref{app:tails}.

	\begin{lemma}\label{lem:secondW}
		There is \(C>0\), independent of \(A\), such that
		\begin{equation}\label{eq:secondW}
			(X_A^2W)_+\le C\Phi(\tau,\xi)
			\quad\text{for }\tau\ge\tau_*,\quad \xi\ge0,\quad
			\omega\in\Sph,\quad \|A\|_{\mathrm{op}}\le1.
		\end{equation}
	\end{lemma}
	
	\begin{proof}
		We follow the scalar-comparison strategy of
		\citet[Proposition~6]{RRM2019},
		using the one-sided sign noted in their Remark~1. For a fixed admissible \(A\), put
		\(q_A=X_A^2W\). By \eqref{eq:commute} and the angular independence
		of the coefficients in \eqref{eq:W},
		\begin{equation}\label{eq:qeq}
			\PP q_A+2BWq_A=-2B(X_AW)^2.
		\end{equation}
		Since \(B,W,\Phi\ge0\) and \(\PP\Phi\ge0\), the smooth difference
		\(z_A=q_A-C_*\Phi\), for any \(C_*>0\), satisfies
		\begin{equation}\label{eq:directcomparison}
			(\PP+2BW)z_A
			=-2B(X_AW)^2-C_*(\PP\Phi+2BW\Phi)\le0.
		\end{equation}
		
		We choose \(C_*\) using the initial and rear-boundary data.
		At \(\xi=0\), \(r=-3\log t\) and \(\rho=O(t)\).
		Lemma~\ref{lem:firstmixed} gives
		\[
		|X_A^2u|=|(Ax)\cdot\nabla_x(X_Au)|\le C\rho\le Ct.
		\]
		Consequently, uniformly in \(A,\omega\),
		\begin{equation*}
			|q_A(\tau,0,\omega)|
			=t^{-1/2}t^{-3}|X_A^2u|\le Ct^{-5/2}.
		\end{equation*}
		Together with \eqref{eq:barrierleft}, this bounds the ratio
		\((q_A)_+/\Phi\) on the entire rear boundary by \(Ce^{-2\tau}\).
		At the initial time, apply Lemma~\ref{lem:tails} with the fixed
		horizon \(T_0=t_*+1\). Positivity of \(\Phi(\tau_*,\cdot)\)
		on compact intervals and \eqref{eq:barrierfar} in the tail give
		\[
		(q_A)_+(\tau_*,\xi,\omega)\le C_*\Phi(\tau_*,\xi),\qquad
		(q_A)_+(\tau,0,\omega)\le C_*\Phi(\tau,0).
		\]
		Choose \(C_*\ge1\) so that both inequalities hold for all admissible
		\(A\). 
		
		Now fix \(T_1>t_*\). For \(\xi\ge2\),
		\eqref{eq:qtail} and \eqref{eq:barrierfar} imply
		\[
		\sup_{\substack{\tau_*\le\tau\le\log T_1\\
				\omega\in\Sph,\ \|A\|_{\mathrm{op}}\le1}}
		\frac{(q_A)_+(\tau,\xi,\omega)}{\Phi(\tau,\xi)}
		\le\frac{C_{T_1}\sqrt{T_1}}c e^{-\xi^2/32}
		\longrightarrow0.
		\]
		Hence, for every sufficiently large \(L_1\), depending on \(T_1\)
		but not on \(A\), we have \(z_A\le0\) on the parabolic boundary of
		\[
		Q=(\tau_*,\log T_1)\times(0,L_1)\times\Sph.
		\]
		On \(\overline Q\), all coefficients are bounded and
		\(\alpha\) is bounded away from zero. The standard parabolic
		comparison principle, after the usual exponential normalization
		of the zeroth-order term, gives \(z_A\le0\) on \(\overline Q\).
		
		For this fixed \(T_1\), taking arbitrarily large \(L_1\) covers every
		finite \(\xi\). Since \(C_*\) was chosen before \(T_1\), letting
		\(T_1\) increase covers every finite \(\tau\ge\tau_*\) with the same
		constant. Thus \(q_A\le C_*\Phi\) everywhere, and hence
		\((q_A)_+\le C_*\Phi\), uniformly for
		\(\|A\|_{\mathrm{op}}\le1\).
	\end{proof}

	We now return to physical variables to complete the proof of Proposition~\ref{prop:coreHess}.
	
	\begin{proof}[Proof of Proposition~\ref{prop:coreHess}]
		{
			We first prove (i). For \(0\le\xi\le a_0/2\),
			\eqref{eq:barriernear} gives
			\(\Phi(\tau,\xi)\le C(e^{-\tau/2}+\xi)\).
			For \(\xi\ge a_0/2\), the Gaussian upper bound in
			\eqref{eq:barrierP} gives
			\(\Phi(\tau,\xi)\le Ce^{-\xi^2/16}\le C\xi\).
			Consequently,
			\[
			\Phi(\tau,\xi)\le C(e^{-\tau/2}+\xi)
			\qquad (\tau\ge\tau_*,\ \xi\ge0).
			\]
			For \(t\ge t_*\) and \(r\ge-3\log t\), we have
			\(\xi=(r+3\log t)/\sqrt t\ge0\).
			Since the prefactor in \eqref{eq:transform} is radial,
			Lemma~\ref{lem:secondW} gives, uniformly in \(\omega\) and
			\(\|A\|_{\mathrm{op}}\le1\),
			\[
			\begin{aligned}
				X_A^2u\bigl(t,(\RR(t)+r)\omega\bigr)
				&\le C\sqrt t\,e^{-r-\xi^2/8}\Phi(\tau,\xi)\\
				&\le Ce^{-r}(1+r+3\log t).
			\end{aligned}
			\]
			
			We now derive (ii). Fix \(\omega\in\Sph\) and a unit tangent vector
			\(v\perp\omega\). Choose
			\(A=v\otimes\omega-\omega\otimes v\); then
			\[
			A\omega=v,\quad Av=-\omega,\qquad
			e^{sA}\omega=(\cos s)\omega+(\sin s)v.
			\]
			The rotation orbit is a unit-speed great circle, so
			\(X_A^2u=\HS u[v,v]\) at fixed \(\rho\).
			For a fixed \(L>0\), choose \(t\) large enough that
			\(3\log t\ge L\). On \(|r|\le L\), \eqref{eq:globalRotation} then yields
			\[
			\HS u(t,\rho,\omega)[v,v]
			\le Ce^{-r}(1+r+3\log t)
			\le C_L\log t.
			\]
			Since \(v\) was arbitrary, this proves (ii).
		}
	\end{proof}
	
	\section{Front geometry and logarithmic-Hessian estimates}\label{sec:front-geometry-concavity}
	
	We first establish the compact-level geometry and curvature bounds, then obtain a negative upper bound for the full logarithmic Hessian in the fixed-front region. Using the fixed-front estimates, we propagate negativity of the ordinary Hessian through the invaded interior.
	
	\subsection{Geometry and curvature on compact level intervals}\label{sec:geometry}

	\begin{proof}[Proof of Theorem~\ref{thm:main}]
		Fix a closed interval \(K\subset(0,1)\). It suffices to work with
		\(K=[\varepsilon,1-\varepsilon]\), where \(0<\varepsilon<1/2\). 
		We first establish the radial-graph representation
		\eqref{eq:radialmain} and the location asymptotics
		\eqref{eq:locationmain}.
		Set
		\[
		J=[z_{1-\varepsilon}-1,z_\varepsilon+1],\qquad
		m=\min_{z\in J}(-\UU'(z))>0,
		\]
		where \(z_\lambda=\UU^{-1}(\lambda).\)
		At the endpoints of \(J\), the values of \(\UU\) lie strictly above
		\(1-\varepsilon\) and below \(\varepsilon\), respectively.
		The uniform convergence \eqref{eq:rrm} therefore shows that, for all
		sufficiently large \(t\), the same inequalities hold for \(u\) at the
		corresponding radii, uniformly in \(\omega\).
		By the intermediate value theorem, each ray meets
		\(\Gamma_\lambda(t)\) for every \(\lambda\in K\).
		By the strict monotonicity of \(\UU\) and the same uniform convergence,
		every such intersection \((\rho,\omega)\) lies in the phase band
		\(\rho-\RR(t)+s^\infty(\omega)\in J\).
		Since \(s^\infty\) is bounded, this phase band lies in a fixed-front
		region \(|\rho-\RR(t)|\le L_K\) for some \(L_K>0\).
		Lemma~\ref{lem:radialC1} gives \(-u_\rho\ge m/2\)
		throughout the phase band
		\(\rho-\RR(t)+s^\infty(\omega)\in J\)
		for all sufficiently large \(t\). Thus each ray has exactly one crossing,
		and the implicit function theorem gives the smooth radial graph
		\eqref{eq:radialmain}. We have also established, uniformly in
		\(\lambda\in K\) and \(\omega\in\Sph\),
		\begin{equation}\label{eq:nondegenerate}
			-u_\rho(t,R_\lambda(t,\omega),\omega)\ge c_K:=m/2>0.
		\end{equation}
		At a crossing,
		\[
		\UU(\rho-\RR(t)+s^\infty(\omega))=\lambda+o(1)
		\]
		uniformly in \(\lambda\) and \(\omega\). Since \(-\UU'\ge m\) on \(J\),
		\[
		\rho-\RR(t)+s^\infty(\omega)=z_\lambda+o(1)
		\]
		uniformly, proving \eqref{eq:locationmain}.

		We next establish the derivative bounds \eqref{eq:derivmain}.
		Differentiate \(u(t,R_\lambda(t,\omega),\omega)=\lambda\) at fixed \(t\).
		All derivatives of \(u\) below are evaluated at
		\((t,R_\lambda(t,\omega),\omega)\), with angular derivatives first
		taken at fixed \(\rho\). One differentiation gives
		\begin{equation*}
			\nS R_\lambda=-\frac{\nS u}{u_\rho}.
		\end{equation*}
		In fact, the reflection estimate \eqref{eq:angularcone} already yields
		\begin{equation}\label{eq:Rgradient}
			|\nS R_\lambda|
			\le\frac{R_*}{\sqrt{1-R_*^2/R_\lambda^2}}\le C
			\qquad(t\text{ sufficiently large}).
		\end{equation}
		For the second derivatives, the covariant chain rule gives
		\begin{equation}\label{eq:implicit2}
			\begin{split}
				0={}&u_\rho\HS R_\lambda+\HS u
				+d_Su_\rho\otimes dR_\lambda+dR_\lambda\otimes d_Su_\rho\\
				&\hspace{25mm}+u_{\rho\rho}\,dR_\lambda\otimes dR_\lambda .
			\end{split}
		\end{equation}
		Equation \eqref{eq:implicit2} is an identity of symmetric \(2\)-tensors on
		\(T_\omega\mathbb S^{N-1}\).
		Applying {\eqref{eq:coreHess}} with \(L=L_K\)
		and using Lemma~\ref{lem:firstmixed},
		\eqref{eq:nondegenerate}, and \eqref{eq:Rgradient}
		in \eqref{eq:implicit2}, we obtain, for every unit tangent
		vector \(v\),
		\begin{align}
			\HS R_\lambda[v,v]
			&=\frac{\HS u[v,v]
				+2\,d_Su_\rho[v]\,dR_\lambda[v]
				+u_{\rho\rho}\bigl(dR_\lambda[v]\bigr)^2}{-u_\rho}\notag\\
			&\le C(\log t+1)\le C\log t. \label{eq:RHess}
		\end{align}
		The uniform lower bound on \(-u_\rho\) is essential: an upper bound on
		\(\HS u\) gives the upper bound on \(\HS R_\lambda\) required by
		the curvature formula.
		
		To establish the curvature bound \eqref{eq:curvmain},
		write \(R=R_\lambda(t,\cdot)\) and parametrize the level hypersurface by
		\(F(\omega)=R(\omega)\omega\). Its outward normal is
		\[
		\nu=\frac{R\omega-\nS R}{\sqrt{R^2+|\nS R|^2}}.
		\]
		The first and second fundamental forms, with the convention specified
		before Theorem~\ref{thm:main}, are
		\begin{equation}\label{eq:forms}
			I=R^2\gS+dR\otimes dR,\qquad
			\mathrm{II}
			=\frac{R^2\gS+2\,dR\otimes dR-R\HS R}
			{\sqrt{R^2+|\nS R|^2}}.
		\end{equation}
		For a spherical unit tangent vector \(v\), \eqref{eq:RHess} gives,
		once \(R-C\log t>0\),
		\begin{equation}\label{eq:rayleigh}
			\frac{\mathrm{II}(v,v)}{I(v,v)}
			\ge
			\frac{R^2-CR\log t}{(R^2+|\nS R|^2)^{3/2}}.
		\end{equation}
		From \eqref{eq:locationmain} and \eqref{eq:Rgradient},
		\[
		R=2t+O(\log t),\qquad |\nS R|=O(1),
		\]
		uniformly in \(\lambda,\omega\). Substituting into \eqref{eq:rayleigh} and enlarging \(C\) if necessary,
		we obtain, after possibly increasing \(T\),
		\begin{equation}\label{eq:curvproved}
			\frac{\mathrm{II}(v,v)}{I(v,v)}
			\ge \frac{1}{2t}\left(1-C\frac{\log t}{t}\right)>0,
			\qquad t\ge T.
		\end{equation}
		The principal curvatures are the eigenvalues of \(\mathrm{II}\) with
		respect to \(I\). The Rayleigh-quotient bound proves
		\eqref{eq:curvmain} for all \(N-1\) principal curvatures.

		Since $\Gamma_\lambda(t)$ is a smooth closed radial graph
		with positive definite second fundamental form,
		the classical Hadamard convexity theorem
		(or its planar counterpart when $N=2$) implies
		that it bounds a strictly convex body.
		Moreover, the local uniform convergence
		$u(t,\cdot)\to 1$ in $\mathbb R^N$ gives
		$u(t,0)>\max K$ for all sufficiently large $t$.
		The unique intersection of each ray from the origin
		with $\Gamma_\lambda(t)$ then shows that this body
		is precisely $\{u(t,\cdot)\ge\lambda\}$.
	\end{proof}

	\subsection{Critical-wave concavity and the fixed-front logarithmic Hessian}\label{subsec:fixed-front-loghessian}
	
	We first establish the strict log-concavity of the critical travelling wave. 
	
	\begin{lemma}\label{lem:waveconcavity}
		The profile \(\UU\) satisfies
		\begin{equation}\label{eq:wavelogconcave}
			(\log\UU)''(z)<0\qquad(z\in\R).
		\end{equation}
		Moreover,
		\begin{equation}\label{eq:waveconcavebehind}
			\UU''(z)<0\qquad\text{whenever }\UU(z)\ge\frac12.
		\end{equation}
	\end{lemma}
	
	\begin{proof}

		Write \(U=\UU\) and \(Q=\log U\). Dividing \eqref{eq:wave} by
		\(U\), then differentiating, gives
		\[
		Q''+(Q')^2+2Q'+1-U=0,\qquad
		Q'''+2(1+Q')Q''=U'.
		\]
		Consequently,
		\((e^{2z}U^2Q'')'=e^{2z}U^2U'\).
		Since \(U(z)\to1\) as \(z\to-\infty\) and its
		derivatives are bounded, the integrating-factor
		boundary term vanishes. Thus
		\[
		Q''(z)=\frac{e^{-2z}}{U(z)^2}
		\int_{-\infty}^z e^{2s}U(s)^2U'(s)\,ds<0,
		\]
		since \(U'<0\).
		
		Differentiating \eqref{eq:wave} directly gives
		\(U'''+2U''+(1-2U)U'=0\).
		Another integrating factor yields
		\begin{equation}\label{eq:wave-second-integral}
			U''(z)=-e^{-2z}\int_{-\infty}^z
			e^{2s}(1-2U(s))U'(s)\,ds.
		\end{equation}
		If \(U(z)\ge1/2\), strict monotonicity implies \(U(s)>1/2\)
		for \(s<z\); the integrand in \eqref{eq:wave-second-integral}
		is then positive. This proves \eqref{eq:waveconcavebehind}.
	\end{proof}
	
	In the fixed-front region, this limiting profile provides the required radial negativity. Combining the resulting radial bounds with the one-sided spherical Hessian estimate of Proposition~\ref{prop:coreHess}(ii) and the mixed derivative estimates of Lemma~\ref{lem:firstmixed}, we obtain the full logarithmic-Hessian bound in each fixed-front region.
	
	\begin{proposition}\label{prop:frontlog}
		For every \(L>0\), there are \(T_L,c_L>0\) such that
		\begin{equation}\label{eq:frontlog}
			D_x^2\log u(t,x)\le-\frac{c_L}{t}\id
			\quad\text{if }t\ge T_L,\quad
			\bigl||x|-\RR(t)\bigr|\le L.
		\end{equation}
	\end{proposition}
	
	\begin{proof}
		Set \(q=\log u\). In the fixed-front region, \eqref{eq:rrm}
		and Lemma~\ref{lem:radialC1} give, uniformly in \(\rho,\omega\),
		\[
		q_\rho\longrightarrow
		(\log\UU)'(\rho-\RR(t)+s^\infty(\omega)),\qquad
		q_{\rho\rho}\longrightarrow
		(\log\UU)''(\rho-\RR(t)+s^\infty(\omega)).
		\]
		All phases here lie in one compact interval. By
		Lemma~\ref{lem:waveconcavity} and \(\UU'<0\), there exist
		\(a_L,b_L>0\) such that, for large \(t\),
		\begin{equation}\label{eq:radial-log-negativity}
			q_{\rho\rho}\le-a_L,\qquad q_\rho\le-b_L.
		\end{equation}
		The same compactness gives a positive lower bound for \(u\) in
		this region. Hence Lemma~\ref{lem:firstmixed} and
		{\eqref{eq:coreHess}} imply
		\begin{equation}\label{eq:angular-log-bounds}
			|\nS q|+|\partial_\rho\nS q|\le C_L,\qquad
			\HS q=\frac{\HS u}{u}-\frac{d_Su\otimes d_Su}{u^2}
			\le C_L\log t\,\gS.
		\end{equation}
		
		At \(x=\rho\omega\), identify \(T_\omega\Sph\) with
		\(\omega^\perp\subset\R^N\). For any smooth scalar function
		\(f=f(\rho,\omega)\) and \(v\perp\omega\), the polar-coordinate
		identities are
		\begin{equation}\label{eq:euclidean-polar-hessian}
			\begin{aligned}
				D_x^2f[\omega,\omega]&=f_{\rho\rho},\\
				D_x^2f[\omega,v]
				&=\rho^{-1}d_Sf_\rho[v]-\rho^{-2}d_Sf[v],\\
				D_x^2f[v,v]
				&=\rho^{-1}f_\rho|v|^2+\rho^{-2}\HS f[v,v].
			\end{aligned}
		\end{equation}
		Since \(\rho=2t+O(\log t)\),
		\eqref{eq:radial-log-negativity}--\eqref{eq:angular-log-bounds}
		give, after increasing the starting time,
		\[
		D_x^2q[v,v]\le-\frac{b}{t}|v|^2,
		\qquad |D_x^2q[\omega,v]|\le\frac Ct|v|,
		\]
		where \(b>0\) and \(C\) depend on \(L,N,u_0\).
		For \(\zeta=a\omega+v\), Young's inequality yields
		\[
		\begin{aligned}
			D_x^2q[\zeta,\zeta]
			&\le-a_La^2+\frac{2C}{t}|a||v|-\frac bt|v|^2\\
			&\le-\frac{a_L}{2}a^2
			-\left(\frac bt-\frac{2C^2}{a_Lt^2}\right)|v|^2
			\le-\frac{c_L}{t}|\zeta|^2
		\end{aligned}
		\]
		for all sufficiently large \(t\). This proves \eqref{eq:frontlog}.
	\end{proof}

	\subsection{Hessian control in the invaded interior}\label{subsec:invaded-interior}
	
	We first use the fixed-front estimates above to provide negative Hessian bounds on a moving sphere at a fixed distance behind the front.

	\begin{lemma}\label{lem:rearboundary}
		There are \(L_0>0\), \(T>1\) and \(c_0>0\) such that, with
		\begin{equation}\label{eq:inner-radius}
			R_{\mathrm{in}}(t)=\RR(t)-L_0,
			\qquad \mathcal D(t)=B_{R_{\mathrm{in}}(t)},
		\end{equation}
		we have, for every \(t\ge T\),
		\begin{equation}\label{eq:inner-state}
			0<R_{\mathrm{in}}(t)\le2t,\qquad
			u(t,x)\ge\frac{19}{20}\quad(x\in\overline{\mathcal D(t)}),
		\end{equation}
		and
		\begin{equation}\label{eq:rear-Hessian}
			D_x^2u(t,x)\le-\frac{c_0}{t}\id
			\qquad(x\in\partial\mathcal D(t)).
		\end{equation}
	\end{lemma}
	
	\begin{proof}
		Choose \(L_0>0\) large enough that
		\[
		\min_{\omega\in\Sph}\UU(-L_0+s^\infty(\omega))
		>\frac{39}{40}.
		\]
		Monotonicity of \(\UU\) and the global uniform convergence
		\eqref{eq:rrm} then imply \(u\ge19/20\) throughout
		\(\overline{\mathcal D(t)}\),
		once the error in that convergence
		is at most \(1/40\). At the origin this follows by continuity.
		After increasing \(T\), \(R_{\mathrm{in}}\) is positive and
		smooth (and may be taken increasing); its upper bound by \(2t\)
		follows from \(k,L_0>0\).
		
		On \(\rho=R_{\mathrm{in}}(t)\), the limiting phases belong to
		the compact set \(\{-L_0+s^\infty(\omega):\omega\in\Sph\}\).
		Both \(\UU'\) and \(\UU''\) are strictly negative there, the
		latter by Lemma~\ref{lem:waveconcavity}. Lemma~\ref{lem:radialC1}
		therefore gives \(u_\rho\le-b_{\mathrm{rear}}\) and
		\(u_{\rho\rho}\le-a_{\mathrm{rear}}\) on this moving sphere, with
		\(a_{\mathrm{rear}},b_{\mathrm{rear}}>0\). Applying \eqref{eq:euclidean-polar-hessian} to
		\(f=u\), Lemma~\ref{lem:firstmixed} and
		{\eqref{eq:coreHess}} show that, for \(v\perp\omega\),
		\[
		D_x^2u[v,v]\le-\frac bt|v|^2,
		\qquad |D_x^2u[\omega,v]|\le\frac Ct|v|.
		\]
		The same Young inequality as in Proposition~\ref{prop:frontlog}
		absorbs the mixed block and proves \eqref{eq:rear-Hessian}.
		All constants depend on this single fixed \(L_0\), not on a level
		parameter.
	\end{proof}
	
	To propagate this boundary negativity throughout the enclosed ball, where
	\(u\) is uniformly close to the stable state \(1\), we construct a supersolution and 
	compare the Cartesian second derivatives of \(u\) with it.
	
	\begin{proposition}\label{prop:innerconcavity}
		For \(L_0\) as in Lemma~\ref{lem:rearboundary}, there are
		\(T_1>1\) and \(c_{\mathrm{in}}>0\) such that
		\begin{equation}\label{eq:inner-concavity}
			D_x^2u(t,x)\le
			-\frac{c_{\mathrm{in}}}{t}e^{-\beta_*t}\id
			\quad(t\ge T_1,\ x\in\overline{\mathcal D(t)}),
		\end{equation}
		where \(\beta_*=2(\sqrt2-1)\). The same upper bound holds for
		\(D_x^2\log u\) on this region.
	\end{proposition}
	
	\begin{proof}
		Let \(T\) and \(c_0\) be supplied by
		Lemma~\ref{lem:rearboundary}, with \(T\) enlarged
		so that \(T\ge T_*\).
		For a fixed unit vector \(e\in\R^N\), write \(h_e=u_{ee}\).
		Two Cartesian differentiations of \eqref{eq:kpp} give
		\begin{equation}\label{eq:cartesian-second-equation}
			(h_e)_t-\Delta h_e+(2u-1)h_e=-2(u_e)^2\le0.
		\end{equation}
		Set \(\mathcal L_{\mathrm{in}}
		=\partial_t-\Delta+(2u-1)\).
		In the moving inner ball, \eqref{eq:inner-state}
		implies \(9/10\le2u-1\le1\).
		
		We now construct a supersolution for \(h_e\).
		Choose \(\mu=\sqrt2-1\), \(\beta_*=2\mu\),
		and \(\gamma=9/10\), so that
		\(\mu^2+\beta_*=1\) and
		\(0<\beta_*<\gamma<1\).
		Define the smooth positive function
		\[
		\Psi(x)=\frac1N\sum_{j=1}^N\cosh(\mu x_j).
		\]
		Then \(\Delta\Psi=\mu^2\Psi\) and
		\(1\le\Psi(x)\le e^{\mu|x|}\) for \(x\in\R^N\).
		For constants \(A,M>0\) to be chosen, put
		\begin{equation}\label{eq:inner-barrier}
			w(t,x)=\frac At e^{-\beta_*t}\Psi(x),\qquad
			p(t)=M e^{-\gamma(t-T)},\qquad
			Z(t,x)=p(t)-w(t,x).
		\end{equation}
		Using the preceding parameter relations and the properties of \(\Psi\),
		\begin{equation}\label{eq:inner-barrier-sign}
			\begin{aligned}
				\mathcal L_{\mathrm{in}}w
				&=\left(-\beta_*-\frac1t-\mu^2+2u-1\right)w
				=\left(2u-2-\frac1t\right)w\le-\frac wt,\\
				\mathcal L_{\mathrm{in}}p&=(2u-1-\gamma)p\ge0.
			\end{aligned}
		\end{equation}
		In particular \(\mathcal L_{\mathrm{in}}Z\ge0\).
		
		Since \(R_{\mathrm{in}}(t)\le2t\) and
		\(\beta_*=2\mu\),
		\begin{equation}\label{eq:inner-w-bounds}
			\frac At e^{-\beta_*t}\le w(t,x)\le\frac At
			\qquad(t\ge T,\ x\in\overline{\mathcal D(t)}).
		\end{equation}
		Choose \(0<A\le c_0\).
		On the moving lateral boundary,
		\eqref{eq:rear-Hessian} gives
		\(h_e\le-c_0/t\le-w\le Z\), uniformly in \(e\).
		Let \(C_2\ge\sup_x\|D_x^2u(T,x)\|_{\mathrm{op}}\)
		and take \(M>C_2+A/T\). Then
		\(h_e(T,x)\le C_2<M-w(T,x)=Z(T,x)\)
		on \(\overline{\mathcal D(T)}\), again uniformly in \(e\).
		
		For any finite \(S>T\), apply the classical
		scalar comparison principle on
		\(Q_{T,S}=\{(t,x):T<t<S,\ |x|<R_{\mathrm{in}}(t)\}\).
		For clarity, the change \(x=R_{\mathrm{in}}(t)y\)
		maps this domain to the fixed unit ball
		and changes the operator to
		\[
		\partial_t-R_{\mathrm{in}}(t)^{-2}\Delta_y
		-\frac{R_{\mathrm{in}}'(t)}{R_{\mathrm{in}}(t)}
		y\cdot\nabla_y
		+(2u(t,R_{\mathrm{in}}(t)y)-1).
		\]
		It is uniformly parabolic on every such finite
		cylinder and its coefficients are bounded there.
		Equations \eqref{eq:cartesian-second-equation}
		and \eqref{eq:inner-barrier-sign}, together with
		the initial and lateral data, therefore imply
		\(h_e\le Z\).
		Neither the comparison function nor \(A,M,T\)
		depends on \(S\) or \(e\).
		Letting \(S\) increase and using all unit vectors gives
		\[
		D_x^2u(t,x)\le(p(t)-w(t,x))\id
		\quad(t\ge T,\ x\in\overline{\mathcal D(t)}).
		\]

		Finally, the lower bound in \eqref{eq:inner-w-bounds} yields
		\begin{equation}\label{eq:inner-ratio}
			\sup_{x\in\overline{\mathcal D(t)}}\frac{p(t)}{w(t,x)}
			\le\frac MA\,t\,e^{\gamma T}
			e^{-(\gamma-\beta_*)t}\longrightarrow0.
		\end{equation}
		Choose \(T_1\ge T\) so that this ratio is at most \(1/2\)
		for every \(t\ge T_1\). Then
		\[
		D_x^2u(t,x)\le-\frac12w(t,x)\id
		\le-\frac A{2t}e^{-\beta_*t}\id.
		\]
		This is \eqref{eq:inner-concavity} with \(c_{\mathrm{in}}=A/2\).
		Since \(0<u<1\),
		\begin{equation}\label{eq:inner-log-concavity}
			D_x^2\log u
			=\frac{D_x^2u}{u}
			-\frac{\nabla_xu\otimes\nabla_xu}{u^2}
			\le-\frac{c_{\mathrm{in}}}{t}e^{-\beta_*t}\id
		\end{equation}
		on the same region.
	\end{proof}

	\section{The logarithmic Hessian through the diffusive scale}
	\label{cont:sec:through-diffusive}
	
	We now extend the fixed-front logarithmic-Hessian estimate into the leading edge, up to the diffusive scale.
	
	\begin{proposition}[Logarithmic Hessian bound through the diffusive scale]
		\label{cont:prop:through-diffusive}
		For every fixed \(K>0\), there are
		\(L_K,T_K,c_K>0\) such that
		\begin{equation}\label{cont:eq:through-diffusive-hessian}
			D_x^2\log u(t,x)\le-\frac{c_K}{t}\id
		\end{equation}
		whenever \(t\ge T_K\) and
		\(L_K\le |x|-R_0(t)\le K\sqrt t\).
	\end{proposition}

	We begin by recording two inputs for the proof. First, recalling the conjugation \(v(t,r,\omega)=e^r u(t,R_0(t)+r,\omega)\) introduced in Section~\ref{Transformed_equation}, the rotation estimate \eqref{eq:globalRotation} gives, for every real skew-symmetric matrix \(A\) with \(\|A\|_{\rm op}\le1\),
	\begin{equation}\label{cont:eq:second-rotation-input}
		X_A^2v\le C(1+r+\log t)
		\qquad (r \ge 0)
	\end{equation}
	for all sufficiently large \(t\).

	Next, we explain how the diffusive-zone asymptotics of Roquejoffre, Rossi and Roussier-Michon apply under \eqref{eq:data}.
	Let \(z\) solve \(z_t=\Delta z+z-z^2\) with
	\(z(1,x)=K_0^{-1}u_0(x)\).
	The general-data remark of \citet{RRM2019}
	allows us to apply their Proposition~5 to \(z\).
	Since \(z\le u\le K_0z\), its weighted upper and
	positive lower bounds transfer to \(u\).
	The first angular estimate for \(u\) follows
	from the Kato argument of
	\cite[Proposition~6]{RRM2019}, applied to its
	differentiated equation.
	The compactness and half-line spectral argument
	in the proof of \cite[Theorem~2.1]{RRM2019}
	then implies that
	there is a positive Lipschitz
	function \(\mathfrak c\) on \(\Sph\) such that
	\begin{equation}\label{cont:eq:rrm-diffusive}
		t^{-1/2}v(t,\sqrt t\,\eta,\omega)
		\longrightarrow
		\mathfrak c(\omega)\eta e^{-\eta^2/4}
	\end{equation}
	uniformly for \(\eta\) in compact subsets of \((0,\infty)\),
	uniformly in \(\omega\).

	To prove Proposition~\ref{cont:prop:through-diffusive}, we first establish two-sided bounds for \(v\) from a fixed distance ahead of the front up to the diffusive scale. Together with the rotation estimate above, these bounds allow us to derive the angular and mixed relative estimates. Radial logarithmic concavity is obtained separately through subdiffusive and diffusive blow-up arguments.

	\subsection{Leading-edge bounds and nondegeneracy}
	\label{cont:sec:leading-edge}
	
	We first adapt the scalar supersolution construction of Section 3 to obtain a Gaussian upper bound for \(v\) and a corresponding absolute bound for its first angular derivatives.
	
	\begin{lemma}[Gaussian upper bounds in the leading edge]
		\label{cont:lem:upper-leading-edge}
		There are \(T_{\rm le},C>0\) such that, for
		\(t\ge T_{\rm le}\), \(r\ge0\), and every admissible \(A\),
		\begin{equation}\label{cont:eq:upper-conjugate}
			v(t,r,\omega)
			+|X_Av(t,r,\omega)|
			\le C(1+r)e^{-r^2/(8t)}.
		\end{equation}
		Consequently,
		\begin{equation}\label{cont:eq:upper-solution}
			u(t,(R_0(t)+r)\omega)
			\le C(1+r)e^{-r-r^2/(8t)}
			\qquad(r\ge0).
		\end{equation}
	\end{lemma}

	\begin{proof}
		
		Set $\tau=\log t$, $\eta=r/\sqrt t$, and
		$W^0=t^{-1/2}e^{\eta^2/8}v$.
		Repeating the calculation of
		Section~\ref{Transformed_equation}
		without the logarithmic shift gives
		\[
		\mathcal P_0W^0=-B_0(W^0)^2,
		\qquad
		\mathcal P_0
		=\partial_\tau+\mathcal M
		-l_1^0\partial_\eta-l_2^0-\alpha_0\DS,
		\]
		where $\mathcal M=-\partial_{\eta\eta}
		+\eta^2/16-3/4$ and $B_0,\alpha_0>0$.
		With \(D=\rho\) and
		\(h=(N-1)e^\tau/D-k\) as in Section~\ref{Transformed_equation},
		we have
		\[
		l_1^0=he^{-\tau/2},\qquad
		l_2^0=-\frac32-h-\frac{\eta}{4}l_1^0.
		\]
		Consequently, for all sufficiently large $\tau$,
		$|l_1^0|\le Ce^{-\tau/2}$ and
		$|l_2^0|\le C(1+\eta)e^{-\tau/2}$.
		These are precisely the coefficient bounds needed in the
		three-profile construction of Lemma~\ref{lem:barrier}
		and Appendix~\ref{app:barriercheck}.
		Using the same amplitude system \eqref{eq:ODE},
		we therefore obtain a positive function
		\(\Phi_{\rm le}\) satisfying
		\[
		\mathcal P_0\Phi_{\rm le}\ge0,\qquad
		\Phi_{\rm le}(\tau,0)\ge ct^{-1/2},\qquad
		\Phi_{\rm le}\le Ce^{-\eta^2/16},
		\]
		and
		\[
		\Phi_{\rm le}\le C(t^{-1/2}+\eta)
		\quad(0\le\eta\le a_0/2),\qquad
		\Phi_{\rm le}\ge ct^{-1/2}e^{-\eta^2/16}
		\quad(\eta\ge2).
		\]

		At $r=0$, the bound $0<u\le K_0$ and
		Lemma~\ref{lem:firstmixed} give
		$W^0+|X_AW^0|\le Ct^{-1/2}$, uniformly in
		the admissible $A$.
		Moreover,
		\[
		\mathcal P_0W^0=-B_0(W^0)^2\le0,\qquad
		(\mathcal P_0+2B_0W^0)X_AW^0=0.
		\]
		The finite-time Gaussian tails of
		Lemma~\ref{lem:tails} control the initial data
		and justify comparison at spatial infinity.
		Together with
		\(\Phi_{\rm le}(\tau,0)\ge ct^{-1/2}\),
		these estimates allow us to choose a fixed
		comparison constant, independent of \(A\)
		and the terminal time.
		Applying the bounded-cylinder comparison and
		exhaustion argument of Lemma~\ref{lem:secondW}
		to \(W^0\) and separately to \(\pm X_AW^0\),
		we obtain
		\[
		W^0+|X_AW^0|\le C\Phi_{\rm le}
		\qquad(\tau\ge\log T_{\rm le},\ \eta\ge0).
		\]

		For \(0\le\eta\le a_0/2\), the near-boundary
		bound for \(\Phi_{\rm le}\) gives
		\(v+|X_Av|
		\le C(1+r)e^{-r^2/(8t)}\).
		For \(\eta\ge a_0/2\), its Gaussian upper bound
		and \(\sqrt t\le2r/a_0\) give the same estimate.
		This proves \eqref{cont:eq:upper-conjugate}.
		Finally,
		\(u=e^{-r}v\) gives \eqref{cont:eq:upper-solution}.
	\end{proof}

	We next use a radial subsolution to connect the positive lower bounds at the fixed-front and diffusive scales.
	
	\begin{lemma}[Linear lower bound through the diffusive scale]
		\label{cont:lem:lower-leading-edge}
		There exist \(L_*,C>0\) such that, for every fixed \(K>0\),
		there are \(T_K,c_K>0\) for which
		\begin{equation}\label{cont:eq:linear-leading-edge}
			c_K r
			\le v(t,r,\omega)
			\le C(1+r)
		\end{equation}
		whenever \(t\ge T_K\) and
		\(L_*\le r\le K\sqrt t\).
	\end{lemma}
	
	\begin{proof}
		The equation for \(v\) is
		\begin{equation}\label{cont:eq:auxiliary-conjugate-pde}
			\begin{split}
				\partial_t v
				&=\partial_{rr}v
				+b_{\rm rad}(\partial_r v-v)
				+\rho^{-2}\DS v-u(t,\rho\omega)v,\\
				b_{\rm rad}(t,r)
				&=\frac{N-1}{R_0(t)+r}-\frac{k}{t},
				\qquad \rho=R_0(t)+r .
			\end{split}
		\end{equation}
		For all sufficiently large \(t\),
		\(b_{\rm rad}(t,r)<0\) for every \(r\ge0\).
		By \eqref{cont:eq:upper-solution}, after increasing the starting
		time if necessary,
		\(u(t,(R_0(t)+r)\omega)\le b_{\rm low}(r)\), where
		\(b_{\rm low}(r):=C(1+r)e^{-r}\).

		With these bounds, we construct a radial subsolution
		comparable to \(r\).
		Choose \(L_*\ge3\) so large that
		\(\int_{L_*}^{\infty}y b_{\rm low}(y)\,dy\le\log2\),
		and let \(\psi_{\rm le}\) solve
		\[
		\psi_{\rm le}''=b_{\rm low}\psi_{\rm le},
		\qquad
		\psi_{\rm le}(L_*)=L_*,
		\qquad
		\psi_{\rm le}'(L_*)=1.
		\]
		Since \(b_{\rm low}\ge0\), we have
		\(\psi_{\rm le}'(r)\ge1\) and \(\psi_{\rm le}(r)\ge r\).
		Moreover,
		\((r\psi_{\rm le}'-\psi_{\rm le})'
		=r\psi_{\rm le}''\ge0\), with
		\(r\psi_{\rm le}'-\psi_{\rm le}=0\) at \(r=L_*\).
		Thus \(\psi_{\rm le}\le r\psi_{\rm le}'\), and
		\[
		\psi_{\rm le}'(r)
		\le1+\int_{L_*}^r
		y b_{\rm low}(y)\psi_{\rm le}'(y)\,dy.
		\]
		Gronwall's inequality gives
		\(\psi_{\rm le}'\le2\), and hence
		\[
		r\le\psi_{\rm le}(r)\le2r,\qquad
		1\le\psi_{\rm le}'(r)\le2
		\quad(r\ge L_*).
		\]
		Since \(L_*\ge3\), we also have
		\(\psi_{\rm le}'-\psi_{\rm le}<0\).
		For any \(c_{\rm sub}>0\), the function
		\(c_{\rm sub}\psi_{\rm le}\) is independent of \(t\) and \(\omega\).
		Using \(\psi_{\rm le}''=b_{\rm low}\psi_{\rm le}\), its residual in
		\eqref{cont:eq:auxiliary-conjugate-pde} is
		\[
		c_{\rm sub}
		\bigl[
		(u-b_{\rm low})\psi_{\rm le}
		-b_{\rm rad}(\psi_{\rm le}'-\psi_{\rm le})
		\bigr]
		\le0.
		\]
		Thus \(c_{\rm sub}\psi_{\rm le}\) is a subsolution.
		
		At the two lateral boundaries,
		\eqref{eq:rrm} and \eqref{cont:eq:rrm-diffusive}
		give, uniformly in \(\omega\),
		\[
		\begin{aligned}
			v(t,L_*,\omega)
			&\longrightarrow
			e^{L_*}U_*(L_*+s^\infty(\omega)),\\
			t^{-1/2}v(t,K\sqrt t,\omega)
			&\longrightarrow
			\mathfrak c(\omega)K e^{-K^2/4}.
		\end{aligned}
		\]
		Hence
		$v(t,L_*,\omega)$ and
		$t^{-1/2}v(t,K\sqrt t,\omega)$ have uniform positive lower bounds for all sufficiently large
		$t$ and all $\omega$.
		Choose $T_K$ so that these bounds hold for
		$t\ge T_K$ and $K\sqrt{T_K}>L_*$.
		Taking \(c_{\rm sub}>0\) sufficiently small,
		we may arrange \(c_{\rm sub}\psi_{\rm le}\le v\)
		on both lateral boundaries and on the initial interval
		\(L_*\le r\le K\sqrt{T_K}\).
		The comparison principle on the moving cylinder
		\(L_*<r<K\sqrt t\), \(t\ge T_K\), then yields
		\(v(t,r,\omega)\ge c_{\rm sub}\psi_{\rm le}(r)
		\ge c_{\rm sub}r\).
		The upper bound in \eqref{cont:eq:linear-leading-edge}
		follows from Lemma~\ref{cont:lem:upper-leading-edge}.
	\end{proof}

	\subsection{Angular and mixed relative estimates}
	
	The bounds above provide the required nondegeneracy.
	They allow us to pass from estimates for \(v\) to relative estimates
	for \(q=\log u\).
	
	\begin{lemma}[Angular and mixed relative estimates]
		\label{cont:lem:angular-relative}
		For every fixed \(K>0\), there are \(T_K,C_K>0\) such that for \(t \ge T_K\),
		\begin{equation}\label{cont:eq:relative-angular}
			|X_Aq|\le C_K,\qquad
			|\partial_rX_Aq|\le\frac{C_K}{r},
			\qquad
			\HS q\le
			C_K\left(1+\frac{\log t}{r}\right)\gS
		\end{equation}
		on \(2L_*\le r\le K\sqrt t\).
		For every physical tangent vector \(e\perp\omega\),
		\begin{equation}\label{cont:eq:relative-mixed}
			|D_x^2q[\omega,e]|
			\le\frac{C_K}{tr}|e|.
		\end{equation}
	\end{lemma}
	
	\begin{proof}
		Since $q=-r+\log v$,
		\eqref{cont:eq:upper-conjugate} and
		\eqref{cont:eq:linear-leading-edge} give
		$|X_Aq|=|X_Av|/v\le C_K$.
		Fix $(t_0,r_0,\omega_0)$ with
		$2L_*\le r_0\le K\sqrt{t_0}$, and set
		$\rho_0=R_0(t_0)+r_0$.
		Choose $\theta_K>0$ sufficiently small and introduce
		\[
		t=t_0+r_0^2s,\qquad
		r=r_0(1+y_1),\qquad
		\omega=
		\exp_{\omega_0}\!\left(\frac{r_0}{\rho_0}y'\right)
		\]
		on the fixed cylinder
		$-\theta_K<s\le0$, $|y_1|\le1/2$, $|y'|\le1$.
		Then $t\asymp t_0$,
		$r_0/2\le r\le3r_0/2$, and
		$\rho\asymp\rho_0\asymp t_0$.
		Since $r\ge r_0/2\ge L_*$, choosing
		$\theta_K$ sufficiently small and $T_K$
		sufficiently large ensures that every cylinder
		lies in the region where
		Lemma~\ref{cont:lem:lower-leading-edge}
		applies with $2K$ in place of $K$.

		Set \(V(s,y)=r_0^{-1}v(t,r,\omega)\). Then
		\(v_t=r_0^{-1}V_s\), \(v_r=V_{y_1}\), and
		\(v_{rr}=r_0^{-1}V_{y_1y_1}\). Under the angular rescaling,
		\(\Delta_Sv=(\rho_0^2/r_0)\Delta_{g_{t_0,r_0}}V\).
		Thus, after multiplying \eqref{cont:eq:auxiliary-conjugate-pde}
		by \(r_0\), we obtain
		\[
		\begin{aligned}
			V_s
			={}&V_{y_1y_1}
			+\frac{\rho_0^2}{\rho^2}\Delta_{g_{t_0,r_0}}V
			+r_0b_{\rm rad}V_{y_1}\\
			&\quad
			-r_0^2b_{\rm rad}V
			-r_0^2u(t,\rho\omega)V .
		\end{aligned}
		\]
		On the cylinder, $|b_{\rm rad}|\le C/t_0$.
		Since $r_0\le K\sqrt{t_0}$, we have
		$r_0|b_{\rm rad}|+r_0^2|b_{\rm rad}|\le C_K$
		and $c_K\le\rho_0^2/\rho^2\le C_K$.
		The rescaled spherical metrics \(g_{t_0,r_0}\) have uniformly
		controlled local coefficients. Moreover, since \(r\ge r_0/2\),
		\eqref{cont:eq:upper-solution} gives
		\(r_0^2u(t,\rho\omega)
		\le Cr_0^2(1+r_0)e^{-r_0/2}\le C_K\).
		Hence the rescaled equations are uniformly parabolic, with uniformly
		bounded lower-order coefficients.

		Let \(p=X_Av\). Since \(X_A\) commutes with \(\partial_r\) and
		\(\Delta_S\), and the coefficients of the linear terms are independent of \(\omega\),
		\[
		p_t=p_{rr}+b_{\rm rad}(p_r-p)+\rho^{-2}\Delta_Sp-2up.
		\]
		Thus \(P=r_0^{-1}p\) satisfies the same rescaled equation as \(V\),
		with \(r_0^2u\) replaced by \(2r_0^2u\). By
		Lemma~\ref{cont:lem:upper-leading-edge}, \(|V|+|P|\le C_K\) on the
		fixed cylinder. Interior parabolic estimates on a smaller concentric
		cylinder therefore give
		\(|\partial_{y_1}V|+|\partial_{y_1}P|\le C_K.\)
		Since \(\partial_rv=\partial_{y_1}V\) and
		\(\partial_rX_Av=\partial_{y_1}P\), we obtain
		\begin{equation}\label{cont:eq:leading-radial-gradient}
			|\partial_r v|+|\partial_r X_Av|\le C_K.
		\end{equation}
		Using \(v\ge c_Kr\), we obtain
		\[
		\left|\partial_rX_Aq\right|
		\le
		\frac{|\partial_rX_Av|}{v}
		+\frac{|X_Av|\,|\partial_rv|}{v^2}
		\le\frac{C_K}{r}.
		\]
		
		For a unit spherical tangent vector \(e\), take
		\(A=e\otimes\omega-\omega\otimes e\). At \(\omega\),
		\(X_A^2v=\HS v[e,e]\). Hence
		\eqref{cont:eq:second-rotation-input} and \(v\ge c_Kr\) give
		\[
		\begin{aligned}
			\HS\log v[e,e]
			&=
			\frac{\HS v[e,e]}{v}
			-\frac{(d_Sv[e])^2}{v^2}\\
			&\le
			C_K\frac{1+r+\log t}{r}
			\le
			C_K\left(1+\frac{\log t}{r}\right).
		\end{aligned}
		\]
		Since \(q=-r+\log v\), this proves the spherical Hessian bound in
		\eqref{cont:eq:relative-angular}.
		
		Finally, the polar Hessian identity
		\eqref{eq:euclidean-polar-hessian} and
		\eqref{cont:eq:relative-angular}, together with
		$\rho\asymp t$ and $r\le K\sqrt t\le t$
		for large $t$, give
		\[
		|D_x^2q[\omega,e]|
		\le C_K\left(\frac1{\rho r}
		+\frac1{\rho^2}\right)|e|
		\le\frac{C_K}{tr}|e|,
		\]
		which proves \eqref{cont:eq:relative-mixed}.
	\end{proof}

	\subsection{Radial logarithmic derivative estimates}

	We now turn to the radial estimates. The first-order
	bound follows from the estimates above, while
	the second-order bound requires a blow-up argument.

	\begin{lemma}[Radial logarithmic concavity]
		\label{cont:lem:radial-curvature}
		For every fixed \(K>0\), there are \(L_K,T_K>0\)
		such that, for \(t\ge T_K\) and
		\(L_K\le r\le K\sqrt t\),
		\begin{equation}\label{cont:eq:radial-curvature}
			q_\rho\le-\frac12,\qquad
			q_{\rho\rho}\le-\frac1{2r^2}
			.
		\end{equation}
	\end{lemma}
	
	\begin{proof}
		Since \(q=-r+\log v\), combining \eqref{cont:eq:leading-radial-gradient}
		with the lower bound in \eqref{cont:eq:linear-leading-edge}
		gives
		\[
		q_\rho=-1+\frac{v_r}{v}
		\le -1+\frac{C_K}{c_Kr}.
		\]
		Choose \(L_K\ge 2L_*\) sufficiently large that
		\(C_K/(c_KL_K)\le 1/2\).
		Then \(q_\rho\le -1/2\).

		We now prove the second estimate in
		\eqref{cont:eq:radial-curvature} by contradiction.
		If it failed for every sufficiently large choice
		of \(L_K,T_K\), there would be
		\(t_n\to\infty\), \(r_n\to\infty\),
		\(r_n\le K\sqrt{t_n}\), and \(\omega_n\in\Sph\)
		such that
		\[
		r_n^2q_{\rho\rho}(t_n,r_n,\omega_n)>-\frac12.
		\]
		Passing to a subsequence, assume
		\(r_n/\sqrt{t_n}\to\eta_*\in[0,K]\).
		
		\emph{Subdiffusive case \(\eta_*=0\).}
		Set
		\[
		t=t_n+r_n^2\upsilon,\qquad
		r=r_nz_1,\qquad
		\omega=
		\exp_{\omega_n}\!\left(
		\frac{r_n}{R_0(t_n)}z'
		\right),
		\]
		and define \(v_n(\upsilon,z_1,z')=r_n^{-1}v(t,r,\omega)\).
		If \(g_n\) denotes the rescaled round metric in the \(z'\)-chart,
		then \eqref{cont:eq:auxiliary-conjugate-pde} becomes
		\begin{equation}\label{cont:eq:blowup-equation}
			\begin{aligned}
				\partial_\upsilon v_n
				={}&\partial_{z_1z_1}v_n
				+\frac{R_0(t_n)^2}{\rho^2}\Delta_{g_n}v_n\\
				&+r_n\left(\frac{N-1}{\rho}-\frac{k}{t}\right)
				\partial_{z_1}v_n
				-r_n^2\left(\frac{N-1}{\rho}-\frac{k}{t}\right)v_n\\
				&-r_n^3e^{-r_nz_1}v_n^2,
				\qquad
				\rho=R_0(t)+r .
			\end{aligned}
		\end{equation}
		On every compact subset of
		$\mathbb R_\upsilon\times(0,\infty)_{z_1}
		\times\mathbb R^{N-1}_{z'}$,
		the relation $r_n^2/t_n\to0$ gives
		$t/t_n\to1$. Since $r=r_nz_1=o(t_n)$,
		we also have $\rho/R_0(t_n)\to1$.
		Moreover, \(r_n/R_0(t_n)\to0\), so the rescaled spherical metrics
		\(g_n\) converge locally smoothly to the Euclidean metric on
		\(\mathbb R^{N-1}\). Consequently,
		\[
		\frac{R_0(t_n)^2}{\rho^2}\to1,\qquad
		r_n\left(\frac{N-1}{\rho}-\frac{k}{t}\right)\to0,\qquad
		r_n^2\left(\frac{N-1}{\rho}-\frac{k}{t}\right)\to0,
		\]
		and, on compact subsets with \(z_1>0\),
		\(r_n^3e^{-r_nz_1}\to0\).

		Lemmas~\ref{cont:lem:upper-leading-edge} and
		\ref{cont:lem:lower-leading-edge} give
		\(cz_1\le v_n\le C(z_1+r_n^{-1})\) locally for \(z_1>0\).
		The equations \eqref{cont:eq:blowup-equation} are uniformly
		parabolic on compact subsets of \(z_1>0\), so standard interior
		estimates give local Hölder compactness. Passing to a subsequence,
		we obtain
		\(v_n\to v_\infty\) locally uniformly,
		with \(v_\infty\) satisfying
		\begin{equation}\label{cont:eq:subdiffusive-limit}
			\partial_\upsilon v_\infty
			=
			\partial_{z_1z_1}v_\infty+\Delta_{z'}v_\infty,
			\qquad z_1>0,\quad \upsilon\in\mathbb R.
		\end{equation}
		The full time axis appears because \(t_n/r_n^2\to\infty\).
		The upper bound above also gives
		\(v_\infty(\upsilon,z_1,z')\le Cz_1\), so \(v_\infty\)
		extends continuously by zero to \(z_1=0\).
		
		By the zero-Dirichlet reflection principle, the odd extension of
		\(v_\infty\) across \(z_1=0\) is an entire caloric function with at
		most linear spatial growth. At a fixed point, the interior estimate
		on a parabolic cylinder \(Q_R\) gives
		\[
		|D_z^2v_\infty|+|\partial_\upsilon v_\infty|
		\le
		\frac{C}{R^2}\|v_\infty\|_{L^\infty(Q_R)}
		\le\frac{C}{R}.
		\]
		Letting \(R\to\infty\) yields
		\(v_\infty=c_{\rm lin}z_1\), with \(c_{\rm lin}>0\).

		To transfer this rigidity to the logarithmic curvature,
		we upgrade the convergence to spatial \(C^2_{\rm loc}\).
		Let \(Q'\Subset Q\Subset
		\mathbb R\times(0,\infty)\times\mathbb R^{N-1}\).
		Standard interior \(W^{2,1}_p\) estimates, parabolic Sobolev
		embedding, and Schauder estimates applied to
		\eqref{cont:eq:blowup-equation} give, for some
		\(\alpha\in(0,1)\),
		\begin{equation}\label{cont:eq:blowup-regularity}
			\|v_n\|_{C^{2+\alpha,\,1+\alpha/2}(Q')}
			\le C_{Q'}.
		\end{equation}
		Since the locally uniform limit is already identified,
		\(v_n\to v_\infty\) in spatial \(C^2_{\rm loc}\). Hence, at
		\((\upsilon,z_1,z')=(0,1,0)\),
		\[
		r_n^2q_{\rho\rho}
		=
		\frac{\partial_{z_1z_1}v_n}{v_n}
		-\left(\frac{\partial_{z_1}v_n}{v_n}\right)^2
		\longrightarrow-1,
		\]
		contradicting the choice of the sequence.
		
		\emph{Diffusive case \(\eta_*>0\).}
		Set
		\[
		t=t_n(1+\upsilon),\qquad
		r=\sqrt{t_n}\,z_1,\qquad
		\omega=
		\exp_{\omega_n}\!\left(
		\frac{\sqrt{t_n}}{R_0(t_n)}z'
		\right),
		\]
		and define
		\(v_n(\upsilon,z_1,z')=t_n^{-1/2}v(t,r,\omega)\).
		If \(g_n\) denotes the rescaled round metric in the \(z'\)-chart,
		then \eqref{cont:eq:auxiliary-conjugate-pde} becomes
		\begin{equation}\label{cont:eq:diffusive-blowup-equation}
			\begin{aligned}
				\partial_\upsilon v_n
				={}&\partial_{z_1z_1}v_n
				+\frac{R_0(t_n)^2}{\rho^2}\Delta_{g_n}v_n\\
				&+\sqrt{t_n}
				\left(\frac{N-1}{\rho}-\frac{k}{t}\right)
				\partial_{z_1}v_n\\
				&-t_n
				\left(\frac{N-1}{\rho}-\frac{k}{t}\right)v_n
				-t_n^{3/2}e^{-\sqrt{t_n}z_1}v_n^2,
				\qquad
				\rho=R_0(t)+\sqrt{t_n}z_1 .
			\end{aligned}
		\end{equation}
		Passing to a further subsequence, we may assume
		\(\omega_n\to\omega_*\in\Sph\).
		On compact subsets of \(z_1>0\) and \(1+\upsilon>0\), set
		\(\eta=z_1/\sqrt{1+\upsilon}\).
		Since \(\sqrt{t_n}/R_0(t_n)\to0\), the rescaled angular points converge
		to \(\omega_*\) uniformly for \(z'\) in compact sets. Moreover,
		\(v_n=\sqrt{1+\upsilon}\,t^{-1/2}v(t,\sqrt t\,\eta,\omega)\).
		Hence \eqref{cont:eq:rrm-diffusive} gives
		\begin{equation}\label{cont:eq:diffusive-profile}
			v_n(\upsilon,z_1,z')
			\longrightarrow
			v_\infty(\upsilon,z_1)
			:=
			\mathfrak c(\omega_*)z_1
			\exp\!\left(
			-\frac{z_1^2}{4(1+\upsilon)}
			\right)
		\end{equation}
		locally uniformly for \(z_1>0\) and \(1+\upsilon>0\).

		To recover second derivatives from this \(C^0\) limit,
		we use the rescaled equation.
		On such compact sets,
		\(R_0(t_n)^2/\rho^2\to(1+\upsilon)^{-2}\),
		\[
		\sqrt{t_n}
		\left(\frac{N-1}{\rho}-\frac{k}{t}\right)\to0,
		\qquad
		-t_n
		\left(\frac{N-1}{\rho}-\frac{k}{t}\right)
		\to\frac{3}{2(1+\upsilon)},
		\]
		and \(t_n^{3/2}e^{-\sqrt{t_n}z_1}\to0\), while \(g_n\)
		converges locally smoothly to the Euclidean metric.
		Thus \eqref{cont:eq:diffusive-blowup-equation} is uniformly
		parabolic on compact subsets away from \(z_1=0\), with uniformly
		controlled coefficients. The same interior estimates as in
		\eqref{cont:eq:blowup-regularity} upgrade
		\eqref{cont:eq:diffusive-profile} to
		\(v_n\to v_\infty\) in spatial \(C^2_{\rm loc}\).

		Put $z_{1,n}=r_n/\sqrt{t_n}\to\eta_*$.
		Since $q=-r+\log v$, the spatial
		$C^2_{\rm loc}$ convergence and positivity of
		$v_\infty$ yield
		\[
		\begin{aligned}
			r_n^2q_{\rho\rho}(t_n,r_n,\omega_n)
			&=z_{1,n}^2
			(\log v_n)_{z_1z_1}(0,z_{1,n},0)\\
			&\longrightarrow
			\eta_*^2(\log v_\infty)_{z_1z_1}(0,\eta_*)\\
			&=-1-\frac{\eta_*^2}{2},
		\end{aligned}
		\]
		again contradicting the choice of the sequence.
		This proves \eqref{cont:eq:radial-curvature}.
	\end{proof}

	Finally, we combine the radial, angular, and mixed
	estimates to obtain the  logarithmic-Hessian bound through the diffusive scale.

	\begin{proof}[Proof of Proposition~\ref{cont:prop:through-diffusive}]
		Since
		\(\rho=R_0(t)+r=2t+O_K(\sqrt t+\log t)\), we have
		\(\rho\asymp t\) uniformly in the region under consideration.
		For \(e\perp\omega\), the polar Hessian formula
		\eqref{eq:euclidean-polar-hessian} and the preceding lemmas give
		\[
		D_x^2q[e,e]
		=
		\rho^{-1}q_\rho|e|^2+\rho^{-2}\HS q[e,e]
		\le
		-\frac{c}{t}|e|^2
		+\frac{C_K\log t}{t^2}|e|^2.
		\]
		After increasing \(T_K\), this yields
		\[
		D_x^2q[e,e]\le-\frac{c_{\rm tan}}{t}|e|^2
		\]
		for some \(c_{\rm tan}>0\).
		
		For a general vector \(b\omega+e\), with \(e\perp\omega\),
		the radial and mixed estimates give
		\[
		\begin{aligned}
			D_x^2q[b\omega+e,b\omega+e]
			&\le
			-\frac{b^2}{2r^2}
			+\frac{2C_K}{tr}|b||e|
			-\frac{c_{\rm tan}}{t}|e|^2\\
			&\le
			-\frac{b^2}{4r^2}
			-\left(
			\frac{c_{\rm tan}}{t}
			-\frac{4C_K^2}{t^2}
			\right)|e|^2.
		\end{aligned}
		\]
		Increasing \(T_K\) once more, we have
		\(c_{\rm tan}/t - 4C_K^2/t^2 \ge c_{\rm tan}/(2t)\).
		Since \(r^2\le K^2t\),
		\[
		D_x^2q[b\omega+e,b\omega+e]
		\le
		-\frac1t
		\min\left\{\frac1{4K^2},\frac{c_{\rm tan}}2\right\}
		\left(b^2+|e|^2\right),
		\]
		which proves \eqref{cont:eq:through-diffusive-hessian}.
	\end{proof}

	\section{Amplitude estimates near the critical speed}
	\label{cont:sec:near-critical}
	
	In Sections~\ref{cont:sec:near-critical} and \ref{cont:sec:fixed-gap}, we develop amplitude estimates in the near-critical and fixed-gap regimes, respectively. These estimates will be used in Section~\ref{cont:sec:exterior} to establish the logarithmic-Hessian bound in the Gaussian exterior.
	For the analysis in these sections and the Gaussian exterior,
	we use the elapsed time
	\begin{equation}\label{cont:eq:elapsed-time}
		\sigma=t-1,
	\end{equation}
	where \(t\) denotes the original physical time, and continue to write
	\(u(\sigma,x)\) for the time-shifted solution \(u(1+\sigma,x)\).
	Thus
	\begin{equation}\label{cont:eq:reset-kpp}
		u_\sigma=\Delta u+u-u^2,\qquad
		\sigma>0,\qquad u(0,x)=u_0(x).
	\end{equation}

	\subsection{Weighted masses and the limiting amplitude}
	
	Set
	\begin{equation}\label{cont:eq:velocity}
		p=\frac{x}{2\sigma},\qquad x=2\sigma p,\qquad
		\chi(p)=|p|-1,\qquad \chi_0=\frac18.
	\end{equation}
	For fixed \(p\), write \(\chi=\chi(p)\).
	Thus \(|p|=1\) corresponds to the critical physical speed \(2\).
	
	Let
	\begin{equation}\label{cont:eq:heat-kernel}
		\mathcal K_\sigma(x)
		=(4\pi\sigma)^{-N/2}e^{-|x|^2/(4\sigma)},
		\qquad \sigma>0.
	\end{equation}
	By comparison,
	\begin{equation}\label{cont:eq:heat-comparison}
		0\le u(\sigma,x)\le
		\min\left\{K_0,\,
		e^\sigma\int\mathcal K_\sigma(x-y)u_0(y)\,dy\right\},
		\qquad \sigma>0.
	\end{equation}
	
	For \(p\in\R^N\), define the initial weighted mass, the evolving
	weighted mass, and the weighted absorption rate by
	\begin{align}
		\mathfrak m_0(p)
		&=\int e^{p\cdot y}u_0(y)\,dy,
		\label{cont:eq:initial-mass}\\
		\mathfrak m_\sigma(p)
		&=e^{-(1+|p|^2)\sigma}
		\int e^{p\cdot y}u(\sigma,y)\,dy,
		\label{cont:eq:tilted-mass}\\
		\mathfrak n_\sigma(p)
		&=e^{-(1+|p|^2)\sigma}
		\int e^{p\cdot y}u(\sigma,y)^2\,dy.
		\label{cont:eq:absorption}
	\end{align}
	Integration by parts in \eqref{cont:eq:reset-kpp} gives
	\begin{equation}\label{cont:eq:mass-derivative}
		\partial_\sigma\mathfrak m_\sigma(p)
		=-\mathfrak n_\sigma(p).
	\end{equation}
	The identity follows by testing against \(e^{p\cdot y}\varphi_L(y)\) and letting \(L\to\infty\); the Gaussian bound \eqref{cont:eq:heat-comparison} makes the cutoff errors vanish.
	Hence
	\[
	0\le \mathfrak m_\sigma(p)\le \mathfrak m_0(p).
	\]
	
	For \(|p|>1\), define the limiting amplitude by
	\begin{equation}\label{cont:eq:limiting-amplitude}
		\mathfrak a(p)
		:=\lim_{\sigma\to\infty}\mathfrak m_\sigma(p)
		=\mathfrak m_0(p)-\int_0^\infty\mathfrak n_a(p)\,da
		\ge0.
	\end{equation}
	The one-dimensional exponential-moment framework for amplitudes
	ahead of the front is developed in \cite{Brunet2023}.
	The estimates below establish the uniform near-critical bounds
	needed here in multiple dimensions.

	\subsection{Near-critical moments and amplitude bounds}
	
	Throughout the moment and point-amplitude estimates, write
	\(d_N=(N-1)/2\), so that \(k=d_N+3/2\).
	The centering at \(2\sigma p\) is natural both for the Gaussian
	observation point and for differentiation with respect to \(p\).
	For \(j=0,1,2\), define
	\begin{equation}\label{cont:eq:centered-moments}
		\mathfrak n_{\sigma,j}(p)
		:=e^{-(1+|p|^2)\sigma}
		\int e^{p\cdot y}|y-2\sigma p|^j
		u(\sigma,y)^2\,dy .
	\end{equation}
	Thus \(\mathfrak n_{\sigma,0}=\mathfrak n_\sigma\).
	
	\begin{lemma}[Near-critical weighted moments]
		\label{cont:lem:near-moments}
		There exist \(\sigma_0\ge1\) and \(C>0\) such that, for
		\(0<\chi(p)\le\chi_0\) and \(\sigma\ge\sigma_0\),
		\begin{align}
			\mathfrak n_\sigma(p)
			&\le C\sigma^{-3/2}e^{-\chi^2\sigma},
			\label{cont:eq:near-zeroth-moment}\\
			\mathfrak n_{\sigma,2}(p)
			&\le C(\chi^2\sigma^2+\sigma+\log^2\sigma)
			\sigma^{-3/2}e^{-\chi^2\sigma},
			\label{cont:eq:near-second-moment}\\
			|\nabla_p\mathfrak n_\sigma(p)|
			&\le C(1+\log\sigma+\chi\sigma)
			\sigma^{-3/2}e^{-\chi^2\sigma}.
			\label{cont:eq:near-first-moment}
		\end{align}
		The constants are uniform in
		\(0<\chi(p)\le\chi_0\) and \(p/|p|\in\Sph\).
		On \(0\le\sigma\le\sigma_0\), all finite-time quantities appearing
		above are uniformly bounded for \(1\le|p|\le1+\chi_0\).
	\end{lemma}
	
	\begin{proof} 
		After increasing \(\sigma_0\) if necessary, we may assume that 
		\(1+\sigma\ge T_{\rm le}\) for every \(\sigma\ge\sigma_0\). 
		Put 
		\(\rho=|y|\) and \(r=\rho-R_0(1+\sigma)\),
		so that \(r\) is the original physical front coordinate at time 
		\(1+\sigma\). Lemma~\ref{cont:lem:upper-leading-edge} and \(u\le K_0\) 
		give 
		\begin{equation}\label{cont:eq:squared-envelope} 
			\begin{split} 
				u(\sigma,\rho\omega')^2 
				&\le f_{\rm env}(r),\\ 
				f_{\rm env}(r) 
				&= 
				\begin{cases} 
					K_0^2, & r<0,\\ 
					C(1+r)^2e^{-2r}, & r\ge0. 
				\end{cases} 
			\end{split} 
		\end{equation} 
		
		For \(j=0,1\), the elementary spherical integrals satisfy 
		\begin{equation}\label{cont:eq:spherical-moments} 
			\int_{\Sph} 
			e^{b_{\rm sph}\omega\cdot\omega'} 
			(1-\omega\cdot\omega')^j\,d\omega' 
			\le 
			C_j e^{b_{\rm sph}} 
			(1+b_{\rm sph})^{-d_N-j}, 
			\qquad b_{\rm sph}\ge0. 
		\end{equation} 
		Indeed, setting \(\zeta=1-\omega\cdot\omega'\in[0,2]\),
		and splitting into a neighborhood of \(\zeta=0\) and its 
		complement gives \eqref{cont:eq:spherical-moments}. 
		
		We shall also use the following radial moment bound. 
		For \(1\le |p|\le1+\chi_0<2\) and each fixed 
		\(\ell=0,1,2\), 
		\begin{equation}\label{cont:eq:radial-moment-bound} 
			\int_{-R_0(1+\sigma)}^\infty 
			(1+R_0(1+\sigma)+r)^{d_N} 
			|r|^\ell e^{|p|r} 
			f_{\rm env}(r)\,dr 
			\le C_\ell(1+\sigma)^{d_N}. 
		\end{equation} 
		The negative-\(r\) part is exponentially integrable because 
		\(e^{|p|r}\) decays as \(r\to-\infty\), while on \(r\ge0\) 
		the envelope in \eqref{cont:eq:squared-envelope} leaves the fixed 
		exponential margin \(2-|p|\ge1-\chi_0>0\).
		
		We first prove the zeroth-moment estimate \eqref{cont:eq:near-zeroth-moment}.
		Write $p=|p|\omega$ and $y=\rho\omega'$.
		Using \eqref{cont:eq:spherical-moments} with $j=0$,
		the inequality
		$\rho^{N-1}(1+|p|\rho)^{-d_N}
		\le C(1+\rho)^{d_N}$,
		and \eqref{cont:eq:radial-moment-bound} with $\ell=0$,
		we obtain
		\[
		\mathfrak n_\sigma(p)
		\le C(1+\sigma)^{d_N}
		e^{|p|R_0(1+\sigma)-(1+|p|^2)\sigma}.
		\]
		Since $|p|=1+\chi$, the exact front factor is
		\[
		|p|R_0(1+\sigma)-(1+|p|^2)\sigma
		=2|p|-k|p|\log(1+\sigma)-\chi^2\sigma.
		\]
		Consequently, using \(k=d_N+3/2\), we obtain
		\[
		\mathfrak n_\sigma(p)
		\le C(1+\sigma)^{-3/2-k\chi}
		e^{-\chi^2\sigma}
		\le C\sigma^{-3/2}e^{-\chi^2\sigma}.
		\]
		Here \(e^{2|p|}\) is uniformly bounded on the near-critical 
		annulus. This proves \eqref{cont:eq:near-zeroth-moment}.

		We next estimate the second centered moment.
		Again write $p=|p|\omega$ and $y=\rho\omega'$.
		Then
		\[
		\begin{aligned}
			|y-2\sigma p|^2
			&=(\rho-2|p|\sigma)^2
			+4|p|\sigma\rho(1-\omega\cdot\omega'),\\
			\rho-2|p|\sigma
			&=r+2-k\log(1+\sigma)-2\chi\sigma.
		\end{aligned}
		\]
		For the radial part, the bound
		$(\rho-2|p|\sigma)^2
		\le C(r^2+1+\log^2(1+\sigma)+\chi^2\sigma^2)$,
		together with
		\eqref{cont:eq:spherical-moments} for $j=0$
		and \eqref{cont:eq:radial-moment-bound} for
		$\ell=0,2$, gives a contribution bounded by
		$C(1+\log^2(1+\sigma)+\chi^2\sigma^2)
		\sigma^{-3/2}e^{-\chi^2\sigma}$.
		For the angular part \(4|p|\sigma\rho(1-\omega\cdot\omega')\), we use
		\eqref{cont:eq:spherical-moments} with $j=1$
		and \eqref{cont:eq:radial-moment-bound} with
		$\ell=0$. The extra factor
		$(1+|p|\rho)^{-1}$ compensates the factor
		$|p|\rho$, so this contribution is bounded by
		$C\sigma\,\sigma^{-3/2}e^{-\chi^2\sigma}$.
		Combining both contributions yields
		\[
		\mathfrak n_{\sigma,2}(p)
		\le C(\chi^2\sigma^2+\sigma+\log^2\sigma)
		\sigma^{-3/2}e^{-\chi^2\sigma},
		\]
		after increasing $\sigma_0$ if necessary.
		This proves \eqref{cont:eq:near-second-moment}.

		It remains to prove the gradient estimate
		\eqref{cont:eq:near-first-moment}.
		Differentiating $\mathfrak n_\sigma$ with respect to $p$ gives
		\begin{equation}\label{cont:eq:gradient-absorption-identity}
			\nabla_p\mathfrak n_\sigma(p)
			=
			e^{-(1+|p|^2)\sigma}
			\int (y-2\sigma p)e^{p\cdot y}u(\sigma,y)^2\,dy.
		\end{equation}
		For the normal component, use
		$\omega\cdot(y-2\sigma p)
		=(\rho-2|p|\sigma)-\rho(1-\omega\cdot\omega')$.
		The first term is controlled by
		\eqref{cont:eq:radial-moment-bound} with $\ell=0,1$ and \eqref{cont:eq:spherical-moments}
		with $j=0$,
		while the second is controlled by
		\eqref{cont:eq:radial-moment-bound} with $\ell=0$ and \eqref{cont:eq:spherical-moments}
		with $j=1$. Hence
		\begin{equation}\label{cont:eq:normal-first-moment}
			|\omega\cdot\nabla_p\mathfrak n_\sigma(p)|
			\le C(1+\log\sigma+\chi\sigma)
			\sigma^{-3/2}e^{-\chi^2\sigma}.
		\end{equation}

		We now turn to the transverse components. A direct 
		absolute-value estimate of the transverse first moment would incur 
		an unnecessary \(O(\sqrt\sigma)\) loss, so we instead use the 
		rotational cancellation. Let \(e\perp\omega\) be a unit vector and
		set \(A=e\otimes\omega-\omega\otimes e\). 
		Since \(X_A\) is divergence-free and 
		\(X_A(e^{p\cdot y})=-|p|(y\cdot e)e^{p\cdot y}\),
		integration by parts gives 
		\begin{equation}\label{cont:eq:transverse-rotation-identity} 
			|p| 
			\int 
			(y\cdot e)e^{p\cdot y} 
			u(\sigma,y)^2\,dy 
			= 
			\int 
			e^{p\cdot y}X_A(u(\sigma,y)^2)\,dy. 
		\end{equation} 
		For $\rho\ge2R_*$, the reflection estimate yields
		$|X_A(u^2)|\le C(-\partial_\rho u^2)$.
		Integrating by parts in $\rho$ gives
		\[
		\left|\int e^{p\cdot y}X_A(u^2)\,dy\right|
		\le C\int e^{p\cdot y}u^2\,dy+C.
		\]
		Indeed, the logarithmic derivative of the radial weight
		is $|p|\omega\cdot\omega'+(N-1)/\rho$, which is uniformly
		bounded on $\rho\ge2R_*$.
		The boundary term at infinity vanishes by Gaussian decay,
		while the boundary term at $\rho=2R_*$ and the contribution
		from the fixed inner ball are uniformly bounded.
		Consequently, \eqref{cont:eq:transverse-rotation-identity}
		and \eqref{cont:eq:near-zeroth-moment} give
		\[
		|e\cdot\nabla_p\mathfrak n_\sigma(p)|
		\le C\mathfrak n_\sigma(p)
		+Ce^{-(1+|p|^2)\sigma}
		\le C\sigma^{-3/2}e^{-\chi^2\sigma},
		\]
		where we used $|p|\ge1$ and
		$1+|p|^2-\chi^2=2+2\chi$.
		Together with \eqref{cont:eq:normal-first-moment},
		this proves \eqref{cont:eq:near-first-moment}.

		Finally, on the fixed time interval 
		\(0\le\sigma\le\sigma_0\), the required weighted moments and 
		\(p\)-derivatives are uniformly bounded on 
		\(1\le|p|\le1+\chi_0\). 
		Indeed, compact support of the initial datum and the heat-kernel 
		comparison give uniform Gaussian spatial tails on bounded time 
		intervals, which dominate the exponential weights and the polynomial 
		factors appearing in these quantities. 
	\end{proof}

	The moment estimates above control the gradient of the limiting amplitude. Its positive lower bound requires a separate Feynman--Kac argument.
	
	\begin{lemma}[Near-critical amplitude lower bound and gradient estimate]
		\label{cont:lem:near-amplitude}
		For \(0<\chi(p)\le\chi_0\), one has
		\begin{equation}\label{cont:eq:near-amplitude-bounds}
			\mathfrak a(p)\ge c\chi(p),\qquad
			|\nabla_p\mathfrak a(p)|\le C,
		\end{equation}
		with constants uniform throughout the near-critical annulus.
	\end{lemma}
	
	\begin{proof}

		We first prove the gradient bound. By
		Lemma~\ref{cont:lem:near-moments}, differentiation of
		\eqref{cont:eq:limiting-amplitude} under the time integral
		is justified locally in $p$.
		Since \(u_0\) is compactly supported,
		\(|\nabla_p\mathfrak m_0(p)|\le C\)
		uniformly on the near-critical annulus.
		To estimate the integral, split at the fixed time
		\(\sigma_0\) from Lemma~\ref{cont:lem:near-moments}.
		The finite-time bounds in that lemma control
		the contribution from \(0\le a\le\sigma_0\).
		For the remaining part, \eqref{cont:eq:near-first-moment} gives
		\[
		\int_{\sigma_0}^\infty
		|\nabla_p\mathfrak n_a(p)|\,da
		\le C\int_{\sigma_0}^\infty
		(1+\log a+\chi a)a^{-3/2}e^{-\chi^2a}\,da
		\le C.
		\]
		Indeed, the first two terms are integrable uniformly,
		while
		$\chi\int_0^\infty a^{-1/2}e^{-\chi^2a}\,da=\sqrt\pi$.
		Thus $|\nabla_p\mathfrak a(p)|\le C$.

		We next prove the lower bound. The Feynman--Kac
		representation for the weighted mass $\mathfrak m_\sigma(p)$, followed by
		bounded convergence as $\sigma\to\infty$, gives
		\begin{equation}\label{cont:eq:near-probabilistic-lower}
			\mathfrak a(p)
			=\int u_0(y)e^{p\cdot y}\,
			\mathbb E\exp\!\left[
			-\int_0^\infty
			u(a,y+2pa+\sqrt2\,\mathsf B_a)\,da
			\right]dy,
		\end{equation}
		where $\mathsf B_a$ is standard Brownian motion
		with covariance $a\id$.

		Write $p=(1+\chi)\omega$ and
		$\mathsf X_a=y+2pa+\sqrt2\,\mathsf B_a$.
		Fix $L_{\rm prob}>0$.
		For \(y\in\supp u_0\), the Gaussian density and compactness of
		\(\supp u_0\times\Sph\times[1,1+\chi_0]\) give a constant
		\(c_0>0\), independent of \(y,\omega,|p|\), such that
		\[
		\mathbb P\!\left(
		\omega\cdot \mathsf{X}_1-(2+L_{\rm prob})\in[1,2]
		\right)\ge c_0.
		\]
		Condition on this event and write
		\(y_{\rm ex}=\omega\cdot \mathsf{X}_1-(2+L_{\rm prob})\in[1,2]\).
		By the Markov property, the future excess over the line
		\(2a+L_{\rm prob}\) is, for \(s\ge0\),
		\[
		Y_s
		:=\omega\cdot \mathsf{X}_{1+s}
		-\bigl(2(1+s)+L_{\rm prob}\bigr)
		=
		y_{\rm ex}+2\chi s+\sqrt2\,\beta_s,
		\]
		where \(\beta_s\) is a standard one-dimensional Brownian motion.
		Its probability of never hitting zero is
		\(1-e^{-2\chi y_{\rm ex}}\ge c_1\chi\) for some constant \(c_1>0\), uniformly for
		\(y_{\rm ex}\in[1,2]\).

		Consequently, with probability at least $c_0c_1\chi$,
		the path satisfies
		$\omega\cdot\mathsf X_a\ge 2a+L_{\rm prob}$
		for every $a\ge1$.
		We call such paths good paths.
		On these paths, the
		physical front coordinate satisfies
		\(r=|\mathsf X_a|-R_0(1+a)
		\ge L_{\rm prob}-2+k\log(1+a)\).
		Hence Lemma~\ref{cont:lem:upper-leading-edge},
		applied at original physical time $1+a$, gives,
		for all sufficiently large $a$,
		\[
		u(a,\mathsf X_a)\le
		C\bigl(1+L_{\rm prob}+k\log(1+a)\bigr)
		e^{-L_{\rm prob}}(1+a)^{-k}.
		\]
		Since $k>1$, this bound is integrable in $a$.
		On the remaining fixed initial time interval,
		we use $u\le K_0$.
		Thus, on the good-path event,
		\[
		\int_0^\infty u(a,\mathsf X_a)\,da\le C,
		\]
		with $C$ independent of $y\in\supp u_0$,
		$\omega$, and $\chi$.
		Consequently,
		\[
		\mathbb E\exp\!\left[
		-\int_0^\infty u(a,\mathsf X_a)\,da
		\right]\ge c\chi,
		\]
		uniformly for $y\in\supp u_0$.

		Inserting this bound into
		\eqref{cont:eq:near-probabilistic-lower} yields
		\[
		\mathfrak a(p)
		\ge
		c\chi
		\int u_0(y)e^{p\cdot y}\,dy
		=
		c\chi\,\mathfrak m_0(p).
		\]
		Since \(\mathfrak m_0(p)\) has a positive uniform lower bound on
		the compact annulus \(1\le|p|\le1+\chi_0\), we conclude that
		\(\mathfrak a(p)\ge c\chi\). This completes the proof.
	\end{proof}
	
	\subsection{Finite-time point-amplitude estimate}

	The weighted mass \(\mathfrak m_\sigma(p)\) defined in
	\eqref{cont:eq:tilted-mass} tracks nonlinear absorption,
	whereas the Gaussian exterior requires pointwise control
	at the moving point \(x=2\sigma p\).
	We therefore define the pointwise amplitude
	\begin{equation}\label{cont:eq:point-amplitude}
		\mathfrak a_\sigma(p)
		:=\frac{u(\sigma,2\sigma p)}
		{e^\sigma\mathcal K_\sigma(2\sigma p)}
		\qquad(\sigma>0).
	\end{equation}

	\begin{proposition}[Near-critical point-amplitude estimate]
		\label{cont:prop:near-finite-amplitude}
		There are fixed \(\Sigma_{\rm near},C>0\), independent of
		\(0<\chi(p)\le\chi_0\), such that
		\begin{equation}\label{cont:eq:near-finite-amplitude}
			|\mathfrak a_\sigma(p)-\mathfrak a(p)|
			\le\frac{C}{\sigma\chi(p)}
			\qquad(\sigma\ge\Sigma_{\rm near}).
		\end{equation}
		The corresponding relative error is at most
		\(C/(\sigma\chi(p)^2)\).
	\end{proposition}
	
	\begin{proof}
		Applying Duhamel's formula at $x=2\sigma p$
		and dividing by
		$e^\sigma\mathcal K_\sigma(2\sigma p)$ gives
		\begin{equation}\label{cont:eq:point-amplitude-duhamel}
			\begin{split}
				\mathfrak a_\sigma(p)
				&=\mathfrak m_0^{[\sigma]}(p)
				-\int_0^\sigma\mathfrak j_\sigma(a,p)\,da,\\
				\mathfrak m_0^{[\sigma]}(p)
				&:=\int u_0(y)e^{p\cdot y-|y|^2/(4\sigma)}\,dy.
			\end{split}
		\end{equation}
		Here completion of the square gives
		\begin{equation}\label{cont:eq:point-duhamel-source}
			\mathfrak j_\sigma(a,p)
			:=e^{-(1+|p|^2)a}
			\left(\frac{\sigma}{\sigma-a}\right)^{N/2}
			\int e^{p\cdot y-|y-2ap|^2/[4(\sigma-a)]}
			u(a,y)^2\,dy.
		\end{equation}
		
		\paragraph{Early source times $0\le a\le\sigma/2$.}
		Using
		$(\sigma/(\sigma-a))^{N/2}-1\le Ca/\sigma$
		and
		$1-e^{-|y-2ap|^2/[4(\sigma-a)]}
		\le C|y-2ap|^2/\sigma$, we obtain
		\[
		|\mathfrak j_\sigma(a,p)-\mathfrak n_a(p)|
		\le\frac C\sigma
		\bigl[a\mathfrak n_a(p)+\mathfrak n_{a,2}(p)\bigr].
		\]
		Lemma~\ref{cont:lem:near-moments}, including its
		finite-time bounds, therefore gives
		\begin{equation}\label{cont:eq:near-early-source}
			\int_0^{\sigma/2}
			|\mathfrak j_\sigma(a,p)-\mathfrak n_a(p)|\,da
			\le\frac C{\sigma\chi}.
		\end{equation}
		Indeed, the only potentially nonuniform integrals
		are
		$\int_0^\infty a^{-1/2}e^{-\chi^2a}\,da$
		and
		$\chi^2\int_0^\infty a^{1/2}e^{-\chi^2a}\,da$,
		both $O(\chi^{-1})$.
		Finally, 
		compact support of \(u_0\) gives
		\[
		|\mathfrak m_0^{[\sigma]}(p)-\mathfrak m_0(p)|
		\le C/\sigma
		\]
		uniformly on the near-critical annulus.

		\paragraph{Intermediate source times
			$\sigma/2\le a\le\sigma-2\chi\sigma$.}
		Write $p=|p|\omega$,
		$y=\rho\omega'$, and $r=\rho-R_0(1+a)$.
		Applying \eqref{cont:eq:spherical-moments} with $j=0$
		and \eqref{cont:eq:squared-envelope} to
		\eqref{cont:eq:point-duhamel-source}, we obtain
		\[
		\begin{aligned}
			\mathfrak j_\sigma(a,p)
			&\le C\left(\frac{\sigma}{\sigma-a}\right)^{N/2}
			\int_0^\infty
			\rho^{N-1}(1+b_{\rm sph})^{-d_N}
			e^{E(a,\rho)}f_{\rm env}(r)\,d\rho\\
			&\le C\left(\frac{\sigma}{\sigma-a}\right)^{1/2}
			\int_0^\infty
			\rho^{d_N}e^{E(a,\rho)}f_{\rm env}(r)\,d\rho.
		\end{aligned}
		\]
		Here $b_{\rm sph}=\frac{|p|\sigma\rho}{\sigma-a}$ and
		$E(a,\rho)=-(1+|p|^2)a+b_{\rm sph}
		-\frac{\rho^2+4a^2|p|^2}{4(\sigma-a)}$.
		The second inequality uses
		$b_{\rm sph}\ge\frac{\sigma\rho}{\sigma-a}$
		and $d_N=(N-1)/2$.
		
		Now set $z_1=\rho-2a=r+2-k\log(1+a)$ and
		$\gamma_{\rm late}=1+\frac{\chi\sigma}{\sigma-a}
		\in[1,3/2]$.
		Rearranging the exponent gives
		$E(a,\rho)=-\frac{\chi^2a\sigma}{\sigma-a}
		+\gamma_{\rm late}z_1
		-\frac{z_1^2}{4(\sigma-a)}$.
		Hence
		\[
		\begin{aligned}
			\mathfrak j_\sigma(a,p)
			\le{}
			C\left(\frac{\sigma}{\sigma-a}\right)^{1/2}
			e^{-\frac{\chi^2a\sigma}{\sigma-a}}
			\int_0^\infty
			\rho^{d_N}
			e^{\gamma_{\rm late}z_1-\frac{z_1^2}{4(\sigma-a)}}
			f_{\rm env}(r)\,d\rho.
		\end{aligned}
		\]
		Since $\gamma_{\rm late}\ge1$, we have
		$e^{\gamma_{\rm late}z_1}
		\le C(1+a)^{-k}e^{\gamma_{\rm late}r}$.
		Moreover, 
		since $1\le\gamma_{\rm late}\le3/2$,
		\eqref{cont:eq:squared-envelope} gives
		$e^{\gamma_{\rm late}r}f_{\rm env}(r)\le Ce^r$
		for $r<0$, and
		$e^{\gamma_{\rm late}r}f_{\rm env}(r)
		\le C(1+r)^2e^{-r/2}$ for $r\ge0$.
		Consequently,
		\[
		\int_{-R_0(1+a)}^\infty
		(R_0(1+a)+r)^{d_N}
		e^{\gamma_{\rm late}r}f_{\rm env}(r)\,dr
		\le C(1+a)^{d_N}.
		\]
		Dropping the negative quadratic exponent and
		using $k=d_N+3/2$, we obtain
		\[
		\begin{aligned}
			\mathfrak j_\sigma(a,p)
			\le C\left(\frac{\sigma}{\sigma-a}\right)^{1/2}
			a^{-3/2}
			e^{-\frac{\chi^2a\sigma}{\sigma-a}}.
		\end{aligned}
		\]
		Since $a\sigma/(\sigma-a)\ge\sigma$ for
		$a\ge\sigma/2$,
		we conclude that
		\begin{equation}\label{cont:eq:near-late-middle}
			\int_{\sigma/2}^{\sigma-2\chi\sigma}
			\mathfrak j_\sigma(a,p)\,da
			\le C\sigma^{-1/2}e^{-\chi^2\sigma}.
		\end{equation}

		\paragraph{Terminal source times
			$0<\sigma-a<2\chi\sigma$.}
		By \eqref{cont:eq:squared-envelope}, we have,
		globally in $y$,
		\[
		u(a,y)^2
		\le C e^{-\frac32(|y|-R_0(1+a))}
		\le C e^{\frac32 R_0(1+a)}
		e^{-\frac32\omega\cdot y}.
		\]
		Substituting into \eqref{cont:eq:point-duhamel-source}
		and evaluating the Gaussian integral gives
		\[
		\mathfrak j_\sigma(a,p)
		\le C\sigma^{N/2}(1+a)^{-3k/2}
		\exp\!\left(
		(|p|-3/2)^2\sigma-a/4
		\right).
		\]
		Since $|p|=1+\chi$ and
		$0<\sigma-a<2\chi\sigma$,
		we obtain
		\((|p|-3/2)^2\sigma-a/4 \le -3\chi\sigma/8\),
		where we used $\chi\le\chi_0=1/8$.
		Since $a\ge(1-2\chi)\sigma\ge3\sigma/4$,
		using $k=(N+2)/2$ and the interval length
		$2\chi\sigma$, we obtain
		\begin{equation}\label{cont:eq:near-terminal-source}
			\int_{\sigma-2\chi\sigma}^{\sigma}
			\mathfrak j_\sigma(a,p)\,da
			\le C\sigma^{-N/4-1/2}e^{-3\chi\sigma/8}
			\le C\sigma^{-1/2}e^{-\chi^2\sigma}.
		\end{equation}

		Finally, \eqref{cont:eq:near-zeroth-moment} implies
		\[
		\int_{\sigma/2}^\infty\mathfrak n_a(p)\,da
		\le C\sigma^{-1/2}e^{-\chi^2\sigma/2}.
		\]
		Combining \eqref{cont:eq:near-early-source},
		\eqref{cont:eq:near-late-middle},
		\eqref{cont:eq:near-terminal-source},
		and the initial amplitude error, the elementary bound
		$\sigma^{-1/2}e^{-\chi^2\sigma/2}
		\le C/(\sigma\chi)$
		proves \eqref{cont:eq:near-finite-amplitude}.
		Division by the
		lower bound in \eqref{cont:eq:near-amplitude-bounds} gives the
		relative estimate.
	\end{proof}
	
	\section{Uniform amplitude estimates away from the critical speed}
	\label{cont:sec:fixed-gap}
	
	In this section, we retain the notation of Section~\ref{cont:sec:near-critical}:
	the weighted masses \(\mathfrak m_\sigma\), the centered absorption moments
	\(\mathfrak n_{\sigma,j}\), the limiting amplitude \(\mathfrak a\), and
	the point amplitude \(\mathfrak a_\sigma\).
	Fix \(\chi_{\min}>0\). Throughout this section,
	all estimates are uniform for 
	\(|p|\ge1+\chi_{\min}\).

	The near-critical analysis of
	Section~\ref{cont:sec:near-critical} allows
	\(\chi=|p|-1\downarrow0\), so that the radial
	separation from the critical-speed radius,
	\(|x|-2\sigma=2\chi\sigma\), may be comparable
	to the diffusive scale \(\sqrt{\sigma}\).
	In the present regime, the fixed positive gap
	\(\chi\ge\chi_{\min}\) instead gives
	\(2\chi\sigma\ge2\chi_{\min}\sigma\gg\sqrt{\sigma}\) as \(\sigma\to\infty\).
	At the same time, \(|p|\) is unbounded, so the
	weighted-moment argument of
	Section~\ref{cont:sec:near-critical},
	which exploits \(2-|p|\ge1-\chi_0>0\),
	does not apply uniformly in this regime.
	We therefore use the moving frame centered at \(2\sigma p\) to establish analogous amplitude estimates to those in Section~\ref{cont:sec:near-critical}, uniformly for \(|p|\ge 1+\chi_{\min}\).

	\subsection{Fixed-gap Gaussian separation and relative centered moments}
	
	Motivated by the identity
	\[
	e^{p\cdot x-|p|^2\sigma}\mathcal K_\sigma(x)
	=\mathcal K_\sigma(x-2\sigma p),
	\]
	we introduce the moving coordinate
	\(x=2\sigma p+z\). In this frame, the renormalized exponential weight
	becomes
	\(e^{-(1+|p|^2)\sigma+p\cdot x}
	=e^{(|p|^2-1)\sigma+p\cdot z}\).
	Accordingly, for each fixed \(p\), define
	\begin{equation}\label{cont:eq:unnormalized-galilean}
		\begin{split}
			\mathcal U_p(\sigma,z)
			&=e^{(|p|^2-1)\sigma+p\cdot z}
			u(\sigma,2p\sigma+z),\\
			\mathcal V_p(\sigma,z)
			&=u(\sigma,2p\sigma+z),\qquad
			f_p(z)=e^{p\cdot z}u_0(z).
		\end{split}
	\end{equation}
	Using \(u_\sigma=\Delta_xu+u-u^2\), a direct computation gives
	\begin{equation}\label{cont:eq:unnormalized-pde}
		(\partial_\sigma-\Delta_z)\mathcal U_p
		=-\mathcal V_p\mathcal U_p,\qquad
		\mathcal U_p(0,z)=f_p(z).
	\end{equation}
	Thus \(\mathcal U_p\) satisfies the heat equation
	with the nonnegative absorption coefficient \(\mathcal V_p\).
	
	The free heat evolution is
	\begin{equation}\label{cont:eq:free-heat}
		P_p(\sigma,z):=(\mathcal K_\sigma*f_p)(z).
	\end{equation}
	Since \(0\le\mathcal V_p\le K_0\), comparison gives
	\begin{equation}\label{cont:eq:free-comparison}
		0<e^{-K_0\sigma}P_p\le\mathcal U_p\le P_p
		\quad(\sigma>0),\qquad
		\int P_p(\sigma,z)\,dz=\mathfrak m_0(p).
	\end{equation}
	The change of variables \(x=2p\sigma+z\) gives
	\begin{equation}\label{cont:eq:mass-galilean}
		\int\mathcal U_p(\sigma,z)\,dz=\mathfrak m_\sigma(p),\qquad
		\int\mathcal V_p(\sigma,z)\mathcal U_p(\sigma,z)\,dz
		=\mathfrak n_\sigma(p).
	\end{equation}

	\begin{lemma}[Uniform Gaussian separation in the fixed-gap regime]
		\label{cont:lem:gaussian-separation}
		There exist \(C,c>0\) such that, for
		\(|p|\ge1+\chi_{\min}\), \(|y|\le R_*\),
		\(0<s\le\sigma\), and \(j=0,1,2\),
		\begin{equation}\label{cont:eq:gaussian-separation}
			\int\mathcal K_s(z)
			u(\sigma,2\sigma p+y+z)|y+z|^j\,dz
			\le C(1+\sigma^{j/2})e^{-c\sigma}.
		\end{equation}
		The estimate is uniform for \(0<s\le\sigma\), in particular as
		\(s\downarrow0\).
	\end{lemma}
	
	\begin{proof}
		Take $c=\chi_{\min}^2/32$. It suffices to consider
		$\sigma\ge\max\{1,2R_*/\chi_{\min}\}$, since
		for bounded \(\sigma\), the result follows from \(u\le K_0\) after
		enlarging \(C\).
		Set $x=2\sigma p+y+z$ and split the integral at
		$|x|=(2+\chi_{\min})\sigma$.
		In the inner region,
		\(|z|
		\ge|2\sigma p|-|y|-|x|
		\ge {\chi_{\min}\sigma}/{2}\).
		Since $s\le\sigma$, we obtain
		\[
		\mathcal K_s(z)
		\mathbf1_{\{|z|\ge\chi_{\min}\sigma/2\}}
		\le 2^{N/2}e^{-c\sigma}\mathcal K_{2s}(z).
		\]
		In the outer region,
		$|x|-R_*\ge(2+\chi_{\min}/2)\sigma$.
		Thus \eqref{cont:eq:heat-comparison} gives
		\[
		\begin{aligned}
			u(\sigma,x)
			&\le (4\pi\sigma)^{-N/2}\|u_0\|_1
			\exp\!\left(\sigma-
			\frac{(|x|-R_*)^2}{4\sigma}\right)\\
			&\le C e^{-(\chi_{\min}/2+
				\chi_{\min}^2/16)\sigma}
			\le Ce^{-c\sigma}.
		\end{aligned}
		\]
		Finally, for $\theta=1,2$, the Gaussian moment bound gives
		\[
		\int\mathcal K_{\theta s}(z)|y+z|^j\,dz
		\le C(1+s^{j/2})
		\le C(1+\sigma^{j/2}).
		\]
		Combining this with the estimates in the two regions
		proves \eqref{cont:eq:gaussian-separation}.
	\end{proof}

	\begin{lemma}[Relative centered absorption moments]
		\label{cont:lem:fixed-gap-moments}
		For \(j=0,1,2\), \(\sigma>0\), and
		\(|p|\ge1+\chi_{\min}\),
		\begin{equation}\label{cont:eq:fixed-gap-moments}
			\mathfrak n_{\sigma,j}(p)
			\le
			C\mathfrak m_0(p)(1+\sigma^{j/2})e^{-c\sigma}.
		\end{equation}
		Moreover, there exists \(C_{\rm abs}>0\) such that
		\begin{equation}\label{cont:eq:free-absorption}
			\begin{split}
				\int\mathcal V_p(\sigma,z)P_p(\sigma,z)\,dz
				&\le
				Ce^{-c\sigma}\mathfrak m_0(p),\\
				\int_0^\infty\!\!\int
				\mathcal V_p(a,z)P_p(a,z)\,dz\,da
				&\le
				C_{\rm abs}\mathfrak m_0(p).
			\end{split}
		\end{equation}
	\end{lemma}
	
	\begin{proof}
		Applying \eqref{cont:eq:heat-comparison} to one
		factor of $u^2$ and using
		\[
		e^{-|p|^2\sigma+p\cdot(x-y)}
		\mathcal K_\sigma(x-y)
		=\mathcal K_\sigma(x-y-2\sigma p),
		\]
		we obtain, after changing variables,
		\[
		\begin{aligned}
			\mathfrak n_{\sigma,j}(p)
			&\le\int e^{p\cdot y}u_0(y)
			\int\mathcal K_\sigma(z)
			u(\sigma,2\sigma p+y+z)|y+z|^j\,dz\,dy\\
			&\le C\mathfrak m_0(p)
			(1+\sigma^{j/2})e^{-c\sigma},
		\end{aligned}
		\]
		by Lemma~\ref{cont:lem:gaussian-separation} with \(s=\sigma\).
		
		Similarly, the definition of $P_p$ gives
		\[
		\begin{aligned}
			\int\mathcal V_pP_p\,dz
			&=\int e^{p\cdot y}u_0(y)
			\int\mathcal K_\sigma(z)
			u(\sigma,2\sigma p+y+z)\,dz\,dy\\
			&\le Ce^{-c\sigma}\mathfrak m_0(p),
		\end{aligned}
		\]
		using the same lemma with $j=0$.
		Integration over $\sigma>0$ proves
		\eqref{cont:eq:free-absorption}.
	\end{proof}

	\subsection{Lower bound and relative regularity for the limiting amplitude}
	
	First, we establish a uniform fixed-gap lower bound for \(\mathfrak a(p)\) relative to \(\mathfrak m_0(p)\). Whereas Lemma~\ref{cont:lem:near-amplitude} obtains \(\mathfrak a(p)\ge c\chi\,\mathfrak m_0(p)\) in the near-critical regime using a Feynman--Kac argument, here we use \eqref{cont:eq:free-absorption} to construct a positive subsolution of the transformed equation.

	\begin{lemma}[Fixed-gap lower bound for the limiting amplitude]
		\label{cont:lem:fixed-gap-mass}
		For every fixed \(\chi_{\min}>0\), there is
		\(c_{\rm mass}>0\) such that
		\begin{equation}\label{cont:eq:fixed-gap-positive-mass}
			c_{\rm mass}\mathfrak m_0(p)
			\le\mathfrak a(p)\le\mathfrak m_0(p)
			\qquad(|p|\ge1+\chi_{\min}).
		\end{equation}
		The same lower bound holds for \(\mathfrak m_\sigma(p)\) at
		every \(\sigma\ge0\).
	\end{lemma}
	
	\begin{proof}

		Define
		\begin{equation}\label{cont:eq:accumulated-absorption}
			E_p(\sigma,z)
			:=\int_0^\sigma
			\mathcal K_{\sigma-a}*
			\bigl(\mathcal V_p(a,\cdot)P_p(a,\cdot)\bigr)(z)\,da.
		\end{equation}
		Then $(\partial_\sigma-\Delta_z)E_p=\mathcal V_pP_p$
		and $E_p(0,z)=0$. Since $0\le\mathcal V_p\le K_0$
		and $\mathcal K_{\sigma-a}*P_p(a)=P_p(\sigma)$,
		we have $0\le E_p\le K_0\sigma P_p$.
		For $\sigma>0$, $P_p(\sigma,z)>0$. Set
		\begin{equation}\label{cont:eq:pde-subsolution}
			\Theta_p=\frac{E_p}{P_p},\qquad
			\underline P_p=P_pe^{-\Theta_p}.
		\end{equation}
		Since $P_p$ solves the free heat equation,
		the
		quotient \(\Theta_p\) satisfies
		\[
		(\partial_\sigma-\Delta_z)\Theta_p
		-2\nabla_z\log P_p\cdot\nabla_z\Theta_p
		=\mathcal V_p.
		\]
		Consequently,
		\begin{equation}\label{cont:eq:subsolution-residual}
			(\partial_\sigma-\Delta_z+\mathcal V_p)\underline P_p
			=-\underline P_p|\nabla_z\Theta_p|^2\le0.
		\end{equation}
		
		Since $f_p=P_p(0)$ may vanish, we start comparison
		at $\sigma=\delta>0$.
		The bound $0\le\Theta_p\le K_0\sigma$ gives
		$e^{-K_0\sigma}P_p\le\underline P_p\le P_p$.
		For each fixed $p$, both $\underline P_p(\delta)$
		and $\mathcal U_p(\delta)$ converge uniformly to
		$f_p$ as $\delta\downarrow0$. Thus
		$\varepsilon_\delta:=
		\|\underline P_p(\delta)-\mathcal U_p(\delta)\|_\infty
		\to0$.
		Since $\mathcal V_p\ge0$,
		$\mathcal U_p+\varepsilon_\delta$ is a supersolution.
		Whole-space comparison on finite time intervals,
		followed by $\delta\downarrow0$, yields
		\begin{equation}\label{cont:eq:subsolution-comparison}
			\mathcal U_p(\sigma,z)\ge\underline P_p(\sigma,z).
		\end{equation}

		We now pass from the subsolution to a positive mass fraction.
		Heat propagation preserves mass, so
		\eqref{cont:eq:free-absorption} gives
		\[
		\int E_p(\sigma,z)\,dz
		=\int_0^\sigma\!\!\int
		\mathcal V_p(a,z)P_p(a,z)\,dz\,da
		\le C_{\rm abs}\mathfrak m_0(p).
		\]
		Since $\int P_p(\sigma,z)\,dz=\mathfrak m_0(p)>0$,
		Jensen's inequality with probability density
		$P_p/\mathfrak m_0(p)$ gives
		\[
		\begin{aligned}
			\int\underline P_p\,dz
			=\int P_pe^{-E_p/P_p}\,dz
			\ge\mathfrak m_0(p)
			\exp\left(-\frac{\int E_p\,dz}{\mathfrak m_0(p)}\right)
			\ge e^{-C_{\rm abs}}\mathfrak m_0(p).
		\end{aligned}
		\]
		Combining this with \eqref{cont:eq:mass-galilean} and
		\eqref{cont:eq:subsolution-comparison} yields
		\(\mathfrak m_\sigma(p)\ge
		\mathfrak m_0(p)e^{-C_{\rm abs}}\).
		Take \(c_{\rm mass}=e^{-C_{\rm abs}}>0\). Then we obtain the lower bound by letting
		\(\sigma\to\infty\). The upper bound follows from
		\eqref{cont:eq:mass-derivative}.
	\end{proof}

	While Section~\ref{cont:sec:near-critical} gives an absolute gradient bound for \(\mathfrak a\) on the bounded near-critical annulus, here we establish a uniform relative gradient bound, namely a bound on \(|\nabla_p\log\mathfrak a|\), for \(|p|\ge1+\chi_{\min}\), using the first centered-moment estimate and the amplitude lower bound.

	\begin{lemma}[Relative regularity of the limiting amplitude]
		\label{cont:lem:neighboring-velocities}
		The limiting amplitude belongs to \(C^1(\{|p|>1\})\).
		For every fixed \(\chi_{\min}>0\),
		\begin{equation}\label{cont:eq:relative-amplitude-gradient}
			|\nabla_p\log\mathfrak a(p)|\le C
			\qquad(|p|\ge1+\chi_{\min}).
		\end{equation}
		If \(\chi(p),\chi(p')\ge\chi_{\min}\) and
		\(|p'-p|\le\chi_{\min}/2\), then
		\begin{equation}\label{cont:eq:neighboring-amplitude}
			\left|
			\frac{\mathfrak a(p')}{\mathfrak a(p)}-1
			\right|
			\le C|p'-p|.
		\end{equation}
		The vectors need not have the same direction.
	\end{lemma}
	
	\begin{proof}
		Differentiating the weighted absorption gives
		\[
		\nabla_p\mathfrak n_\sigma(p)
		=e^{-(1+|p|^2)\sigma}
		\int(x-2\sigma p)e^{p\cdot x}u(\sigma,x)^2\,dx.
		\]
		The two \(p\)-derivative contributions combine into the centered
		factor \(x-2\sigma p\), avoiding any separate \(\sigma|p|\) loss.
		By \eqref{cont:eq:fixed-gap-moments},
		\[
		|\nabla_p\mathfrak n_\sigma(p)|
		\le\mathfrak n_{\sigma,1}(p)
		\le C\mathfrak m_0(p)(1+\sqrt\sigma)e^{-c\sigma}.
		\]
		
		To justify differentiation under the time integral, fix
		\(p_0\) with \(|p_0|>1\) and choose a neighborhood \(U\) of \(p_0\)
		such that \(\inf_{q\in U}(|q|-1)>0\). On \(U\),
		\(\mathfrak m_0\) is bounded by compact support, while
		\eqref{cont:eq:fixed-gap-moments} gives an integrable local bound for
		\(\nabla_p\mathfrak n_a\). Hence differentiation under the integral is
		valid, and the same local dominated-convergence argument gives
		\(\mathfrak a\in C^1(\{|p|>1\})\).
		
		For \(|p|\ge1+\chi_{\min}\), we then have
		\[
		|\nabla_p\mathfrak a(p)|
		\le |\nabla_p\mathfrak m_0(p)|
		+\int_0^\infty|\nabla_p\mathfrak n_a(p)|\,da
		\le R_*\mathfrak m_0(p)
		+\int_0^\infty\mathfrak n_{a,1}(p)\,da
		\le C\mathfrak m_0(p).
		\]
		Dividing by \eqref{cont:eq:fixed-gap-positive-mass} proves
		\eqref{cont:eq:relative-amplitude-gradient}.

		Since \(|p'-p|\le\chi_{\min}/2\), every point \(q\) of the segment
		\([p,p']\) satisfies
		\[
		|q|-1\ge\chi_{\min}-|p'-p|\ge\chi_{\min}/2.
		\]
		Applying \eqref{cont:eq:relative-amplitude-gradient} with the smaller
		fixed gap \(\chi_{\min}/2\) along the segment gives
		\[
		|\log\mathfrak a(p')-\log\mathfrak a(p)|
		\le C|p'-p|.
		\]
		Exponentiating and using \(|p'-p|\le\chi_{\min}/2\) yields
		\eqref{cont:eq:neighboring-amplitude}.
	\end{proof}
	
	\subsection{Point-amplitude convergence in the fixed-gap regime}

	In this subsection, we use the point-amplitude Duhamel representation \eqref{cont:eq:point-amplitude-duhamel} and its source term \eqref{cont:eq:point-duhamel-source}. We first establish a bound on the source that remains uniform up to the terminal time.

	\begin{lemma}[Uniform bound for the Duhamel source]
		\label{cont:lem:fixed-gap-source}
		For \(0<a<\sigma\) and \(|p|\ge1+\chi_{\min}\),
		\begin{equation}\label{cont:eq:fixed-gap-source}
			0\le\mathfrak j_\sigma(a,p)
			\le Ce^{-ca}\mathfrak m_0(p).
		\end{equation}
		The bound is uniform as \(a\downarrow0\) and \(a\uparrow\sigma\).
	\end{lemma}

	\begin{proof}
		The source term \eqref{cont:eq:point-duhamel-source} admits the
		kernel representation
		\begin{equation}\label{cont:eq:source-kernel-form}
			\mathfrak j_\sigma(a,p)
			=\frac{e^{-a}}{\mathcal K_\sigma(2\sigma p)}
			\int\mathcal K_{\sigma-a}(2\sigma p-x)
			u(a,x)^2\,dx .
		\end{equation}
		Applying \eqref{cont:eq:heat-comparison} to one factor of
		\(u(a,x)^2\) in \eqref{cont:eq:source-kernel-form}, we obtain
		\[
		\mathfrak j_\sigma(a,p)
		\le
		\frac{1}{\mathcal K_\sigma(2\sigma p)}
		\iint
		\mathcal K_a(x-y)
		\mathcal K_{\sigma-a}(2\sigma p-x)
		u_0(y)u(a,x)\,dx\,dy.
		\]
		The two heat kernels can be combined by the exact identity
		\begin{equation}\label{cont:eq:two-gaussian-product}
			\frac{\mathcal K_a(x-y)\mathcal K_{\sigma-a}(2\sigma p-x)}
			{\mathcal K_\sigma(2\sigma p-y)}
			=
			\mathcal K_s
			\bigl(x-[2ap+(1-a/\sigma)y]\bigr),
			\qquad
			s:=\frac{a(\sigma-a)}{\sigma}.
		\end{equation}
		Moreover,
		since
		$\mathcal K_\sigma(2\sigma p-y)/
		\mathcal K_\sigma(2\sigma p)
		=e^{p\cdot y-|y|^2/(4\sigma)}$, we obtain
		\[
		\begin{split}
			\mathfrak j_\sigma(a,p)
			&\le
			\int u_0(y)e^{p\cdot y-|y|^2/(4\sigma)}\\
			&\qquad\cdot
			\int\mathcal K_s
			\bigl(x-[2ap+(1-a/\sigma)y]\bigr)
			u(a,x)\,dx\,dy .
		\end{split}
		\]
		Here \(0<s\le a\) and \(|(1-a/\sigma)y|\le R_*\).
		Lemma~\ref{cont:lem:gaussian-separation}, applied at time \(a\),
		therefore yields
		\[
		\int\mathcal K_s
		\bigl(x-[2ap+(1-a/\sigma)y]\bigr)
		u(a,x)\,dx
		\le Ce^{-ca}.
		\]
		Since \(e^{-|y|^2/(4\sigma)}\le1\), integration in \(y\) gives
		\eqref{cont:eq:fixed-gap-source}.
		
		Finally, since \(s\downarrow0\) at both endpoints and
		Lemma~\ref{cont:lem:gaussian-separation} is uniform for
		\(0<s\le a\), the exact Gaussian identity
		\eqref{cont:eq:two-gaussian-product} yields the bound uniformly as
		\(a\downarrow0\) and \(a\uparrow\sigma\).
	\end{proof}

	Unlike the near-critical argument in
	Proposition~\ref{cont:prop:near-finite-amplitude},
	the uniform source estimate in
	Lemma~\ref{cont:lem:fixed-gap-source}
	allows us to split the Duhamel time integral into only
	early and late source times, without a separate
	terminal-time estimate.

	\begin{proposition}[Relative point-amplitude convergence in the fixed-gap regime]
		\label{cont:prop:fixed-gap-amplitude}
		For every fixed \(\chi_{\min}>0\),
		\begin{equation}\label{cont:eq:fixed-gap-amplitude}
			\sup_{|p|\ge1+\chi_{\min}}
			\left|\frac{\mathfrak a_\sigma(p)}{\mathfrak a(p)}-1\right|
			\le\frac{C}{\sigma}\qquad(\sigma\ge1).
		\end{equation}
	\end{proposition}

	\begin{proof}
		\emph{Early source times.}
		For \(0\le a\le\sigma/2\), the exact formula
		\eqref{cont:eq:point-duhamel-source} gives
		\[
		|\mathfrak j_\sigma(a,p)-\mathfrak n_a(p)|
		\le\frac{C_N}{\sigma}
		\bigl[a\mathfrak n_a(p)+\mathfrak n_{a,2}(p)\bigr],
		\]
		by using 
		$(\sigma/(\sigma-a))^{N/2}-1\le C_Na/\sigma$
		and
		$1-e^{-|x-2ap|^2/[4(\sigma-a)]}
		\le C|x-2ap|^2/\sigma$.
		Lemma~\ref{cont:lem:fixed-gap-moments} and the integrability of
		\((1+a)e^{-ca}\) therefore give
		\[
		\int_0^{\sigma/2}
		|\mathfrak j_\sigma(a,p)-\mathfrak n_a(p)|\,da
		\le\frac{C}{\sigma}\mathfrak m_0(p).
		\]

		\emph{Late source times.}
		On \(\sigma/2\le a<\sigma\), 
		we obtain \(\mathfrak j_\sigma(a,p)+\mathfrak n_a(p)
		\le Ce^{-ca}\mathfrak m_0(p)\) by
		using
		\eqref{cont:eq:fixed-gap-source} and
		\eqref{cont:eq:fixed-gap-moments} directly.
		Thus
		\[
		\begin{split}
			\int_{\sigma/2}^\sigma
			(\mathfrak j_\sigma(a,p)+\mathfrak n_a(p))\,da
			+\int_\sigma^\infty\mathfrak n_a(p)\,da
			\le
			Ce^{-c\sigma/2}\mathfrak m_0(p).
		\end{split}
		\]
		Compact support gives
		\(0\le\mathfrak m_0(p)-\mathfrak m_0^{[\sigma]}(p)
		\le\frac{R_*^2}{4\sigma}\mathfrak m_0(p)\).
		Combining with \eqref{cont:eq:point-amplitude-duhamel} and
		\eqref{cont:eq:limiting-amplitude},
		\[
		|\mathfrak a_\sigma(p)-\mathfrak a(p)|
		\le C\mathfrak m_0(p)
		\left(\sigma^{-1}+e^{-c\sigma/2}\right).
		\]
		For $\sigma\ge1$, absorb the exponential into
		$C/\sigma$ and use
		\eqref{cont:eq:fixed-gap-positive-mass}
		to obtain the relative estimate.
	\end{proof}
	
	\section{Gaussian exterior and whole-space log-concavity}
	\label{cont:sec:exterior}
	
	In this section, we establish the logarithmic-Hessian bound in the Gaussian exterior and combine it with the preceding estimates to complete the proof of whole-space eventual log-concavity.

	\begin{proposition}[Uniform exterior logarithmic Hessian bound]
		\label{cont:prop:exterior}
		There exist \(K_{\rm ext},\Sigma_{\rm ext}>0\) such that
		\begin{equation}\label{cont:eq:exterior-hessian}
			D_x^2\log u(\sigma,2\sigma p)
			\le -\frac{1}{4\sigma}\id
		\end{equation}
		for every \(\sigma\ge\Sigma_{\rm ext}\) and every
		\(p\in\mathbb R^N\) satisfying
		\(\chi(p)\sqrt{\sigma}\ge K_{\rm ext}.\)
		Equivalently, writing \(x=2\sigma p\), the estimate holds throughout
		\(|x|\ge 2\sigma+2K_{\rm ext}\sqrt{\sigma}\).
	\end{proposition}

	We use a single normalized profile throughout the Gaussian exterior.
	Its rescaled equation and exact Gaussian--amplitude identity reduce
	the proof of Proposition \ref{cont:prop:exterior} to three analytical ingredients: a uniform \(C^0\) Gaussian
	approximation, uniform smallness of the rescaled absorption coefficient,
	and local \(C^2\) recovery.  
	Unless stated otherwise, the constants in the following estimates may change from line to line but are independent of \(\sigma\), \(p\), and \(K_{\rm ext}\).
	
	\subsection{A normalized exterior profile}
	
	Recall
	\[
	\mathcal U_p(\sigma,z)
	=e^{(|p|^2-1)\sigma+p\cdot z}
	u(\sigma,2p\sigma+z)
	\]
	from \eqref{cont:eq:unnormalized-galilean}. If \(u\) satisfies the linearized Fisher--KPP equation \(u_\sigma=\Delta_xu+u\), then \(\mathcal U_p\) satisfies the ordinary heat equation and its total mass \(\mathfrak m_\sigma(p)=\int_{\mathbb R^N}\mathcal U_p(\sigma,z)\,dz=\mathfrak m_0(p)\) is conserved. Moreover, by the heat-kernel representation and compact support of \(u_0\), for each fixed \(p\) and \(\theta>0\),
	\[
	\sigma^{N/2}\mathcal U_p(\sigma\theta,\sqrt{\sigma}z)
	\longrightarrow \mathfrak m_0(p)\mathcal K_\theta(z)
	\qquad (\sigma\to\infty),
	\]
	locally uniformly in \(z\).
	For the nonlinear equation \(u_\sigma=\Delta_xu+u-u^2\), the absorption in \eqref{cont:eq:unnormalized-pde} instead decreases the weighted mass \(\mathfrak m_\sigma(p)\), and Sections~\ref{cont:sec:near-critical} and~\ref{cont:sec:fixed-gap} establish its positive limit \(\mathfrak a(p)\). This motivates the normalization introduced below.

	For a fixed terminal pair \((\sigma,p)\), with \(\sigma>0\) and
	\(|p|>1\), define the normalized profile
	\begin{equation}\label{cont:eq:galilean-profile}
		\begin{split}
			\mathcal W(\theta,z)
			&=
			\frac{\sigma^{N/2}}{\mathfrak a(p)}
			\mathcal U_p(\sigma\theta,\sqrt{\sigma}\,z).
		\end{split}
	\end{equation}
	We suppress the terminal pair \((\sigma,p)\) in the notation for
	\(\mathcal W\) and \(\mathcal V\) below. The factor
	\(\sigma^{N/2}\) removes the diffusive height decay, while division by
	\(\mathfrak a(p)\) normalizes the amplitude. We therefore aim to show that \(\mathcal W\) is close to the unit-scale heat kernel \(\mathcal K_\theta(z)\) on fixed cylinders in \((\theta,z)\). The preceding positivity lemmas ensure
	that this normalization is well defined for every \(|p|>1\).

	We work on the fixed nested cylinders
	\[
	\mathcal Q^+=(1/3,5/3)\times B_3,\qquad
	\mathcal Q^0=(1/2,3/2)\times B_2,\qquad
	\mathcal Q^-=(3/4,5/4)\times B_1.
	\]
	Set
	\begin{equation}\label{cont:eq:rescaled-potential}
		\mathcal V(\theta,z)
		=
		\sigma u(\sigma\theta,2p\sigma\theta+\sqrt{\sigma}\,z)
		=
		\sigma\mathcal V_p(\sigma\theta,\sqrt{\sigma}\,z).
	\end{equation}
	Applying \eqref{cont:eq:unnormalized-pde} at elapsed time
	\(\sigma\theta\) and spatial point \(\sqrt{\sigma}\,z\) gives the
	exact equation
	\begin{equation}\label{cont:eq:galilean-pde}
		(\partial_\theta-\Delta_z)\mathcal W
		=
		-\mathcal V\mathcal W.
	\end{equation}
	
	We next use the exact relation between \(\mathcal U_p\) and the point amplitude.
	Indeed, since
	\(2sp+y=2s\bigl(p+y/(2s)\bigr)\), the definition
	\eqref{cont:eq:point-amplitude} gives
	\begin{equation}\label{cont:eq:galilean-point-amplitude}
		\mathcal U_p(s,y)
		=
		\mathcal K_s(y)\,
		\mathfrak a_s\!\left(p+\frac{y}{2s}\right),
		\qquad s>0.
	\end{equation}
	Thus diffusive-scale displacements correspond to small
	changes in \(p\).
	For \((\theta,z)\in\mathcal Q^+\), introduce the neighboring parameter
	\begin{equation}\label{cont:eq:neighboring-velocity}
		p'
		=
		p+\frac{z}{2\theta\sqrt{\sigma}},
		\qquad
		2\sigma\theta p'
		=
		2p\sigma\theta+\sqrt{\sigma}\,z.
	\end{equation}
	Applying \eqref{cont:eq:galilean-point-amplitude} with
	\(s=\sigma\theta\) and \(y=\sqrt{\sigma}\,z\), and using the heat-kernel
	scaling identity
	\(\sigma^{N/2}\mathcal K_{\sigma\theta}(\sqrt{\sigma}\,z)
	=\mathcal K_\theta(z)\), yields the exact factorization
	\begin{equation}\label{cont:eq:profile-amplitude-identity}
		\mathcal W(\theta,z)
		=
		\mathcal K_\theta(z)
		\frac{\mathfrak a_{\sigma\theta}(p')}{\mathfrak a(p)}.
	\end{equation}

	\subsection{Uniform Gaussian approximation and Hessian recovery}
	
	The profile-amplitude identity \eqref{cont:eq:profile-amplitude-identity} reduces the Gaussian approximation to
	relative amplitude control. We use the near-critical estimates of
	Section~\ref{cont:sec:near-critical} and the fixed-gap estimates of
	Section~\ref{cont:sec:fixed-gap} in their respective regimes.
	
	\begin{lemma}[Uniform \(C^0\) Gaussian approximation]
		\label{cont:lem:uniform-gaussian-approximation}
		There exist fixed constants \(C>0\) and \(\Sigma_0\ge1\) such that,
		for every \(K_{\rm ext}\ge9\), every \(\sigma\ge\Sigma_0\), and
		every \(p\in\mathbb R^N\) satisfying
		\(\chi(p)\sqrt{\sigma}\ge K_{\rm ext}\),
		\begin{equation}\label{cont:eq:uniform-gaussian-approximation}
			\|\mathcal W-\mathcal K_\theta\|_
			{C^0(\overline{\mathcal Q^+})}
			\le
			C\left(
			K_{\rm ext}^{-1}+K_{\rm ext}^{-2}+\sigma^{-1/2}
			\right).
		\end{equation}
		The constants \(C,\Sigma_0\) are independent of \(K_{\rm ext}\)
		and of the magnitude and direction of \(p\).
	\end{lemma}
	
	\begin{proof}
		Choose \(\Sigma_0\ge3\Sigma_{\rm near}\) sufficiently large that
		\(9/(2\sqrt{\Sigma_0})\le\chi_0/8\).
		
		\emph{Case 1: near-critical regime.}
		Write \(\chi=\chi(p)=|p|-1\), and suppose
		\(K_{\rm ext}/\sqrt{\sigma}\le\chi\le\chi_0/2\).
		For \((\theta,z)\in\overline{\mathcal Q^+}\),
		\eqref{cont:eq:neighboring-velocity} gives
		\(|p'-p|\le9/(2\sqrt{\sigma})\).
		Since \(K_{\rm ext}\ge9\) and \(\sigma\ge\Sigma_0\),
		we have \(|p'-p|\le\min\{\chi/2,\chi_0/8\}\).
		Consequently, every \(q\in[p,p']\) satisfies
		\(\chi/2\le\chi(q)\le\chi_0\).
		Hence the whole segment remains in the annulus of
		Lemma~\ref{cont:lem:near-amplitude}, which gives
		\(\mathfrak a(p)\ge c\chi\) and a uniform gradient
		bound \(|\nabla\mathfrak a|\le C\) along the segment.
		Since \(\sigma\theta\ge\sigma/3\ge\Sigma_{\rm near}\),
		Proposition~\ref{cont:prop:near-finite-amplitude} gives
		\[
		\begin{aligned}
			\left|
			\frac{\mathfrak a_{\sigma\theta}(p')}{\mathfrak a(p)}-1
			\right|
			&\le
			\frac{
				|\mathfrak a_{\sigma\theta}(p')-\mathfrak a(p')|
			}{\mathfrak a(p)}
			+
			\frac{
				|\mathfrak a(p')-\mathfrak a(p)|
			}{\mathfrak a(p)}
			\\
			&\le
			\frac{C}{\sigma\chi^2}
			+
			\frac{C}{\chi\sqrt{\sigma}}.
		\end{aligned}
		\]
		Since \(\chi\sqrt{\sigma}\ge K_{\rm ext}\),
		\begin{equation}\label{cont:eq:near-cylinder-amplitude}
			\sup_{\overline{\mathcal Q^+}}
			\left|
			\frac{\mathfrak a_{\sigma\theta}(p')}{\mathfrak a(p)}-1
			\right|
			\le C\left(K_{\rm ext}^{-2}+K_{\rm ext}^{-1}\right).
		\end{equation}

		\emph{Case 2: fixed-gap regime.}
		Suppose \(\chi(p)\ge\chi_0/2\).
		Since \(|p'-p|\le\chi_0/8\),
		every \(q\in[p,p']\) satisfies
		\(\chi(q)\ge\chi_0/4\).
		Thus the whole segment remains in the fixed-gap regime,
		and Lemma~\ref{cont:lem:neighboring-velocities}
		applies uniformly along it.
		Proposition~\ref{cont:prop:fixed-gap-amplitude} and
		Lemma~\ref{cont:lem:neighboring-velocities} give,
		uniformly on \(\overline{\mathcal Q^+}\),
		\[
		\frac{\mathfrak a_{\sigma\theta}(p')}{\mathfrak a(p')}
		=1+O(\sigma^{-1}),\qquad
		\frac{\mathfrak a(p')}{\mathfrak a(p)}
		=1+O(\sigma^{-1/2}).
		\]
		Multiplying these two ratios yields
		\begin{equation}\label{cont:eq:fixed-gap-cylinder-amplitude}
			\sup_{\substack{\chi(p)\ge\chi_0/2\\
					(\theta,z)\in\overline{\mathcal Q^+}}}
			\left|
			\frac{\mathfrak a_{\sigma\theta}(p')}{\mathfrak a(p)}-1
			\right|
			\le C\left(\sigma^{-1}+\sigma^{-1/2}\right).
		\end{equation}
		The constants are independent of the magnitude
		and direction of \(p\).
		
		The two cases cover all \(p\) satisfying
		\(\chi(p)\sqrt{\sigma}\ge K_{\rm ext}\)
		and meet at \(\chi(p)=\chi_0/2\).
		Combining \eqref{cont:eq:near-cylinder-amplitude}
		and \eqref{cont:eq:fixed-gap-cylinder-amplitude}
		with \eqref{cont:eq:profile-amplitude-identity}
		yields \eqref{cont:eq:uniform-gaussian-approximation}.
	\end{proof}

	We next estimate the rescaled absorption coefficient \(\mathcal V\).
	Fix \(0<\alpha<1\) for the following two lemmas.
	Their constants may depend on \(\alpha\).

	\begin{lemma}[Uniform smallness of the rescaled absorption coefficient]
		\label{cont:lem:rescaled-potential-smallness}
		There exist fixed constants \(C>0\) and
		\(\Sigma_1\ge1\) such that, for every \(K_{\rm ext}\ge30\),
		every \(\sigma\ge\Sigma_1\), and every \(p\in\mathbb R^N\)
		satisfying \(\chi(p)\sqrt{\sigma}\ge K_{\rm ext}\),
		\begin{equation}\label{cont:eq:potential-holder}
			\|\mathcal V\|_
			{C^{\alpha/2,\alpha}(\overline{\mathcal Q^+})}
			\le
			C\sigma^2e^{-K_{\rm ext}\sqrt{\sigma}/6}.
		\end{equation}
		The constants and the threshold are independent of
		\(K_{\rm ext}\) and of the magnitude and direction of \(p\).
	\end{lemma}
	
	\begin{proof}
		Write
		\(\chi=\chi(p)\), and suppose
		\(\chi\sqrt{\sigma}\ge K_{\rm ext}\ge30\).
		Given \((\theta,z)\in\overline{\mathcal Q^+}\), set
		\(X=2p\sigma\theta+\sqrt{\sigma}\,z\), and consider the fixed local parabolic cylinder
		\(s\in[\sigma\theta-4,\sigma\theta+1]\), \(Y\in B_2(X)\).
		Its corresponding original physical time is \(1+s\).
		
		We first show that this cylinder remains at least \(c\chi\sigma\) ahead of the front, for some fixed \(c>0\).
		After increasing \(\Sigma_1\) if necessary, every point of this
		cylinder satisfies
		\(\sigma/4\le s\), \(1+s\le2\sigma\), and
		\(1+s\ge T_{\rm le}\), where \(T_{\rm le}\) is the starting time
		in Lemma~\ref{cont:lem:upper-leading-edge}. Moreover, since
		\(|z|\le3\), \(|Y-X|\le2\), and \(s\le\sigma\theta+1\), the front
		distance satisfies
		\[
		\begin{split}
			r'&:=|Y|-R_0(1+s)
			=|Y|-2s-2+k\log(1+s)\\
			&\ge 2\chi\sigma\theta-3\sqrt{\sigma}-6\\
			&\ge \frac23\chi\sigma-3\sqrt{\sigma}-6
			\ge \frac13\chi\sigma.
		\end{split}
		\]
		Here the last inequality follows directly from
		\(\chi\sqrt{\sigma}\ge K_{\rm ext}\ge30\) and \(\sigma\ge1\).
		
		Equation~\eqref{cont:eq:upper-solution}, applied at physical time
		\(1+s\), therefore gives throughout the cylinder
		\begin{equation}\label{cont:eq:physical-cylinder-smallness}
			u(s,Y)
			\le
			C\exp\!\left(
			-\frac{\chi\sigma}{6}
			-\frac{\chi^2\sigma}{144}
			\right).
		\end{equation}
		Indeed, the factor \(1+r'\) is absorbed into half of the linear
		exponential, while \(r'\ge\chi\sigma/3\) and \(1+s\le2\sigma\)
		give the displayed linear and quadratic exponents.
		Standard interior parabolic estimates for
		\(u_\sigma-\Delta u=u-u^2\)
		then transfer this pointwise smallness to the
		derivatives on a smaller fixed cylinder:
		\begin{equation}\label{cont:eq:physical-derivative-smallness}
			u(\sigma\theta,X)
			+|\nabla_xu(\sigma\theta,X)|
			+|D_x^2u(\sigma\theta,X)|
			+|u_\sigma(\sigma\theta,X)|
			\le
			C\exp\!\left(
			-\frac{\chi\sigma}{6}
			-\frac{\chi^2\sigma}{144}
			\right).
		\end{equation}
		Here and below the constants in these local estimates are uniform in
		\(p\).
		
		The chain rule gives
		\(\nabla_z\mathcal V
		=\sigma^{3/2}\nabla_xu(\sigma\theta,X)\) and
		\(\partial_\theta\mathcal V
		=\sigma^2(u_\sigma+2p\cdot\nabla_xu)(\sigma\theta,X)\).
		Since \(|p|=1+\chi\),
		\[
		\begin{aligned}
			\|\mathcal V\|_{C^1_{\theta,z}(\overline{\mathcal Q^+})}
			&\le C\sigma^2(2+\chi)
			e^{-\chi\sigma/6-\chi^2\sigma/144}\\
			&\le C\sigma^2e^{-\chi\sigma/6},
		\end{aligned}
		\]
		where the last inequality uses
		\((2+\chi)e^{-\chi^2\sigma/144}\le C\)
		uniformly for \(\sigma\ge1\) and \(\chi\ge0\).
		On the fixed cylinder, the
		\(C^{\alpha/2,\alpha}\) norm is controlled by
		the \(C^1_{\theta,z}\) norm.
		Using \(\chi\sqrt{\sigma}\ge K_{\rm ext}\)
		now gives \eqref{cont:eq:potential-holder}.
	\end{proof}

	We now upgrade the uniform \(C^0\) Gaussian approximation to spatial
	\(C^2\) control on a smaller cylinder.
	
	\begin{lemma}[Local \(C^2\) Gaussian approximation]
		\label{cont:lem:local-c2-gaussian-recovery}
		Let \(K_{\rm ext}\ge30\),
		\(\sigma\ge\max\{\Sigma_0,\Sigma_1\}\), and
		\(\chi(p)\sqrt{\sigma}\ge K_{\rm ext}\).
		Then
		\begin{equation}\label{cont:eq:profile-c2}
			\|\mathcal W-\mathcal K_\theta\|_
			{C_z^2(\overline{\mathcal Q^-})}
			\le
			C\left(
			\|\mathcal W-\mathcal K_\theta\|_
			{C^0(\overline{\mathcal Q^+})}
			+
			\|\mathcal V\|_
			{C^{\alpha/2,\alpha}(\overline{\mathcal Q^+})}
			\right),
		\end{equation}
		where \(C\) is independent of \(K_{\rm ext}\), \(\sigma\), and \(p\).
		Here \(C_z^2\) denotes uniform control of the function and its first
		spatial derivatives, with the Hessian measured in operator norm.
	\end{lemma}
	
	\begin{proof}
		The preceding two lemmas give uniform \(L^\infty\) bounds for
		\(\mathcal W\) and \(\mathcal V\) on \(\mathcal Q^+\).
		Interior parabolic \(W^{2,1}_q\) estimates and parabolic embedding,
		with \(q\) chosen sufficiently large for the fixed exponent
		\(\alpha\), give
		\(\|\mathcal W\|_
		{C^{\alpha/2,\alpha}(\overline{\mathcal Q^0})}
		\le C\).
		Hence
		\[
		\|\mathcal V\mathcal W\|_
		{C^{\alpha/2,\alpha}(\overline{\mathcal Q^0})}
		\le
		C\|\mathcal V\|_
		{C^{\alpha/2,\alpha}(\overline{\mathcal Q^+})}.
		\]

		Set \(H=\mathcal W-\mathcal K_\theta\). Since
		\((\partial_\theta-\Delta_z)\mathcal K_\theta=0\),
		\eqref{cont:eq:galilean-pde} gives
		\((\partial_\theta-\Delta_z)H=-\mathcal V\mathcal W\).
		The interior Schauder estimate on
		\(\mathcal Q^-\Subset\mathcal Q^0\) yields
		\[
		\|H\|_{C_z^2(\overline{\mathcal Q^-})}
		\le
		C\left(
		\|H\|_{C^0(\overline{\mathcal Q^0})}
		+
		\|\mathcal V\mathcal W\|_
		{C^{\alpha/2,\alpha}(\overline{\mathcal Q^0})}
		\right),
		\]
		and \eqref{cont:eq:profile-c2} follows.
	\end{proof}
	
	We now combine the preceding estimates to establish the exterior logarithmic-Hessian bound.
	
	\begin{proof}[Proof of Proposition~\ref{cont:prop:exterior}]
		At \((\theta,z)=(1,0)\),
		\[
		\mathcal K_1(0)=(4\pi)^{-N/2},\qquad
		\nabla_z\mathcal K_1(0)=0,\qquad
		D_z^2\mathcal K_1(0)
		=-\frac{\mathcal K_1(0)}{2}\id.
		\]
		Combining Lemmas~\ref{cont:lem:uniform-gaussian-approximation},
		\ref{cont:lem:rescaled-potential-smallness}, and
		\ref{cont:lem:local-c2-gaussian-recovery}, we first choose
		\(K_{\rm ext}\ge30\) sufficiently large and then
		\(\Sigma_{\rm ext}\) sufficiently large so that, for every
		\(\sigma\ge\Sigma_{\rm ext}\),
		\[
		\|\mathcal W-\mathcal K_\theta\|_
		{C_z^2(\overline{\mathcal Q^-})}
		\le \frac{\mathcal K_1(0)}8.
		\] Hence
		\[
		\frac{7\mathcal K_1(0)}8
		\le\mathcal W(1,0)
		\le\frac{9\mathcal K_1(0)}8,
		\qquad
		D_z^2\mathcal W(1,0)
		\le-\frac{3\mathcal K_1(0)}8\id.
		\]
		Since
		\[
		D_z^2\log\mathcal W
		=
		\frac{D_z^2\mathcal W}{\mathcal W}
		-
		\frac{\nabla_z\mathcal W\otimes\nabla_z\mathcal W}
		{\mathcal W^2},
		\]
		the second term is negative semidefinite, and therefore
		\(D_z^2\log\mathcal W(1,0)\le-\id/3\le-\id/4\).
		
		Finally, the logarithm of the conjugating factor in
		\eqref{cont:eq:unnormalized-galilean} and \eqref{cont:eq:galilean-profile}
		is affine in \(z\), so
		\begin{equation}\label{cont:eq:hessian-scaling}
			D_z^2\log\mathcal W(\theta,z)
			=
			\sigma D_x^2\log u
			(\sigma\theta,2p\sigma\theta+\sqrt{\sigma}\,z).
		\end{equation}
		Evaluating at \((\theta,z)=(1,0)\) gives
		\eqref{cont:eq:exterior-hessian}.
	\end{proof}
	
	\subsection{Whole-space completion and level-set geometry}
	\label{cont:sec:completion}
	
	We now combine the estimates in the four spatial regions to prove Theorem~\ref{cont:thm:whole-space}, and then derive Corollary~\ref{cont:cor:all-levels}.
	Here we return to the original physical time \(t\), and \(u(t,x)\)
	again denotes the original solution. The preceding analysis in
	Sections~\ref{cont:sec:near-critical}--\ref{cont:sec:exterior} uses the elapsed time
	\(\sigma=t-1\).
	
	\begin{proof}[Proof of Theorem~\ref{cont:thm:whole-space}]
		Fix \(K_{\rm ext}\) from Proposition~\ref{cont:prop:exterior},
		and apply Proposition~\ref{cont:prop:through-diffusive} with
		\(K=4K_{\rm ext}\). Let \(L_K\) be its lower cutoff, and choose
		\(L_{\rm fr}\ge\max\{L_0,L_K\}\).
		
		Proposition~\ref{prop:innerconcavity} covers
		\(\rho\le R_0(t)-L_0\), while Proposition~\ref{prop:frontlog}
		covers \(|r|\le L_{\rm fr}\) after its fixed-parameter threshold.
		These regions overlap because \(L_{\rm fr}\ge L_0\). Moreover,
		the lower boundary \(r=L_K\) of the diffusive leading edge lies
		in the fixed-front region because \(L_{\rm fr}\ge L_K\).
		Hence the bounds in the invaded interior, the fixed-front region, and
		the diffusive leading edge cover
		\(\rho\le R_0(t)+4K_{\rm ext}\sqrt t\) once all their fixed
		starting times have been reached.
		
		In original physical time, the Gaussian exterior begins at
		\(2(t-1)+2K_{\rm ext}\sqrt{t-1}\). The difference between the
		outer radius of the diffusive leading edge and the inner radius of the Gaussian exterior is
		\[
		\begin{split}
			&R_0(t)+4K_{\rm ext}\sqrt t
			-\left[2(t-1)+2K_{\rm ext}\sqrt{t-1}\right]\\
			&=2-k\log t+4K_{\rm ext}\sqrt t
			-2K_{\rm ext}\sqrt{t-1}\\
			&\ge2-k\log t+2K_{\rm ext}\sqrt t>0
		\end{split}
		\]
		for all sufficiently large \(t\). Thus the four regions cover
		the whole space after one fixed time.
		
		The bound in the invaded interior is
		\(-c_{\rm in}t^{-1}e^{-\beta_*t}\id\), the fixed-front bound is
		\(-c_{L_{\rm fr}}t^{-1}\id\), and the bounds in the diffusive leading edge and the Gaussian
		exterior are respectively \(-c_Kt^{-1}\id\) and
		\(-[4(t-1)]^{-1}\id\). Since
		\(1/(t-1)\ge1/t\ge e^{-\beta_*t}/t\), all four estimates imply the
		common weaker bound \eqref{cont:eq:whole-space-hessian} with
		\(c=\min\{c_{\rm in},c_{L_{\rm fr}},c_K,1/4\}\), after increasing
		the physical starting time past the finitely many fixed thresholds.
		
		Finally, integration of \(D_x^2\log u<0\) along straight segments
		in \(\R^N\) gives strict log-concavity.
	\end{proof}
	
	\begin{proof}[Proof of Corollary~\ref{cont:cor:all-levels}]
		Set \(q=\log u\). Spatial decay implies that \(u(t,\cdot)\) attains
		its maximum. If \(x\ne y\), then
		\[
		(\nabla q(x)-\nabla q(y))\cdot(x-y)
		=
		\int_0^1
		D_x^2q\bigl(y+s(x-y)\bigr)[x-y,x-y]\,ds
		<0.
		\]
		Hence \(q\), and therefore \(u\), has at most one critical point.
		Since a maximum is attained, this critical point exists and is the
		unique global maximizer. At that point, \(\nabla q=0\), and hence
		\(D_x^2u=uD_x^2q<0\), so the maximizer is nondegenerate.
		
		If \(0<\lambda<M(t)\), no point of the level set
		\(\{u(t,\cdot)=\lambda\}\) can be critical, since the only critical
		point is the global maximizer. Thus \(\lambda\) is a regular value,
		and the level set is a smooth hypersurface. Spatial decay makes the
		superlevel set \(\{u(t,\cdot)\ge\lambda\}\) compact, while
		log-concavity makes it convex.
		
		With respect to the outward unit normal
		\(\nu_{\rm out}=-\nabla_xu/|\nabla_xu|\), every nonzero tangent
		vector \(e\) satisfies
		\[
		\mathrm{II}[e,e]
		=
		-\frac{D_x^2u[e,e]}{|\nabla_xu|}
		=
		-\frac{uD_x^2\log u[e,e]}{|\nabla_xu|}
		>0.
		\]
		Hence every principal curvature is positive, and the level
		set is strictly convex.
	\end{proof}

	\appendix
	\section*{Appendix}
	\refstepcounter{section}\label{app:sec3}

	\subsection{Moving-frame and self-similar calculations}\label{app:transform}
	
	For completeness, we record the intermediate equations leading to
	\eqref{eq:v} and \eqref{eq:W}.
	
	\paragraph{Moving frame.}
	Set \(U(t,r,\omega)=u(t,R_0(t)+r,\omega)\). Then
	\[
	U_t
	=
	U_{rr}
	+
	\left(
	2-\frac{k}{t}
	+\frac{N-1}{r+R_0(t)}
	\right)U_r
	+
	\frac{\DS U}{(r+R_0(t))^2}
	+
	U-U^2.
	\]
	To remove the critical exponential factor in the leading edge,
	write \(U(t,r,\omega)=e^{-r}v(t,r,\omega)\).
	Differentiating gives \eqref{eq:v}.

	\paragraph{Self-similar variables and rear-boundary shift.}
	We use the shifted self-similar variables defined in
	\eqref{eq:logshift} and set
	\(\widetilde w(\tau,\xi,\omega)=t^{-1/2}v(t,r,\omega)\).
	Since \(s'(\tau)+s(\tau)/2=-3e^{-\tau/2}\),
	the translated function satisfies
	\begin{equation}\label{eq:shifted_selfsimilar}
		\begin{aligned}
			\widetilde w_\tau
			={}&
			\widetilde w_{\xi\xi}
			+\frac{\xi}{2}\widetilde w_\xi
			+(h-3)e^{-\tau/2}\widetilde w_\xi
			-\left(h+\frac12\right)\widetilde w
			\\
			&\quad
			+\frac{e^\tau}{D^2}\DS\widetilde w
			-\exp\left(
			\frac{3\tau}{2}
			-(\xi+s)e^{\tau/2}
			\right)\widetilde w^2.
		\end{aligned}
	\end{equation}
	Here \(D\) and \(h\) are as in \eqref{eq:D}--\eqref{eq:h}.

	\paragraph{Gaussian conjugation.}
	Finally, set
	\(\widetilde w(\tau,\xi,\omega)
	=
	e^{-\xi^2/8}W(\tau,\xi,\omega)\).
	Note that
	\[
	e^{\xi^2/8}
	\left(\partial_{\xi\xi}+\frac{\xi}{2}\partial_\xi\right)
	\left(e^{-\xi^2/8}W\right)
	=
	W_{\xi\xi}
	-\left(\frac{\xi^2}{16}+\frac14\right)W.
	\]
	Substituting this identity into \eqref{eq:shifted_selfsimilar} yields \eqref{eq:W}, with the coefficients given in \eqref{eq:M}--\eqref{eq:alphaB}.
	The cancellation \((N-1)/2-k=-3/2\) gives
	\begin{equation}\label{eq:hexact}
		h+\frac32
		=
		\frac{
			(N-1)\bigl((k+3)\tau-\xi e^{\tau/2}\bigr)
		}{2D}.
	\end{equation}
	Since \(D \ge e^\tau\) and \(h\) is uniformly bounded,
	\eqref{eq:hexact} and \eqref{eq:l} imply the coefficient bounds \eqref{eq:lbounds}.

	\subsection{Three-region verification of the scalar supersolution}\label{app:barriercheck}
	
	Here we verify the supersolution inequality for the function
	\(\Phi\) constructed in the proof of Lemma~\ref{lem:barrier}.
	
	We shall use
	\begin{equation}\label{eq:profiles}
		\MA\phi_0=0,\qquad -\varphi''=\Lambda\varphi,\quad
		\Lambda=\left(\frac{\pi}{2a_0}\right)^2,\qquad
		\MA G=\left(\frac{3\xi^2}{64}-\frac58\right)G.
	\end{equation}
	The present coefficient bounds \eqref{eq:lbounds} give
	\begin{equation}\label{eq:regionalcoeff}
		|l_1|\le Ce^{-\tau/2},\qquad
		|l_2|\le C(1+\xi)e^{-\tau/2},
	\end{equation}
	because \(\tau e^{-\tau}\le Ce^{-\tau/2}\).
	
	The three-region calculation in
	\cite[Section~3.2, equations~(31)--(33)]{RRM2019}
	now applies with \(m\) in place of both cutoff amplitudes.
	Near zero, \(\varphi\ge1/\sqrt2\) and
	\(|\varphi'|\le\sqrt\Lambda\,\varphi\), so
	\[
	\PP\Phi\ge
	\bigl[m'+(\Lambda-C\sqrt\Lambda-C)m\bigr]\varphi
	-C\zeta e^{-\tau/2}
	\qquad(0\le\xi\le a_0/2).
	\]
	Choose \(a_0\) small enough that
	\(\Lambda-C\sqrt\Lambda-C\ge1\).

	In the far field, where \(\gamma_1=0\) and \(\gamma_2=1\),
	we have \(\Phi=mG+\zeta\phi_0\). Since
	\(G'=-(\xi/8)G\) and
	\(\MA G=(3\xi^2/64-5/8)G\),
	\[
	\frac{\PP(mG)}{G}
	=m'+m\left(
	\frac{3\xi^2}{64}-\frac58
	+\frac{\xi l_1}{8}-l_2
	\right)
	=(m'+m)
	+m\left(
	\frac{3\xi^2}{64}-\frac{13}{8}
	+\frac{\xi l_1}{8}-l_2
	\right).
	\]
	By \eqref{eq:regionalcoeff}, after fixing \(R\ge2\) sufficiently
	large, the second bracket is nonnegative for all
	\(\tau\ge\tau_*\) and \(\xi\ge R\).
	Moreover, \(\MA\phi_0=0\) and \(\zeta'\ge0\), while
	\[
	\frac{|l_1\phi_0'|+|l_2\phi_0|}{G}
	\le Ce^{-\tau/2}(1+\xi+\xi^2)e^{-\xi^2/16}
	\le Ce^{-\tau/2}.
	\]
	Hence
	\[
	\PP\Phi\ge
	\bigl(m'+m-C\zeta e^{-\tau/2}\bigr)G
	\qquad(\xi\ge R).
	\]

	On the remaining fixed interval \(I=[a_0/2,R]\), the cutoff
	terms are \(O(m)\), exactly as in the transition-region calculation
	\citep[p.~7278]{RRM2019}, and \(\delta_I:=\min_I\phi_0>0\). Thus
	\[
	\PP\Phi\ge\delta_I\zeta'-C_I(m+\zeta e^{-\tau/2})
	\qquad(\xi\in I),
	\]
	where we have used \(m'\ge-m\) and \(\zeta'\ge0\), which follow
	from \eqref{eq:ODE}. Consequently, after choosing \(a_0\), the cutoffs
	and \(R\), all three regions are covered by
	\[
	m'+m\ge C\zeta e^{-\tau/2},\qquad
	\zeta'\ge C(m+\zeta e^{-\tau/2}).
	\]
	Taking \(C_0\) sufficiently large in \eqref{eq:ODE} therefore gives
	\(\PP\Phi\ge0\).
	
	\subsection{Finite-time Gaussian tails}\label{app:tails}
	
	\begin{proof}[Proof of Lemma~\ref{lem:tails}]
		Heat-kernel comparison gives
		\(0\le u(t,x)\le e^{t-1}(G_{t-1}*u_0)(x)\).
		Choose \(r_0>0\) such that \(t_*-r_0^2>1\), and set
		\(Q_r(t,x)=(t-r^2,t]\times B_r(x)\).
		Standard interior parabolic regularity gives
		\[
		\sup_{Q_{r_0/2}(t,x)}
		\bigl(|\nabla_xu|+|D_x^2u|\bigr)
		\le C\|u\|_{L^\infty(Q_{r_0}(t,x))}.
		\]
		For \((s,y)\in Q_{r_0}(t,x)\), we have
		\(|y|\ge |x|-r_0\) and \(s\le t\). Hence the heat-kernel bound
		and \(\operatorname{supp}u_0\subset B_{R_*}\) imply, for large
		\(|x|\),
		\[
		\|u\|_{L^\infty(Q_{r_0}(t,x))}
		\le
		C\exp\left(
		-\frac{(|x|-\overline R)^2}{4(t-1)}
		\right),
		\qquad \overline R=R_*+r_0.
		\]
		This proves \eqref{eq:physicaltail}.
		
		For \(\|A\|_{\mathrm{op}}\le1\), the identity
		\begin{equation*}
			X_A^2u=(Ax)^TD_x^2u(Ax)+(A^2x)\cdot\nabla_xu
		\end{equation*}
		bounds \(X_A^2u\) by a quadratic polynomial in \(\rho\) times
		the Gaussian in \eqref{eq:physicaltail}. The inverse of
		\eqref{eq:transform} then gives
		\[
		|X_A^2W|\le
		C t^{-1/2}(\rho^2+\rho)
		\exp\left(
		r+\frac{\xi^2}{8}
		-\frac{(\rho-\overline R)^2}{4(t-1)}
		\right).
		\]
		Write \(d(t)=2t-(k+3)\log t-\overline R\). Since
		\(\rho-\overline R=\sqrt t\,\xi+d(t)\) and
		\(r=\sqrt t\,\xi-3\log t\), the coefficient of \(-\xi^2\)
		in the exponent is
		\(t/[4(t-1)]-1/8>1/8\).
		All remaining linear terms in the exponent and the polynomial
		prefactor have uniformly bounded coefficients on \([t_*,T_1]\).
		They can therefore be absorbed while retaining the smaller Gaussian
		coefficient \(3/32\). By smoothness, enlarging the constant covers
		bounded \(\xi\), and hence
		\[
		\sup_{\substack{\tau_*\le\tau\le\log T_1\\
				\omega\in\Sph,\ \|A\|_{\mathrm{op}}\le1}}
		(X_A^2W)_+(\tau,\xi,\omega)
		\le C_{T_1}e^{-3\xi^2/32},
		\qquad \xi\ge0.
		\]
		This proves \eqref{eq:qtail}.
	\end{proof}
	
	\section*{Acknowledgments}
	
	The authors thank Professor Xinan Ma for his continued interest in
	this problem and for motivating us to pursue it.
	
	\section*{Declaration}
	
	The authors initiated and led the research program, formulated
	the mathematical problem, developed the results, and provided
	the proofs. ChatGPT suggested some potentially relevant references
	and provided summaries of selected works. After independently
	examining these references and assessing their relevance to our
	ongoing research, we found that some provided valuable background
	and insights that helped inspire aspects of the present work.
	The authors take full responsibility for the content of the manuscript.

\end{document}